\documentclass{article}

 \usepackage[nonatbib, preprint]{neurips_2019}

\usepackage[utf8]{inputenc} 
\usepackage[T1]{fontenc}    
\usepackage{hyperref}       
\usepackage{url}            
\usepackage{booktabs}       
\usepackage{amsfonts}       
\usepackage{nicefrac}       
\usepackage{microtype}      
\usepackage{subfigure}

\usepackage{amsmath,amsfonts,amssymb,amsthm,array}
\usepackage{mathtools}
\usepackage{algorithm}
\usepackage{caption}
\usepackage{algpseudocode,algorithm,algorithmicx}
\usepackage{bm}
\usepackage{color}
\usepackage{graphicx}
\def\<#1,#2>{\langle #1,#2\rangle}

\newcommand{\prox}{\operatorname{prox}}

\newcommand{\Prob}{\mathbb{P}}
\newcommand{\R}{\mathbb{R}}

\algrenewcommand\algorithmicrequire{\textbf{Input:}}
\algrenewcommand\algorithmicensure{\textbf{Initialize:}}

\newcommand{\compactify}{} 

\newcommand{\bG}{\mathbf{G}}

\newcommand{\eqdef}{\coloneqq}

\newcommand{\cL}{{\cal L}}

\newcommand{\cO}{{\cal O}}

\newcommand{\bE}{\mathbb{E}}

\theoremstyle{plain}
\newtheorem{theorem}{Theorem}[section]

\newtheorem{assumption}{Assumption}[section]
\newtheorem{lemma}[theorem]{Lemma}

\newtheorem{proposition}[theorem]{Proposition}
\newtheorem{definition}[theorem]{Definition}
\newtheorem{remark}{Remark}[]
\newtheorem{corollary}{Corollary}[]
\theoremstyle{remark}

\usepackage[colorinlistoftodos,bordercolor=orange,backgroundcolor=orange!20,linecolor=orange,textsize=scriptsize]{todonotes}

\title{L-SVRG and L-Katyusha with Arbitrary Sampling}

\author{%
  Xun Qian\\
  KAUST\thanks{King Abdullah University of Science and Technology, Thuwal, Saudi Arabia.}\\\
  \texttt{xun.qian@kaust.edu.sa} \\
   \And
   Zheng Qu \\
   The University of Hong Kong\\
   \texttt{zhengqu@hku.hk} \\
   \AND
   Peter Richt\'{a}rik \\
   KAUST${}^*$ and MIPT\thanks{Moscow Institute of Physics and Technology, Dolgoprudny, Russia.}\\
   \texttt{peter.richtarik@kaust.edu.sa} \\
}

\begin{document}

\maketitle

\begin{abstract}
We develop and analyze a new family of {\em nonaccelerated and accelerated  loopless variance-reduced methods} for finite sum optimization problems. Our convergence analysis relies on a novel expected smoothness condition which upper bounds the variance of the stochastic gradient estimation by a constant times a distance-like function. This  allows us to handle with ease {\em arbitrary sampling schemes} as well as the nonconvex case. We perform an in-depth estimation of these expected smoothness parameters  and propose new importance samplings which allow {\em linear speedup} when the expected minibatch size is in a certain range. Furthermore, a connection between these expected smoothness parameters and expected separable overapproximation (ESO) is established, which allows us to exploit  data sparsity as well. Our results recover as special cases the recently proposed  loopless SVRG and loopless Katyusha.
\end{abstract}


\tableofcontents

\clearpage

\section{Introduction}

In this work we consider the composite finite-sum optimization problem
\begin{equation}\label{primal}
\compactify \min_{x\in \mathbb{R}^d} P(x) \eqdef \frac{1}{n} \sum\limits_{i=1}^n  f_i(x) + \psi(x),
\end{equation}
where $f\eqdef\tfrac{1}{n}\sum_i  f_i$ is an average of a very large number of smooth functions  $f_i:\R^d\to \R$, and   $\psi:\R^d \to \R\cup \{+\infty\}$ is a proper closed convex function. We assume that problem (\ref{primal}) has at least one global optimal solution $x^*$.

{\bf Variance reduction.} Variance reduced methods for solving \eqref{primal} have recently become immensely popular and efficient alternatives of SGD \cite{Nemirovski-Juditsky-Lan-Shapiro-2009,RobbinsMonro:1951}. Among the first such methods proposed were SAG \cite{schmidt2017minimizing}, SAGA \cite{SAGA} and SVRG \cite{johnson2013accelerating, proxSVRG}, all with essentially identical theoretical complexity rates, but different practical use cases and different analysis techniques. While the first approaches to this were indirect and dual in nature \cite{SDCA}, it later transpired that variance reduced methods can be accelerated, in the sense of Nesterov, directly. The first such method, Katyusha~\cite{Katyusha}---an accelerated variant of SVRG---has become very popular due its optimal complexity rate, versatility and practical behavior.  Both SVRG and Katyusha have a two loop structure. In order for SVRG to obtain best convergence rate, the inner loop must be terminated after a number of iterations proportional to the condition number of the problem. However, this is often unknown, or hard to estimate, and this has led practitioners to devise various heuristic strategies instead, departing from theory.  

{\bf Loopless methods.} Recently, this problem was remedied by the so called {\em loopless} SVRG (L-SVRG) and {\em loopless} Katyusha (L-Katyusha)~\cite{LSVRG}. These methods  dispense off the outer loop, replacing it with a biased coin-flip to be performed in each step. This simple change makes the methods easier to understand, and easier to analyze. The worst case complexity bounds remain the same. Moreover, for L-SVRG the optimal probability of exit to the outer loop can be made independent of the condition number, which resolves the problem mentioned above, and makes the method more robust and markedly faster in practice. The analysis in \cite{LSVRG} was done in the strongly convex and smooth case $(\psi\equiv 0)$; rates in the convex and nonconvex case are not known.

{\bf Arbitrary sampling.} The {\em arbitrary sampling} paradigm to developing and analyzing stochastic algorithms allows for simultaneous study of of countless importance and minibatch sampling strategies, thus leading to a tight unification of two previously separate topics. It was first proposed in \cite{NSYNC} in the context of randomized coordinate descent methods. Since then, many stochastic methods were studied in this regime.  Methods already endowed with arbitrary sampling variants and analysis include, among others, the  primal-dual method Quartz~ \cite{Quartz}, accelerated randomized coordinate descent \cite{qu2016coordinate, qu2016coordinate2, FilipACD18},  stochastic primal-dual hybrid gradient method  \cite{SCP17}, SGD \cite{SGDAS}, and SAGA \cite{SAGA-AS}. All these methods were studied in a convex or strongly convex setting only. In the nonconvex case, an arbitrary sampling analysis was performed only recently in \cite{Samuel19}, for the SAGA, SVRG and SARAH methods, where an optimal sampling was developed. 

 \section{Contributions}
 
In this paper, we study L-SVRG and L-Katyusha with arbitrary sampling and  sampling with replacement strategies for problem (\ref{primal}). We now describe the sampling strategies employed, and give a summary of our complexity results.
 

\subsection{Sampling}

In the minibatch setting, a collection of the index is needed for each iteration. First we introduce the concepts of sampling and sampling with replacement. 

\begin{definition}[Sampling] 
A {\em sampling} $S$ is a random set-valued mapping with values being the subsets of $[n] \eqdef \{  1, 2, ..., n  \}$. It is uniquely characterized by the choice of probabilities $p_C \eqdef \Prob[S=C]$ associated with every subset $C$ of $[n]$.
Given a sampling $S$, we let $p_i \eqdef \Prob[i \in S] = \sum_{C:i\in C}p_C$. We say that $S$ is {\em proper} if $p_i>0$ for all $i$. We consider proper sampling only. 
\end{definition}

\begin{definition}[Group Sampling] 
A {\em  group sampling} $S$ is formed as follows. First, for each $i \in [n]$, distribute it a $p_i \in (0, 1]$. Then $[n]$ is divided into several groups $C_j$, $j = 1, ..., t\leq n$, where $C_{j_1} \cap C_{j_2} = \emptyset$ for $j_1 \neq j_2$ and $\cup_{j=1, ..., t}C_j = [n]$, such that $\sum_{i\in C_j} p_i \leq 1$ for all $j=1, ..., t$. Finally, each group $C_j$ have a chance to be chosen one index from it with probability $\sum_{i\in C_j}p_i$, and the only one index $i$ is chosen with probability $p_i/\sum_{s \in C_j}p_s$ where $i \in C_j$ within each group. We call an index $i$ is isolated if $i$ itself forms a group. 
\end{definition}


For group sampling $S$, it is easy to see that $\mathbb{P}[i \in S] = (p_i/\sum_{s \in C_j}p_s)\cdot\sum_{s \in C_j}p_s = p_i$. Group sampling contains independent sampling as a special case i.e., if every index $i$ is isolated for a group sampling, then it is independent sampling. Independent sampling is often studied in the arbitrary sampling paradigm \cite{FilipACD18, Samuel19}, however, a drawback is that the cost for each sample is ${\cal O}(n)$. While group sampling has the following nice property. 

\begin{lemma}\label{lm:groups}
	For any set $\{ p_i \}_{i=1}^n$ with $p_i\in (0, 1]$ and $1\leq \tau = \sum_{i\in [n]} p_i \leq n$, if $\tau$ is an integer, then there exists a group sampling $S$ such that $\mathbb{P}[i\in S] = p_i$ and the number of groups $t \leq 2\tau -1$. If $\tau$ is not an integer, then there exists a group sampling $S$ such that $\mathbb{P}[i\in S] = p_i$ and the number of groups $t < 2\tau +1$. 
\end{lemma}


 Apart from the sampling, we also consider the case where $S$ is consisted of {\em $\tau$ independent copies of $i$ from a distribution ${\tilde {\cal D}}$ with replacement}, which is also studied for Katyusha in \cite{Katyusha}. The distribution ${\tilde {\cal D}}$ is to output $i$ with probability ${\tilde p}_i$. In this way, $S$ is different with the sampling generally, and we call $S$ sampling with replacement. In fact, $S$ may contain multiple copies of a same index, and hence $S$ is not a set. When we take expectation with respect to the sampling with replacement $S$, it means that we take expectation with respect to $\tau$ independent copies of $i$. 
 
 \vskip 2mm
 
 For a sampling $S$ or sampling with replacement $S$, we define ${\bf I}_S$ as a diagonal matrix whose $i$-th diagonal entry is the number of copies of $i$ in $S$. Like in \cite{SAGA-AS}, we introduce a random diagnal matrix ${\theta}_{S} \in \R^{n\times n}$, where the $i$-th diagonal entry we denote by ${\theta}_S^i$. For sampling with replacement, let ${\theta_{S}^i \equiv 1/{\tau {\tilde p}_i}}$. Then 
\[\compactify 
 \mathbb{E}[{\theta}_S {\bf I}_S] e = \mathbb{E}\left[\sum_{i\in S} \frac{1}{\tau {\tilde p}_i} e_i \right] = \tau \mathbb{E}_{i\sim {\tilde {\cal D}}} \left[  \frac{1}{\tau {\tilde p}_i} e_i \right] = e. 
\]
 For a sampling $S$, we make the following assumption. 
 
 \begin{assumption}\label{as:SthetaS}
 For a sampling $S$ and $\theta_S \in \R^{n\times n}$, $\mathbb{E}[{\theta}_S {\bf I}_S] e = e$. 
 \end{assumption}
 
 It should be noticed that for a proper sampling, Assumption \ref{as:SthetaS} can be satisfies by $\theta_S^i = 1/p_i$. 
 Let $F: \mathbb{R}^d \to \mathbb{R}^n$ be defined by
$
 F(x) \eqdef (f_1(x), ..., f_n(x))^\top$
 and let 
 $
 {\bf G}(x) \eqdef [\nabla f_1(x), ..., \nabla f_n(x)] \in {\mathbb{R}^{d\times n}}
 $
 be the Jacobian of $F$ at $x$. Then the search direction $g^k$ in minibatch setting can be denoted as 
 $$
\compactify g^k = \frac{1}{n}({\bf G}(x^k) - {\bf G}(w^k)){\theta}_{S_k}{\bf I}_{S_k}e + \frac{1}{n}{\bf G}(w^k)e. 
 $$
 
 \subsection{Complexity rates and sparsity}
 
 {\bf Strongly convex case.}  For L-SVRG, the iteration complexity is at least as good as that of SAGA-AS \cite{SAGA-AS} and Quartz \cite{Quartz}. Assume $f$ is $L_f$-smooth and $f_i$ is $L_i$-smooth. For importance sampling and importance sampling with replacement, we can obtain linear speed up with respect to the expected minibatch size $\tau$ until $\tau=n$ or until the iteration complexity becomes ${\cal O}\left(  \left(  \nicefrac{n}{\tau} + \nicefrac{L_f}{\mu}  \right)\log \tfrac{1}{\epsilon}  \right)$, where $\mu$ is the strongly convexity constant of $P$. For L-Katyusha, the iteration complexity is essentially the same with that of Katyusha \cite{Katyusha}, and has linear speed up with respect to the expected minibatch size $\tau$ until $\tau=n$ or until the iteration complexity becomes ${\cal O} ((  \nicefrac{n}{\tau} + \sqrt{\nicefrac{L_f}{\mu} } )\log \nicefrac{1}{\epsilon} )$. While in minibatch setting, Katyusha \cite{Katyusha} is only studied for the sampling with replacement. The estimation of ${\cal L}_2$ also gives the convergence result of Katyusha with arbitrary sampling. Furthermore, L-Katyusha is simpler and faster considering the running time in practice.

 {\bf Nonconvex and smooth case.} The first arbitrary sampling analysis in a nonconvex setting was performed in \cite{Samuel19}. Our iteration complexity of L-SVRG with importance sampling and importance sampling with replacement is at least as good as that of SAGA and SVRG with optimal sampling in \cite{Samuel19}, and could be better if $L_f$ is smaller than ${\bar L} \eqdef \sum_{i\in [n]}L_i/n$. Moreover, we can obtain linear speed up with respect to $\tau$ until $\tau = n$ or until the iteration complexity becomes ${\cal O}\left(  \nicefrac{L_f}{\epsilon}  \right)$, while the results in \cite{Samuel19} holds for $\tau\leq {\cal O}(n^{\nicefrac{2}{3}})$ only.

 {\bf Sparsity} All our convergence results rely on some expected smoothness parameters such that we can analyze the algorithms with arbitrary sampling and sampling with replacement in a framework. We establish the connection between these expected smoothness parameters and ESO \cite{PCDM}, which allows us the explore the sparsity of data as well.

\section{Strongly Convex Case}


In this section, we develop loopless SVRG and loopless Katyusha for problem (\ref{primal}).  Throughout this section we make the following 
assumption on the functions $f$ and $\psi$.
\begin{assumption}\label{ass:strongconvex}
There is $L_f>0$ such that $\tfrac{L_f}{2}\|x\|^2-f(x)$ is convex.
There are $\mu_f\geq 0$ and $\mu_\psi\geq 0$ such that $f(x)-\tfrac{\mu_f}{2}\|x\|^2$ and $\psi(x)-\tfrac{\mu_\psi}{2}\|x\|^2$ are convex. Moreover, $\mu=\mu_f+\mu_\psi>0$. 
\end{assumption}
 It should be noticed that the results in this section do not require the convexity of  each $f_i$. Instead, we provide convergence guarantees under some expected smoothness assumptions. 
We shall also need the following standard proximal operator of $\psi$:
$$ \compactify 
\prox_{\eta}(x) \eqdef {\arg\min}_{y} \left\{  \frac{1}{2\eta} \|x-y\|^2 + \psi(y) \right\}, \ \eta>0. 
$$

\subsection{Loopless SVRG (L-SVRG)}
The loopless SVRG with arbitrary sampling is described in Algorithm~\ref{alg:lsvrg}.  The convergence will rely on the following assumption.
\begin{assumption}
	[Expected smoothness]\label{as:expsmooth}
	There is a constant ${\cal L}_1> 0$ such that  for any $x\in \R^d$,
	$$
	\mathbb{E}\left[\left\|\tfrac{1}{n}({\bf G}(x) - {\bf G}(x^*)){\theta}_{S}{\bf I}_{S}e \right\|^2 \right] \leq 2{\cal L}_1(f(x) - f(x^*) - \langle \nabla f(x^*), x - x^*\rangle ). 
	$$
	
\end{assumption}

\begin{algorithm}[h]
	\caption{Loopless SVRG (L-SVRG)}
	\label{alg:lsvrg}
	\begin{algorithmic}[1]
		\Require
		stepsize $\eta >0$, probability $p \in (0, 1]$, sampling or sampling with replacement $S$
		\Ensure
		$x^0 = w^0 \in \R^d$
		\For{ $k = 0, 1, 2, ...$}
		\State Sample $S_k \sim S$
		\State $g^k = \tfrac{1}{n}({\bf G}(x^k) - {\bf G}(w^k)){\theta}_{S_k}{\bf I}_{S_k}e + \tfrac{1}{n}{\bf G}(w^k)e$ 
		\State $x^{k+1} = \prox_{\eta} \left(x^k - \eta g^k \right) $
		\State $
		w^{k+1} = \left\{ \begin{array}{rl}
		x^k & \mbox{ with propobility $p$} \\
		w^k &\mbox{ with probability $1-p$}
		\end{array} \right.
		$
		\EndFor
	\end{algorithmic}
\end{algorithm}

\noindent Denote $\mathbb{E}_k[\cdot]$ as the conditional expectation on ${\bf G}(w^k)$ and $x^k$. Consider the stochastic Lyapunov function  $
	\Psi^k_S \eqdef \|x^k -x^*\|^2 + {\cal D}^k_S
	$, where
$
{\cal D}^k_S \eqdef \tfrac{1}{1 + \eta \mu_{\psi}} \cdot \tfrac{4\eta^2}{p} \cdot \left\| \tfrac{1}{n}({\bf G}(w^k) - {\bf G}(x^*)){\theta}_{S}{\bf I}_{S}e \right\|^2 .
$

\begin{theorem}\label{Th:convp}
	Under Assumption~\ref{as:expsmooth},
	if stepsize $\eta$ satisfies $
	\eta \leq \tfrac{1}{6{\cal L}_1}$,
	then 
	$$\compactify 
	\mathbb{E}_k\left[ \|x^{k+1} - x^*\|^2 + {\cal D}^{k+1}_S \right]  \leq \left(1- \frac{\eta \mu}{1 + \eta\mu_{\psi}}\right)\|x^k-x^*\|^2 + \left(1-\frac{p}{2}\right){\cal D}^k_S.
	$$
	If we choose $\eta = \tfrac{1}{6{\cal L}_1}$, then   $\mathbb{E}[{\Psi}^k_S] \leq \epsilon \cdot \mathbb{E}[{\Psi}^0_S]$ as  long as 
	$
	k\geq {\cal O}\left( \left( \tfrac{1}{p} + \tfrac{{\cal L}_1}{\mu}\right) \log\tfrac{1}{\epsilon}\right) .
	$
\end{theorem}

If the sampling $S$ has expected size $\tau$, then the expected iteration cost is $O(\tau+np)$ and 
the expected batch complexity is
$
{\cal O} \left(\left(n+ \tfrac{\tau}{p}+\tfrac{{\cal L}_1 (\tau+np)}{\mu}\right) \log\tfrac{1}{\epsilon}\right),
$
which is $\cO \left(\left(n+\tfrac{{\cal L}_1 \tau}{\mu}\right) \log\tfrac{1}{\epsilon}\right)
$ for any $p$ between $\tau/n$ and $\mu/\cL_1$.  
In the serial and uniform sampling case, i.e., when $\tau=1$ and $p_i=1/n$, 
Algorithm~\ref{alg:lsvrg} and Thm~\ref{Th:convp}  recovers the loopless SVRG algorithm and convergence result given in~\cite{LSVRG}, where  can be found a detailed comparison with the original SVRG method.

\subsection{Loopless Katyusha (L-Katyusha)}
The loopless Katyusha is given in Algorithm~\ref{alg:lkatyusha}.  We shall need the following assumption.

\begin{assumption}\label{as:expL2}
	There is a constant ${\cal L}_2> 0$ such that  for all $x, y\in \R^d$,
	$$
	\mathbb{E}\left[\left\|\tfrac{1}{n}({\bf G}(x) - {\bf G}(y)){\theta}_{S}{\bf I}_{S}e - \tfrac{1}{n}({\bf G}(x) - {\bf G}(y))e\right\|^2 \right] \leq 2{\cal L}_2(f(x) - f(y) - \langle \nabla f(y), x - y\rangle ). 
	$$
\end{assumption}

\begin{algorithm}[h]
	\caption{Loopless Katyusha (L-Katyusha)}
	\label{alg:lkatyusha}
	\begin{algorithmic}[1]
		\Require
		$\mu_f$, $\mu_{\psi}$, ${\cal L}_2$, $L_f$, probability $p \in (0, 1]$, sampling or sampling with replacement $S$
		\Ensure
		$y^0 = z^0 = w^0 \in \R^d$, $\eta = \tfrac{1}{3\theta_1}$, $\sigma_1 = \tfrac{\mu_f}{L}$, $\sigma_2 = \tfrac{\mu_{\psi}}{L}$, $L = \max\{ {\cal L}_2, L_f \}$, $\mu = \mu_f + \mu_{\psi}$, $\sigma = \tfrac{\mu}{L}$, 
		$\theta_2 = \tfrac{{\cal L}_2}{2L} = \min\{ \tfrac{{\cal L}_2}{2L_f}, \tfrac{1}{2} \}$. If $L_f \leq \tfrac{{\cal L}_2}{p}$, $\theta_1 = \min\{ \sqrt{\tfrac{\mu}{{\cal L}_2p}}\theta_2, \theta_2  \}$, else $\theta_1 = \min\{ \sqrt{\tfrac{\mu}{L_f}}, \tfrac{p}{2}  \}$. 
		\For{ $k = 0, 1, 2, ...$}
		\State $x^k = \theta_1 z^k + \theta_2 w^k + (1-\theta_1 -\theta_2)y^k$
		\State Sample $S_k \sim S$
		\State $g^k = \tfrac{1}{n}({\bf G}(x^k) - {\bf G}(w^k)){\theta}_{S_k}{\bf I}_{S_k}e + \tfrac{1}{n}{\bf G}(w^k)e$ 
		\State $z^{k+1} = \prox_{\tfrac{\eta}{(1+ \eta\sigma_1)L}} \left( \tfrac{1}{1+ \eta \sigma_1}(\eta \sigma_1 x^k + z^k - \tfrac{\eta}{L}g^k )  \right). $
		\State $y^{k+1} = x^k + \theta_1 (z^{k+1} - z^k)$
		\State $
		w^{k+1} = \left\{ \begin{array}{rl}
		x^k & \mbox{ with propobility $p$} \\
		w^k &\mbox{ with probability $1-p$}
		\end{array} \right.
		$
		\EndFor
	\end{algorithmic}
\end{algorithm}

We define the Lyapunov function $
\Psi^k \eqdef {\cal Z}^k + {\cal Y}^k + {\cal W}^k,
$
where 
$$\compactify 
{\cal Z}^k = \frac{L(1+\eta \sigma)}{2\eta} \|z^k - x^*\|^2, ~\ {\cal Y}^k = \frac{1}{\theta_1} (P(y^k) - P(x^*)),~
{\cal W}^k = \frac{\theta_2}{pq\theta_1} (P(w^k) - P(x^*)),
$$
for some $0< q < 1$. It should be noticed that the definitions of ${\cal Z}^k$ is the same as that of \cite{LSVRG}, but ${\cal Y}^k$ and ${\cal W}^k$ are different. 

\vskip 2mm

\begin{theorem}\label{Th:LKa}
	Under Assumption~\ref{as:expL2}, 
	$$\compactify 
	\mathbb{E}_k[{\cal Z}^{k+1} + {\cal Y}^{k+1} + {\cal W}^{k+1}] \leq \frac{1}{1+\eta\sigma}{\cal Z}^k + \left(1- \left(\theta_1 + \theta_2 - \frac{\theta_2}{q} \right)\right){\cal Y}^k + \left(1 - p(1-q)\right){\cal W}^k.
	$$
	Moreover, with some $q \in [ \tfrac{2}{3}, 1  )$,  $\mathbb{E}[\Psi^k] \leq \epsilon \Psi^0$  for 
	$
	k \geq  {\cal O}\left( \left( \tfrac{1}{p} + \sqrt{\tfrac{L_f}{\mu}} + \sqrt{\tfrac{{\cal L}_2}{\mu p}}  \right) \log\tfrac{1}{\epsilon}\right).
	$
	\end{theorem}

\begin{remark}\label{rem:LK}
There are two major differences between L-Katyusha (Algorithm~\ref{alg:lkatyusha}) and the original Katyusha 
algorithm~\cite{Katyusha}. 

1. We here consider both arbitrary sampling and sampling with replacement while 
Katyusha~\cite{Katyusha} only considered the second case.  Note however that our Assumption~\ref{as:expL2} allows to easily extend  the original Katyusha~\cite{Katyusha} method into arbitrary sampling scheme as well, by simply replacing everywhere the $ \nicefrac{\bar L}{b}$ in their proof by the constant $\cL_2$. This also yields a direct extension of 
Katyusha when $f_i$ are not necessarily convex.

2.  Our method is loopless and the reference point $w^k$ is set to be $x^{k-1}$ with 
probability $p$. Recall that in the original Katyusha method, the reference point $\tilde x^{s}$ for 
each outer loop $s$ is set to be a weighted average of past iterates of $y^k$. 
Not only this difference brings a simplified algorithm and proof, but also a non-negligible practical
convergence speed up. Indeed, the number of epochs of the two methods are 
essentially the same (see Sec~\ref{sec:eep}), but the computation overhead caused by the calculation of $\tilde x^s$
makes Katyusha slower than our loopless variant, especially in the case when a sparse implementation is needed.
We provide in Appendix~\ref{sec:EI} further details. In Sec~\ref{Sec:num} we show through
 numerical evidence the better convergence speed of our loopless variant, see Fig~\ref{fig8} and Fig~\ref{fig10}.

\end{remark}




\section{L-SVRG in the Non-Strongly Convex Case}

In this section, we consider L-SVRG assuming only $f$  being convex. We assume that the expected smoothness (Assumption~\ref{as:expsmooth}) holds for at least one optimal solution $x^*$.
Consider the Lyapunov function 
$
\Psi^k \eqdef \tfrac{1}{2\eta} \|x^k - x^*\|^2 + \alpha {\cal H}^k_S,
$
where $\alpha = \tfrac{6\eta}{5 p}$, $\beta = \tfrac{5}{6\eta}$, and
$
{\cal H}^k_S \eqdef  \left\| \tfrac{1}{n}({\bf G}(w^k) - {\bf G}(x^*)){\theta}_{S}{\bf I}_{S}e \right\|^2 . 
$

\begin{theorem}\label{Th:gconvex1}
Under Assumption~\ref{as:expsmooth},
 if   $\eta \leq \min\{  \tfrac{1}{8{\cal L}_1}, \tfrac{1}{6L_f}  \}$, then
$$
\compactify \mathbb{E}_k[P(x^{k+1}) - P(x^*)]  - \frac{3}{5} \left(P(x^k) - P(x^*) \right) \leq \Psi^k - \mathbb{E}_k[\Psi^{k+1}] . 
$$
Let $\tilde x^k=(x^0+\dots+x^k)/(k+1)$.
 Assume $L_f \leq {\cal L}_1$. Let $\eta = \tfrac{1}{8{\cal L}_1}$. Then 
$$
\compactify  \mathbb{E}[P(\tilde{x}^k) - P(x^*)] \leq \frac{1}{k+1} \left( 10{\cal L}_1 \|x^0 - x^*\|^2 +  \left(\frac{5}{2} + \frac{3}{4p} \right)\left(P(x^0) - P(x^*)\right) \right). 
$$
This implies that $\mathbb{E}[P(\tilde x^k) - P(x^*)] \leq \epsilon$ as long as 
$
k \geq {\cal O} \left(  \left( {\cal L}_1 + \tfrac{1}{p}\right) \frac{1}{\epsilon}  \right). 
$

\end{theorem}

Thm~\ref{Th:gconvex1} can be compared with Thm 3 in~\cite{VRSGD}. However, note that 
the reference point in our loopless SVRG is simply chosen to be $x^k$ without the need 
of extra averaging step. As discussed in Remark~\ref{rem:LK}, the loopless variant has both 
simpler implementation and faster  convergence speed.

\section{L-SVRG in the Nonconvex and Smooth Case }


In this section, we consider L-SVRG with $\psi \equiv 0$ and $f$ being possibly nonconvex. 

\begin{assumption}\label{as:expL3}
There is a constant ${\cal L}_3> 0$ such that 
$$
\compactify 
\mathbb{E}\left[\left\|\frac{1}{n}({\bf G}(x) - {\bf G}(y)){\theta}_{S}{\bf I}_{S}e - \frac{1}{n}({\bf G}(x) - {\bf G}(y))e\right\|^2 \right] \leq {\cal L}_3 \|x-y\|^2, \ \ \forall x, y\in \R^d. 
$$
\end{assumption}

\begin{theorem}\label{Th:nonconvex1} Consider the Lyapunov function
$
\Psi^k \eqdef f(x^k) + \alpha \|x^k - w^k\|^2, 
$
where $\alpha = 3\eta^2L_f{\cal L}_3/p$. Let $\beta = p/3\eta$. If stepsize $\eta$ satisfies 
\begin{equation}\label{eq:etanonconvex}
\compactify  \eta \leq \min \left\{ \frac{1}{4L_f}, \frac{p^{\frac{2}{3}} }{36^{\frac{1}{3}}(L_f{\cal L}_3)^{\frac{1}{3}} }, \frac{\sqrt{p}}{\sqrt{6{\cal L}_3}}  \right\}, 
\end{equation}
then 
$
\mathbb{E}_k [\Psi^{k+1}] \leq \Psi^k - \tfrac{\eta}{4}\|\nabla f(x^k) \|^2. 
$
\end{theorem}

\begin{corollary}\label{co:nonconvex1}
Let $x^a$ be chosen uniformly at random from $\{  x^i  \}_{i=0}^k$ and the stepsize $\eta$ satisfy (\ref{eq:etanonconvex}). 
Then 
$
\mathbb{E}[\|\nabla f(x^a)\|^2 ] \leq \tfrac{4}{\eta}\cdot \tfrac{f(x^0) - f(x^*)}{k+1}. 
$ If the stepsize $\eta$ is equal to the upper bound in (\ref{eq:etanonconvex}), then $\mathbb{E}[\|\nabla f(x^a)\|^2 ]  \leq \epsilon$ as long as 
$$
\compactify 
k \geq {\cal O} \left(  \left(L_f + \frac{(L_f{\cal L}_3)^{\frac{1}{3}}}{p^{\frac{2}{3}}} +  \sqrt{\frac{{\cal L}_3}{p}} \right)\frac{f(x^0) - f(x^*)}{\epsilon}  \right). 
$$
\end{corollary}

\section{Estimation of Expected Smoothness Parameters}\label{sec:eep}
In this section, we study the constants $\cL_1$,  $\cL_2$  and $\cL_3$ under various circumstances and compare the corresponding iteration complexities of our loopless algorithms. For estimation of $\cL_1$ and $\cL_2$ we require 
the convexity of each $f_i$. We list the upper bounds of the three constants in Table \ref{tab:L1L2L3}. The proofs can be found in Secs \ref{sec:L1L2} and \ref{sec:L3} in Appendix.

{\bf Importance Sampling.}
Let $\tau=\bE[|S|]$ be the expected cardinality of $S$, counting multiplicity.
Since the complexity  bound of the algorithms increase with the constants $\cL_1$,  $\cL_2$  and $\cL_3$.
It is natural to choose the sampling strategy minimizing those constants. 
The detailed analysis can be found in Sec \ref{sec:importance} and Propositions \ref{pro:L3gs}, \ref{pro:L3sp} in Appendix. We summarize the results as following. For group sampling,  let 
$
q_i = {L_i \tau}/(\sum_{i=1}^n L_i), 
$
and choose $p_i$ such that $\min\{q_i, 1\} \leq p_i \leq 1$ and $\sum_{i=1}^np_i = \tau$. Then \begin{align}\label{a:L123}
\compactify \cL_1\leq L_f +\frac{\bar L}{\tau}, \qquad  \cL_2\leq \frac{\bar L}{\tau} ,  \qquad \cL_3\leq \frac{\bar L^2}{\tau}.
\end{align}
For sampling with replacement, by choosing
$\tilde p_i=L_i/(\sum_i L_i)$, the same bound as~\eqref{a:L123} is guaranteed.
Next we insert the bound~\eqref{a:L123} into previous results to get directly the complexity of each algorithm.  Although our theory allows arbitrary changing probability $p$,  for simplicity
we consider $p=\tau/n$. In this case, the expected cost of each iteration is $2\tau$.

{\bf L-SVRG.}
When $f$ or $\psi$ is strongly convex,  by Thm~\ref{Th:convp}, the iteration complexity of L-SVRG is $
{\cal O} \left(\left(\nicefrac{n}{\tau}+\nicefrac{L_f}{\mu}+\nicefrac{\bar L}{\mu\tau}\right) \log\nicefrac{1}{\epsilon}\right)$ with $\mu=\mu_f+\mu_\psi$. Such complexity bound is comparable with that of SAGA-AS with importance mini-batch sampling~\cite{SAGA-AS}. Note that as SAGA-AS, L-SVRG does not need to know the strong convexity parameter $\mu$. For arbitrary sampling, L-SVRG is at least as good as SAGA-AS and Quartz. The detailed comparison can be found in Sec~\ref{sec:importance} in Appendix.

When $f$ is convex,  by Thm~\ref{Th:gconvex1}, the iteration complexity of L-SVRG is ${\cal O} \left(  \left(\nicefrac{n}{\tau}+ L_f + \nicefrac{\bar L}{\tau}\right) \nicefrac{1}{\epsilon}  \right)$.  Therefore, linear speedup is achieved when $\tau\leq \nicefrac{\bar L}{L_f}$. 

When $f$ is {nonconvex $f$ and $\psi\equiv 0$},   by Corollary~\ref{co:nonconvex1}, the iteration complexity of L-SVRG is 
\[ \compactify  {\cal O} \left(  \left(L_f + \frac{n^{\frac{2}{3}}(L_f{\bar L}^2)^{\frac{1}{3}}}{\tau} +  \frac{\sqrt{n}{\bar L}}{\tau} \right)\frac{1}{\epsilon}  \right). \] In~\cite{Samuel19}, the iteration complexity for SVRG and SAGA with importance sampling is proved to be
\[ \compactify  {\cal O}\left(  \frac{(1 + \frac{n-\tau}{n}) {\bar L} n^{\frac{2}{3}}}{\tau \epsilon}  \right)\]
for $\tau \leq {\cal O}(n^{\nicefrac{2}{3}})$. We can see our bound is at least as good as theirs, and could be better  if $L_f$ is smaller than ${\bar L}$. Furthermore, our bound  holds for any $1\leq \tau \leq n$, while the one in~\cite{Samuel19}  only holds for $\tau \leq {\cal O}(n^{\nicefrac{2}{3}})$.

\begin{table*}[t]
	\begin{center}
		{\footnotesize
			\begin{tabular}{|c|c|c|c|}
				\hline
				  \begin{tabular}{c}
				  	\quad \\
				  	\quad\\
				  	\end{tabular} & ${\cal L}_1$ &   ${\cal L}_2$ & ${\cal L}_3$ \\
				\hline 
				AS  & 
				\begin{tabular}{c} 
					$\tfrac{1}{n}\max_{i\in [n]} \left\{ \sum_{j\in [n]} \right.$\\
					$\left.\sum_{C: i, j\in C}p_C{\theta}^i_C{\theta}^j_C L_j   \right\}$
				\end{tabular} 
				& $\leftarrow$  & 
				\begin{tabular}{c}
					$\tfrac{1}{n^2} \sum_{i, j=1}^n \sum_{C: i, j \in C} $\\
					$p_C \theta^i_C\theta^j_C L_iL_j$ 
					\end{tabular}\\
				\hline
				\begin{tabular}{c} 
					AS \\
					$\theta^i_S \equiv \tfrac{1}{p_i}$
				\end{tabular} & 
				\begin{tabular}{c}
				$\tfrac{1}{n}\max_{i\in [n]}$ \\ 
				$\left\{ \sum_{j\in [n]}\tfrac{{\bf P}_{ij}}{p_ip_j}L_j   \right\}$ 
				\end{tabular} & $\leftarrow$ & $ \tfrac{1}{n^2} \sum_{i, j =1}^n \tfrac{{\bf P}_{ij}}{p_ip_j} L_iL_j$    \\ 
				\hline
				\begin{tabular}{c}
				\quad \\
				AS\\ 
				\quad 
				\end{tabular}
				& $\tfrac{1}{n} \max_i \{  L_i \beta_i  \}$   & $\leftarrow$ & $\tfrac{1}{n^2} \sum_{i=1}^n \beta_i L_i^2$ \\ 
				\hline
				\begin{tabular}{c}
				$\tau$-NS\\ 
				$\theta^i_S \equiv \tfrac{1}{p_i}$
				\end{tabular} & 
				\begin{tabular}{c}
				$\tfrac{n(\tau -1)}{\tau (n-1)} L_f + $ \\ 
				$\tfrac{n-\tau}{\tau (n-1)} \max_i \{ L_i  \}$
				\end{tabular} & $\tfrac{n-\tau}{\tau (n-1)} \max_i \{ L_i  \}$ & $ \tfrac{n-\tau}{\tau (n-1)}\cdot \tfrac{1}{n}\sum_{i=1}^n L_i^2$  \\
				\hline 
				\begin{tabular}{c}
					GS \\
					${\theta^i_S \equiv \tfrac{1}{p_i}}$
				\end{tabular}  &
				\begin{tabular}{c}
				$L_f + $ \\ 
				$\tfrac{1}{n} \max\left\{ \max_{i \notin {\cal I}} \tfrac{L_i}{p_i}, \right.$ \\
				$\left.  \max_{i\in {\cal I}} (\tfrac{1}{p_i} - 1)L_i    \right\}$
				\end{tabular} &
				\begin{tabular}{c}
				$\tfrac{1}{n} \max\left\{  \max_{i \notin {\cal I}} \tfrac{L_i}{p_i}, \right.$ \\ 
				$\left.  \max_{i\in {\cal I}} (\tfrac{1}{p_i} - 1)L_i , \right\}$ 
				\end{tabular}& 
				\begin{tabular}{c}
				$\tfrac{1}{n^2} \left(  \sum_{i\in {\cal I}} (\tfrac{1}{p_i} -1)L_i^2 \right.$\\ 
				$\left.  + \sum_{i \notin {\cal I}} \tfrac{1}{p_i} L_i^2  \right)$ 
				\end{tabular}
				\\
				\hline 
				\begin{tabular}{c}
				AS \\
				${\theta^i_S} = \tfrac{1}{p_i}$ \\ 
				$\&$ ESO
				\end{tabular} & 
				$\tfrac{1}{n\gamma} \max_i \{  \tfrac{v_i}{p_i}  \}$ & $\leftarrow$ & 
				$\tfrac{1}{n^2 \gamma^2} \sum_{i=1}^n \tfrac{v_i\|{\bf A}_i\|^2}{p_i}$
				\\
				\hline
				\begin{tabular}{c}
				SR 
				\end{tabular} & 
				\begin{tabular}{c}
				$(1- \tfrac{1}{\tau})L_f + $ \\
				$\tfrac{1}{n\tau} \max_{i} \tfrac{L_i}{{\tilde p}_i}$ 
				\end{tabular} & 
				$ \tfrac{1}{n\tau} \max_{i} \tfrac{L_i}{{\tilde p}_i}$ & 
				$\tfrac{1}{n^2\tau} \sum_{i=1}^n \tfrac{L_i^2}{{\tilde p}_i}$ 
				\\
				\hline
			\end{tabular}
		}
	\end{center} 
	\caption{\footnotesize Upper bounds for ${\cal L}_1$, ${\cal L}_2$, and ${\cal L}_3$ (AS = arbitrary sampling, $\tau$-NS = $\tau$-nice sampling, GS  = group sampling, SR= sampling with replacement). In the column where ${\cal L}_2$ belongs, "$\leftarrow$" means the value is same as that of ${\cal L}_1$. We need $f_i$ to be convex in the estimations of ${\cal L}_1$ and ${\cal L}_2$. We have ${p_i} \eqdef \Prob(i\in S)$ for sampling $S$, and  $\beta_i \eqdef \sum_{C \subseteq[n] : i\in C}p_C|C|(\theta_C^i)^2$, where $|C|$ is the cardinality of the set $C$. The property of $\beta_i$ can be found in Lemma 3.4 in \cite{SAGA-AS}. The definitions of $v_i$ and ESO can be found in (\ref{eq:ESOfirst}) in Appendix.      }
	\label{tab:L1L2L3}
\end{table*}

{\bf L-Katyusha.}
When $f$ or $\psi$ is strongly convex, by Thm~\ref{Th:LKa}, the iteration complexity of L-Katyusha is
$
 {\cal O}(( \nicefrac{n}{\tau} + \sqrt{\nicefrac{L_f}{\mu}} + \nicefrac{1}{\tau}\sqrt{\nicefrac{n \bar L}{\mu}} ) \log\nicefrac{1}{\epsilon}).$  This is  the same iteration complexity bound of the original Katyusha with importance sampling with replacement in \cite{Katyusha}.   The numerical experiments also confirm the similarity of the two methods
 in terms of iteration complexity, see Fig~\ref{fig7} and Fig~\ref{fig9}.

\section{Numerical Experimentation}\label{Sec:num}

We tested L-SVRG (Algorithm~\ref{alg:lsvrg}) and L-Katyusha (Algorithm~\ref{alg:lkatyusha}) on the logistic regression problem with $\lambda_1=10^{-4}$ and different values of $\lambda_2$.
The datasets that we used are all downloaded from https://www.csie.ntu.edu.tw/$\sim$cjlin/libsvmtools/datasets/.  In all the plots, L-SVRG and L-Katyusha refer respectively to Algorithm~\ref{alg:lsvrg} and Algorithm~\ref{alg:lkatyusha} with uniform sampling strategy. 
 L-SVRG IP and L-Katyusha IP mean that importance sampling is used. Katyusha refers to the original Katyusha algorithm proposed in~\cite{Katyusha}. 
 Since in practice group sampling and sampling with replacement have similar convergence behaviour, here we only show the results obtained with sampling with replacement.  In all the plots, the $y$-axis corresponds to the primal dual gap of iterate $\{x^k\}$. The $x$-axis may be the number of epochs,   counted as $k \tau/n$ plus the number of times we change $w^k$, or the actual running time. The experiments were carried out  on a MacBook (1.2 GHz Intel Core m3 with 16 GB RAM) running  MacOS High Sierra 10.13.1.

\textbf{Comparison of L-SVRG and L-Katyusha:}
In Fig~\ref{fig6} and Fig~\ref{fig5} we compare L-SVRG with L-Katyusha, both with importance sampling strategy for w8a and cod\_rna and three different values of $\lambda_2$. In each plot we compare three different minibatch sizes $\tau$. 
The numerical results show that the number of epochs of L-SVRG generally increases with $\tau$ (since ${\bar L}/L_f$ is not large in these examples) while that of L-Katyusha 
is stable and thus achieves a \textbf{linear speedup} in terms of number of epochs.

 \textbf{Comparison of Uniform and Importance Sampling:}
Fig~\ref{1d1}  compares the uniform sampling strategy and the importance sampling strategy, for the dataset cod\_rna and three different values of $\tau$.  As predicted by theory, the importance sampling brings a speedup if $\bar L$ is smaller than $\max_i L_i$. Note that for cod\_rna, $\bar L=259,158$ and $\max_i L_i=3,506,320$.

\textbf{The $w^k$ updating probability:}
Fig~\ref{1d2} and Fig~\ref{1d3} compare the performance of our loopless Katyusha for different choice of $p$. 
Although the total number of epochs does increase with $p\geq \tfrac{\tau}{n}$, the running time can be
significantly reduced by taking $p$ larger than $\tau/n$.

\textbf{Comparison with Katyusha}
Fig~\ref{fig7}, ~\ref{fig8} ~\ref{fig9} and~\ref{fig10}  compare our loopless Katyusha with the original Katyusha proposed in~\cite{Katyusha}, for three different values of $\tau$, based on the importance sampling strategy.
While the performance of the two algorithms are similar in terms of epochs, the actual running time of the loopless variant can be 
20\% to 50\% less than that of Katyusha. This is due to the additional averaging step in the original Katyusha method at the end of every inner loop, see Appendix~\ref{sec:EI} for further details.

 \begin{figure}[!ht]
   \subfigure[$\lambda_2=10^{-3}$]{ \label{5d1}  \includegraphics[scale=0.38]{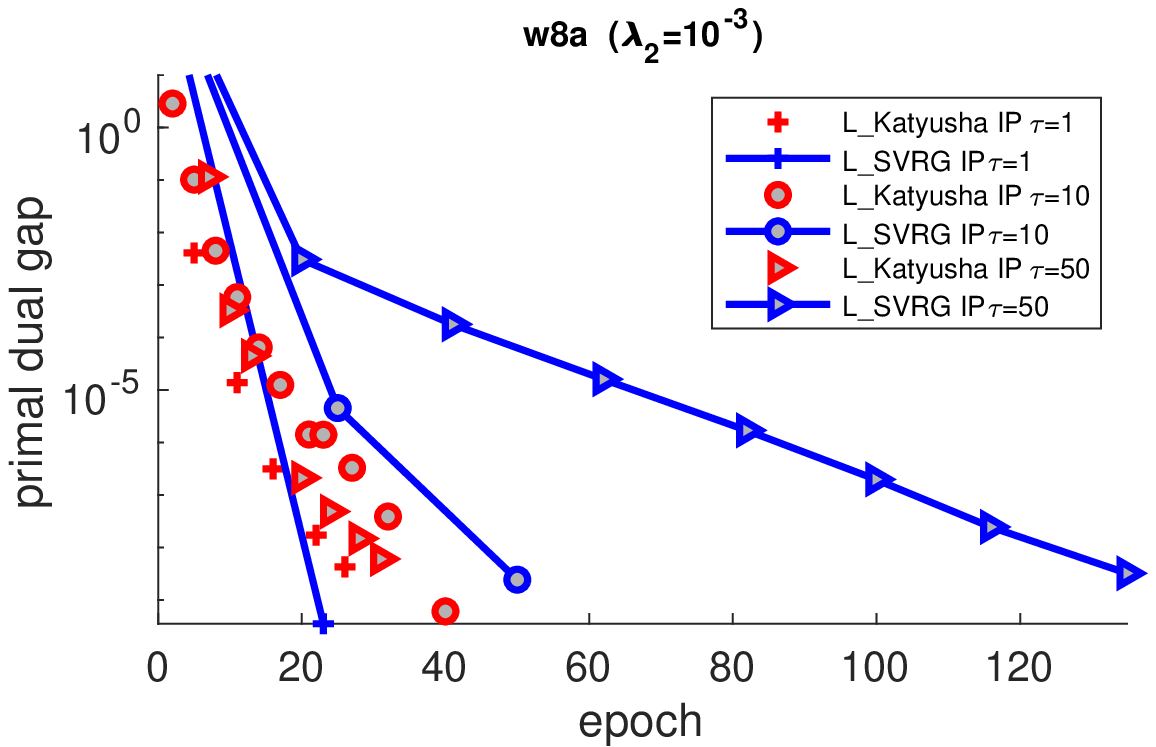}}
  \subfigure[$\lambda_2=10^{-5}$]{\label{5d2}\includegraphics[scale=0.38]{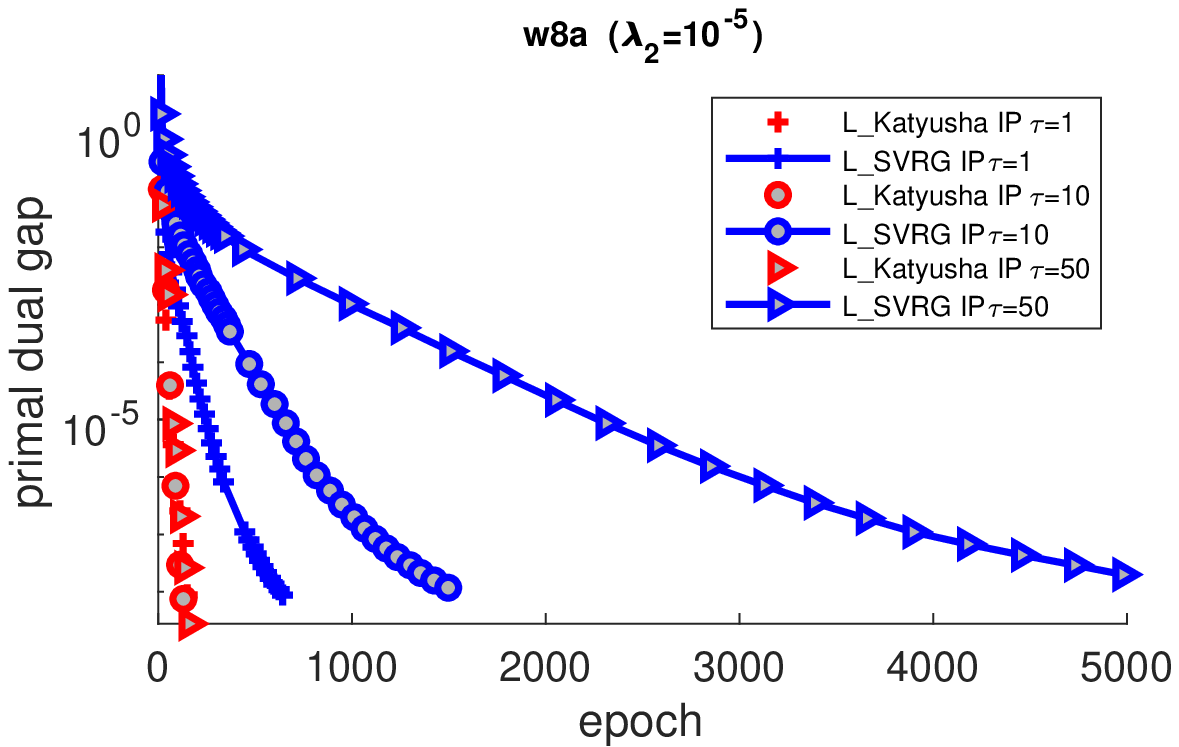}}
  \subfigure[$\lambda_2=10^{-7}$]{\label{5d3}\includegraphics[scale=0.38]{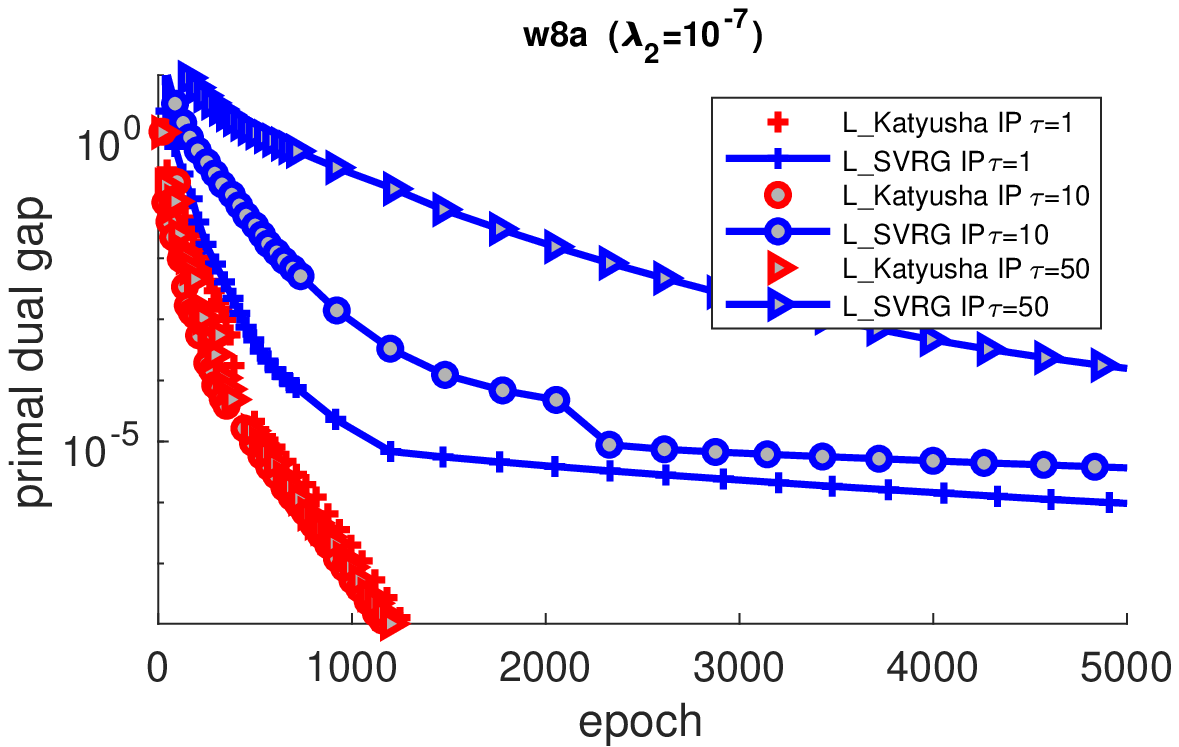}}
\caption{L-SVRG V.S. L-Katyusha, w8a}\label{fig6}
 \end{figure}

\begin{figure}[!ht]\subfigure[uniform V.S. IP]{\label{1d1}
 \includegraphics[scale=0.38]{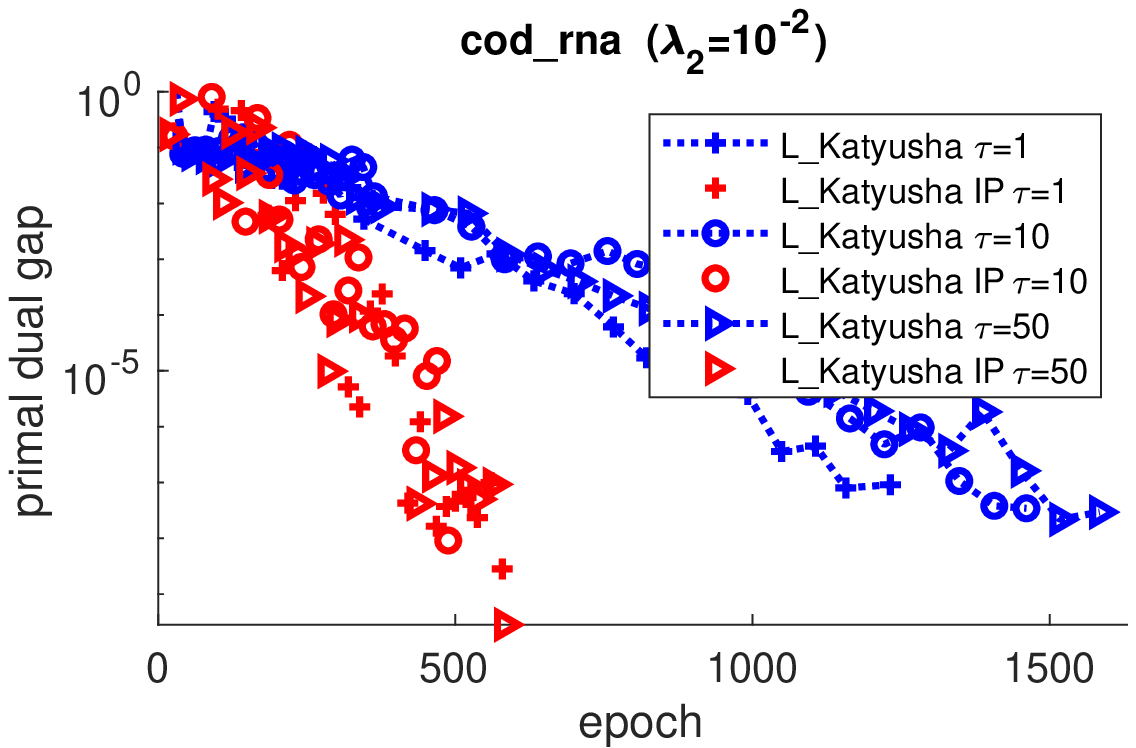}}
    \subfigure[different $p$, epoch]{\label{1d2}  \includegraphics[scale=0.38]{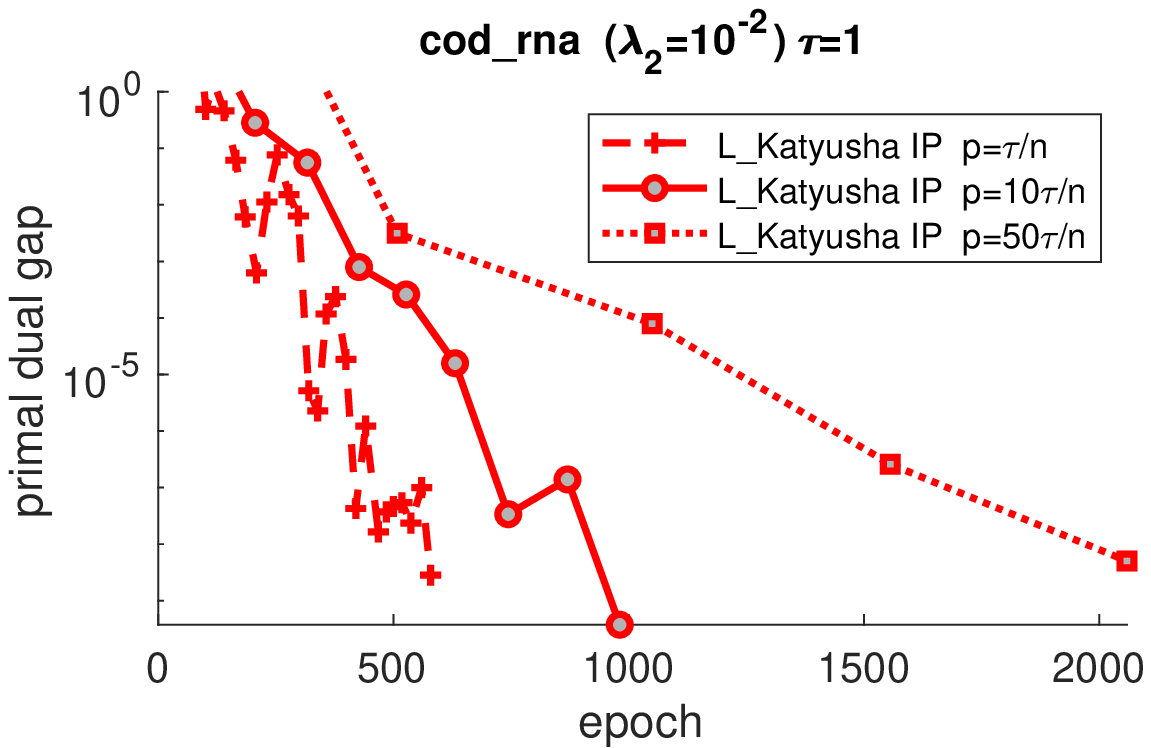}}
\subfigure[different $p$, time]{\label{1d3}  \includegraphics[scale=0.38]{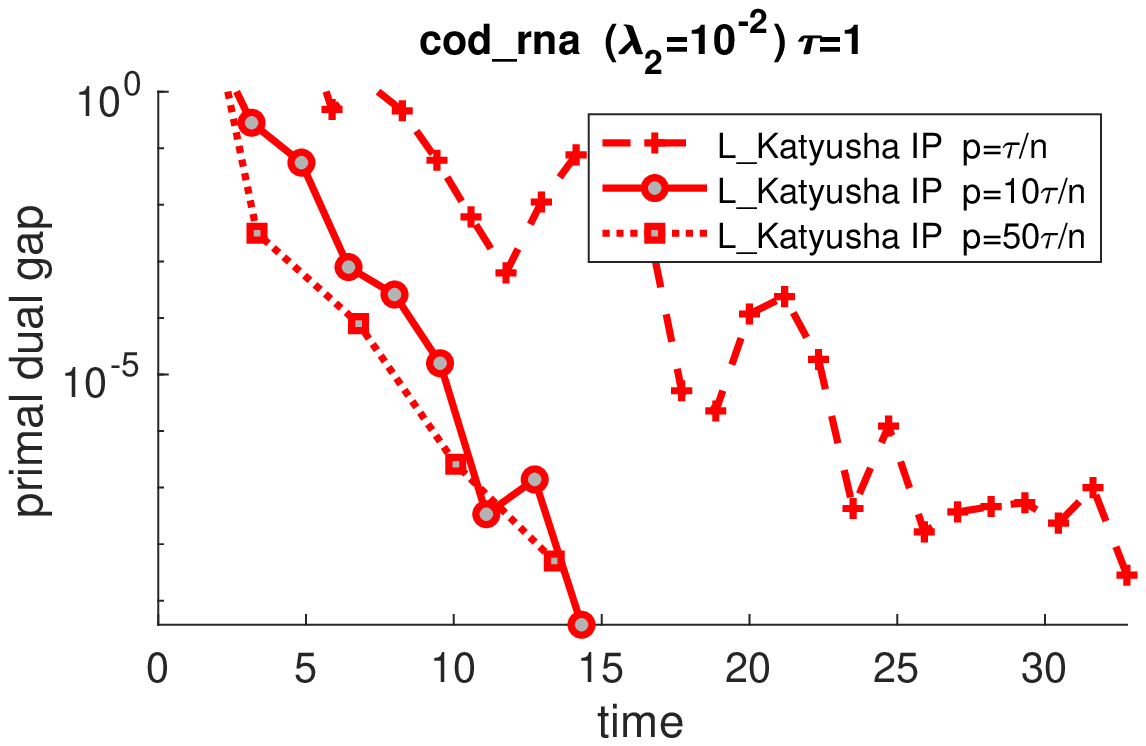}}
\caption{cod-rna}\label{fig1}
 \end{figure}

\begin{figure}[!ht]
 \subfigure[$\tau=1$]{\label{7d1}\includegraphics[scale=0.38]{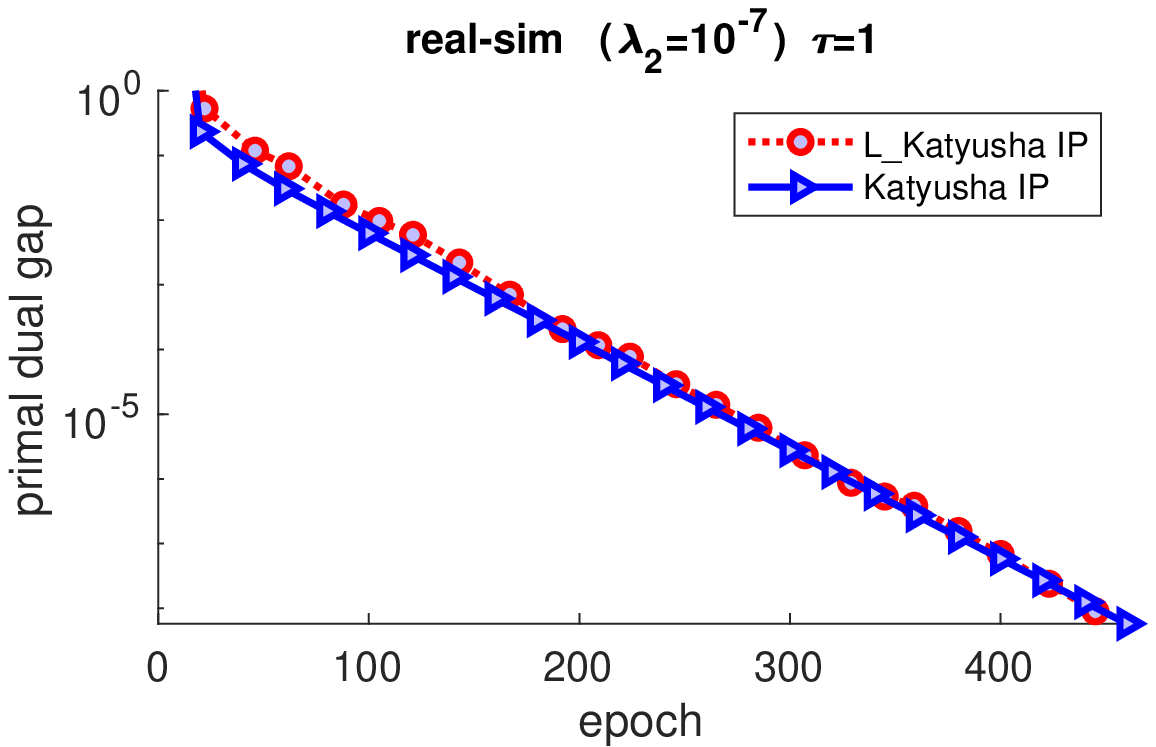}}
  \subfigure[$\tau=10$]{  \label{7d2}  \includegraphics[scale=0.38]{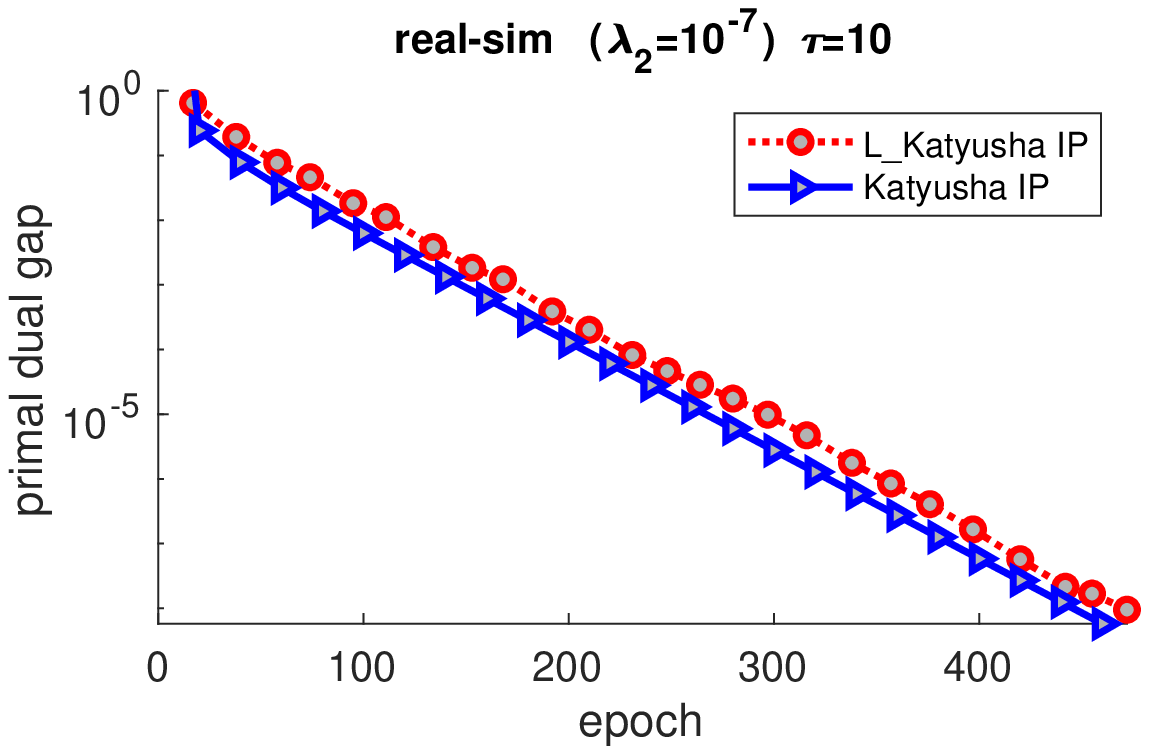}}
    \subfigure[$\tau=50$]{ \label{7d3} \includegraphics[scale=0.38]{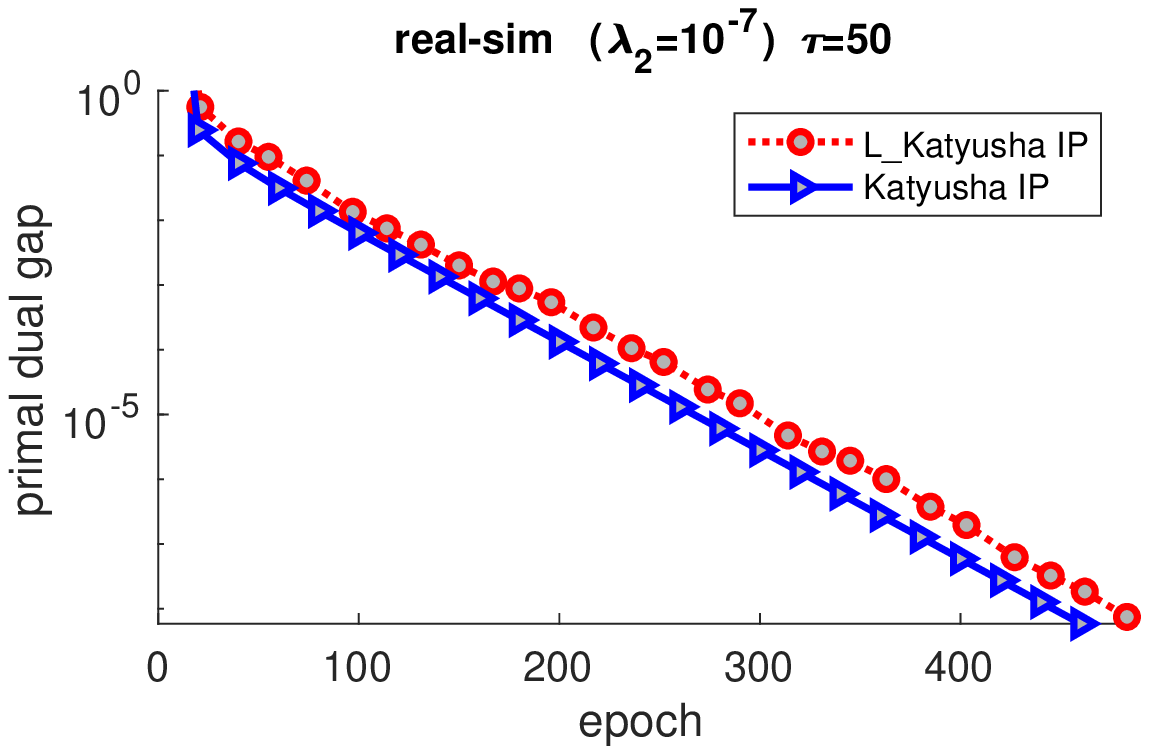}}
    \caption{L-Katyusha V.S. Katyusha, epoch plot, real-sim}\label{fig7}
 \end{figure}
\begin{figure}[!ht]
 \subfigure[$\tau=1$]{\label{8d1}\includegraphics[scale=0.38]{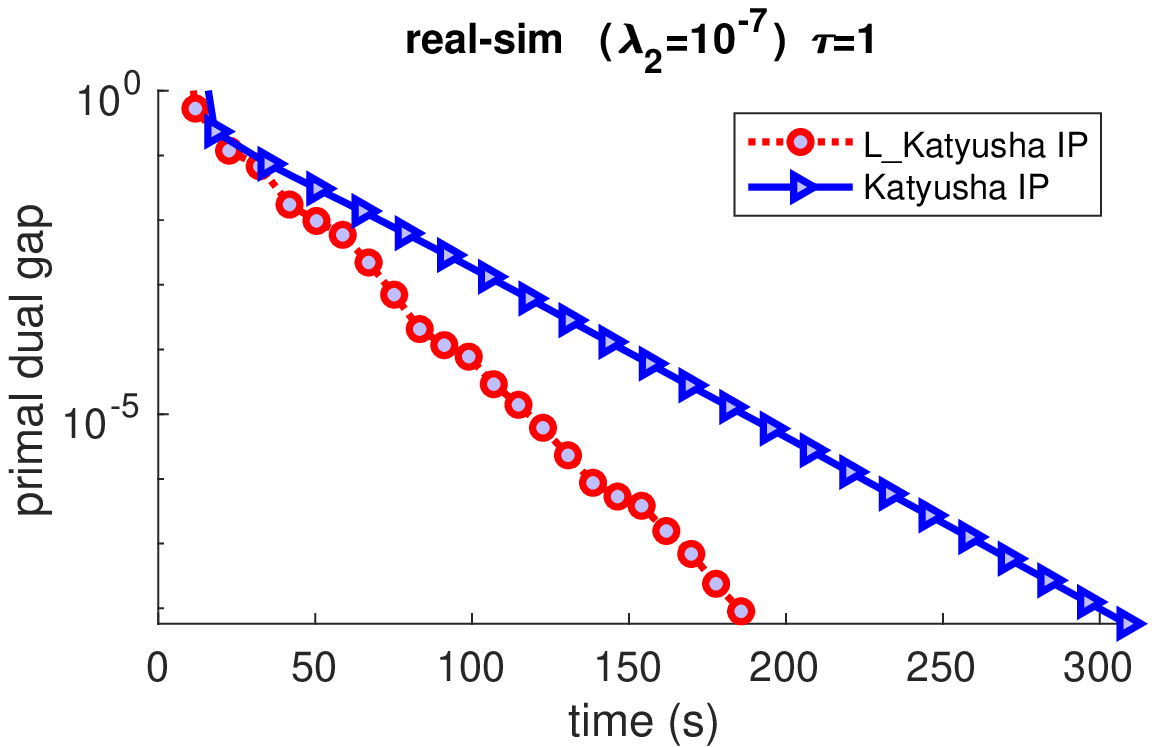}}
  \subfigure[$\tau=10$]{  \label{8d2}  \includegraphics[scale=0.38]{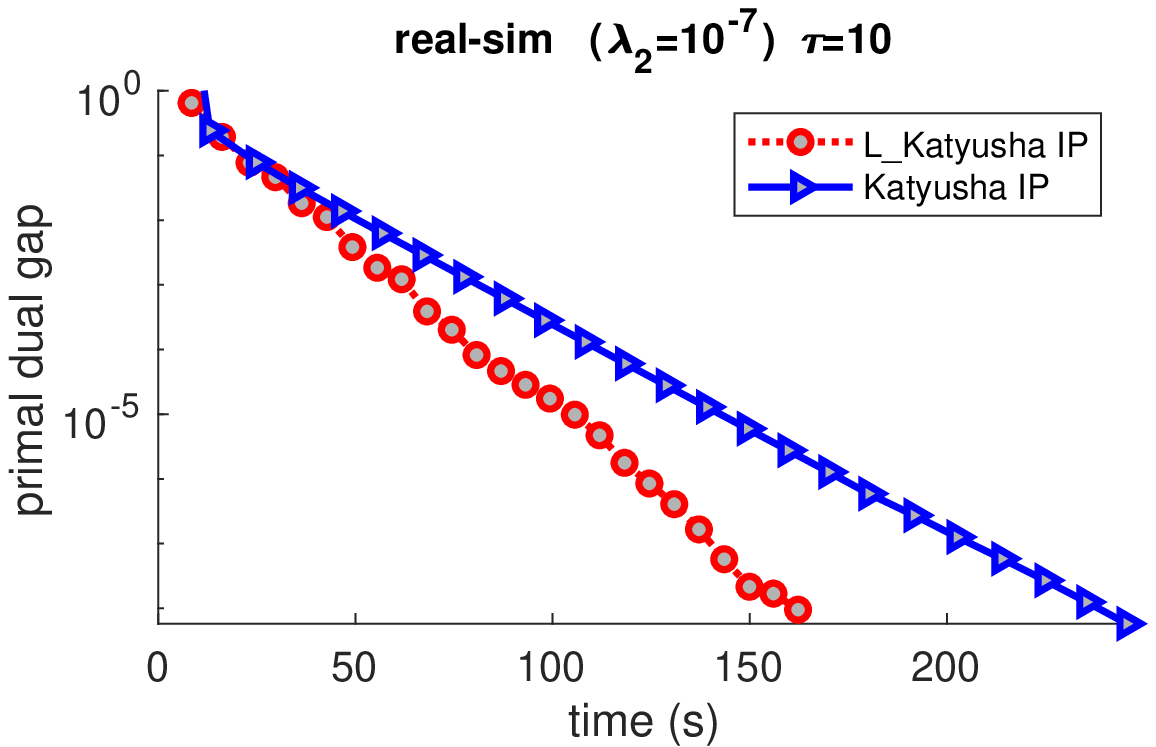}}
    \subfigure[$\tau=50$]{ \label{8d3} \includegraphics[scale=0.38]{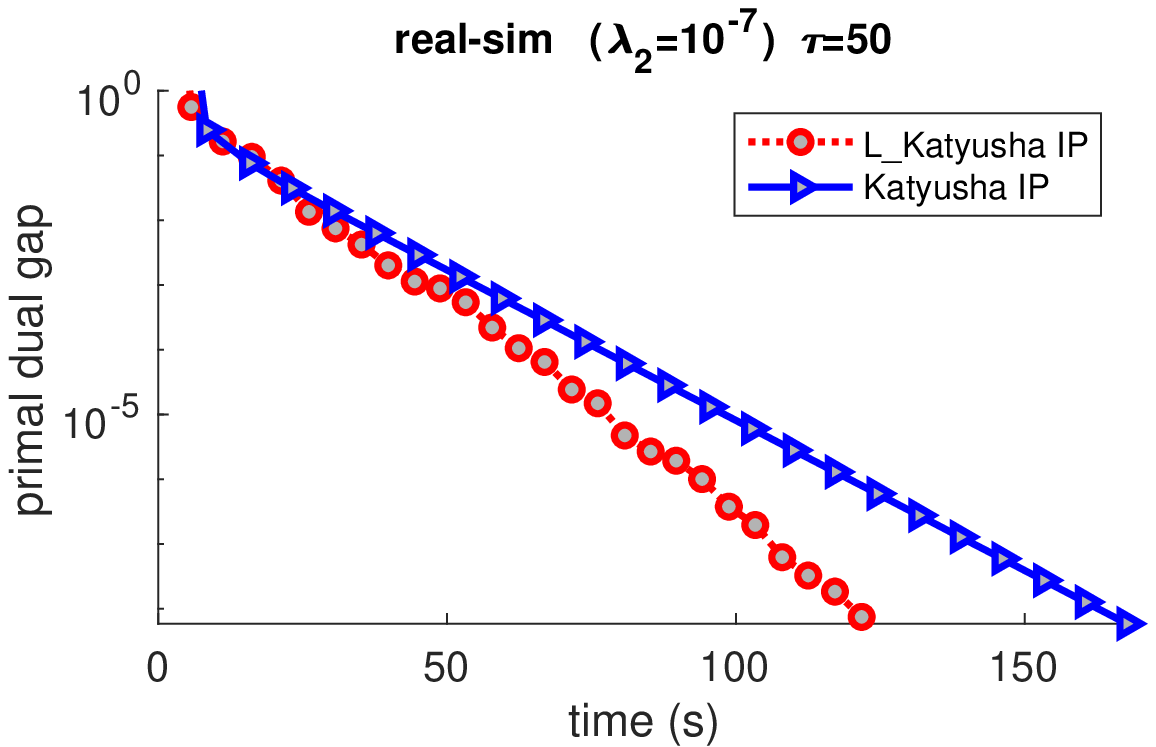}}
    \caption{L-Katyusha V.S. Katyusha, time plot, real-sim}\label{fig8}
 \end{figure}
 
 \begin{figure}[!ht]
 \subfigure[$\tau=1$]{\label{9d1}\includegraphics[scale=0.38]{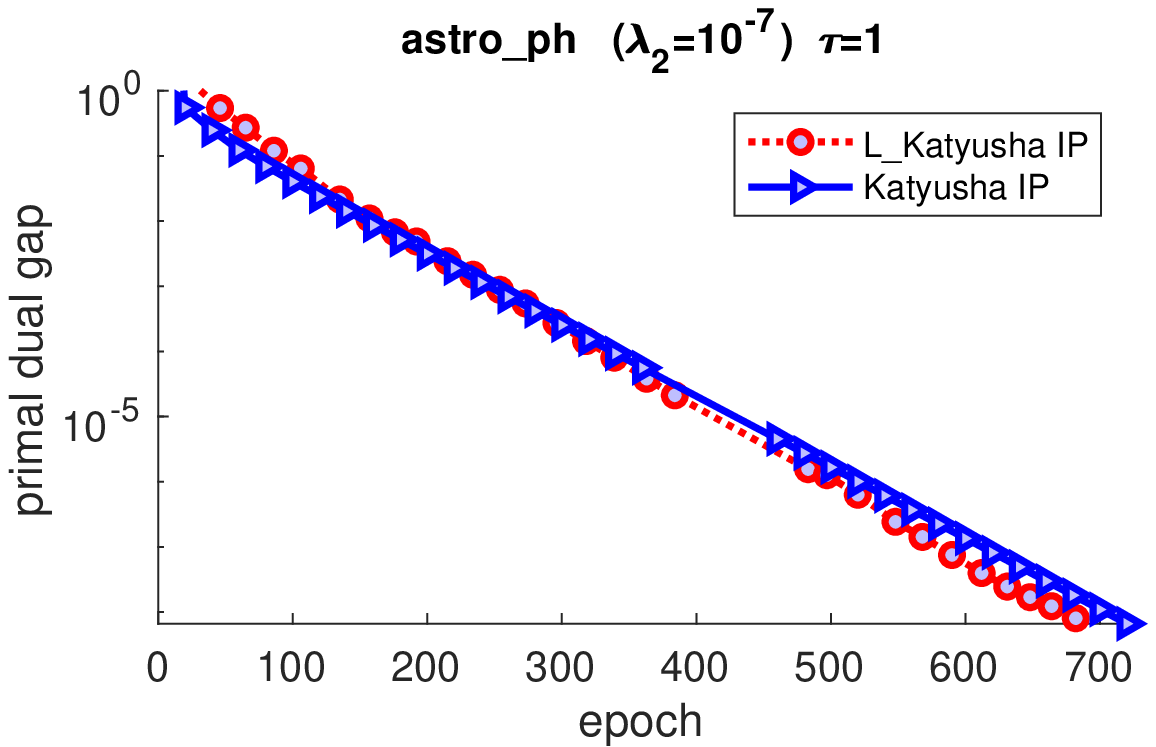}}
  \subfigure[$\tau=10$]{  \label{9d2}  \includegraphics[scale=0.38]{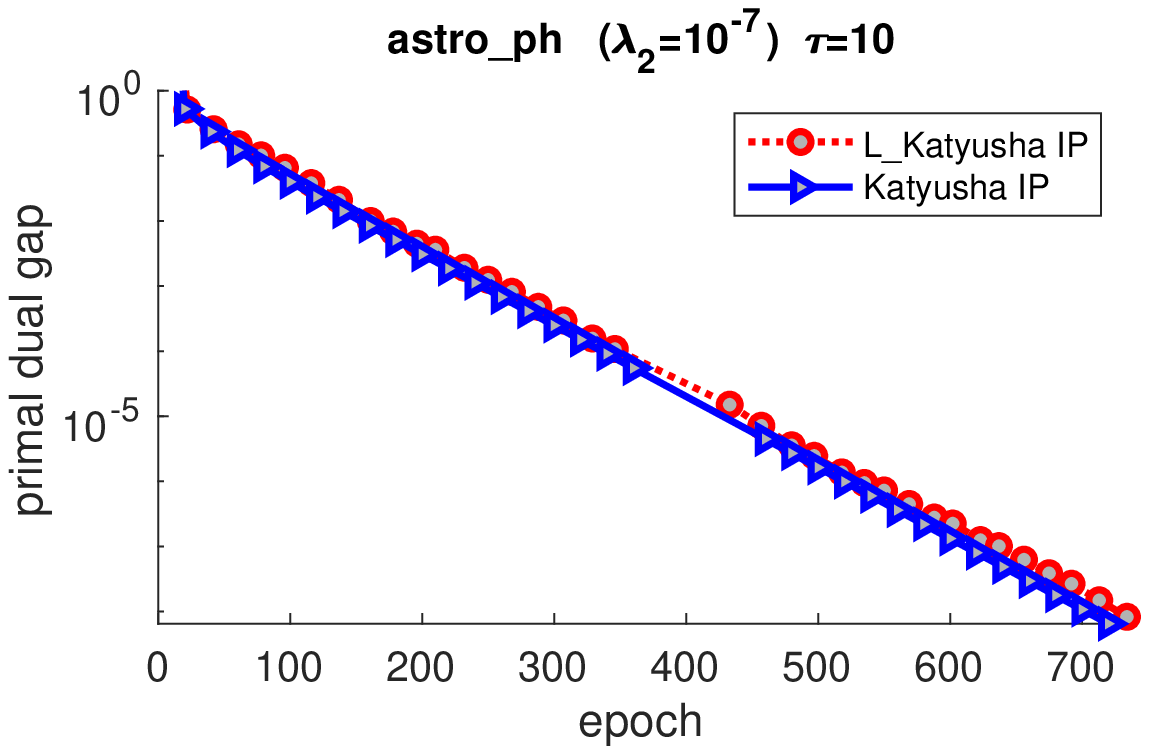}}
    \subfigure[$\tau=50$]{ \label{9d3} \includegraphics[scale=0.38]{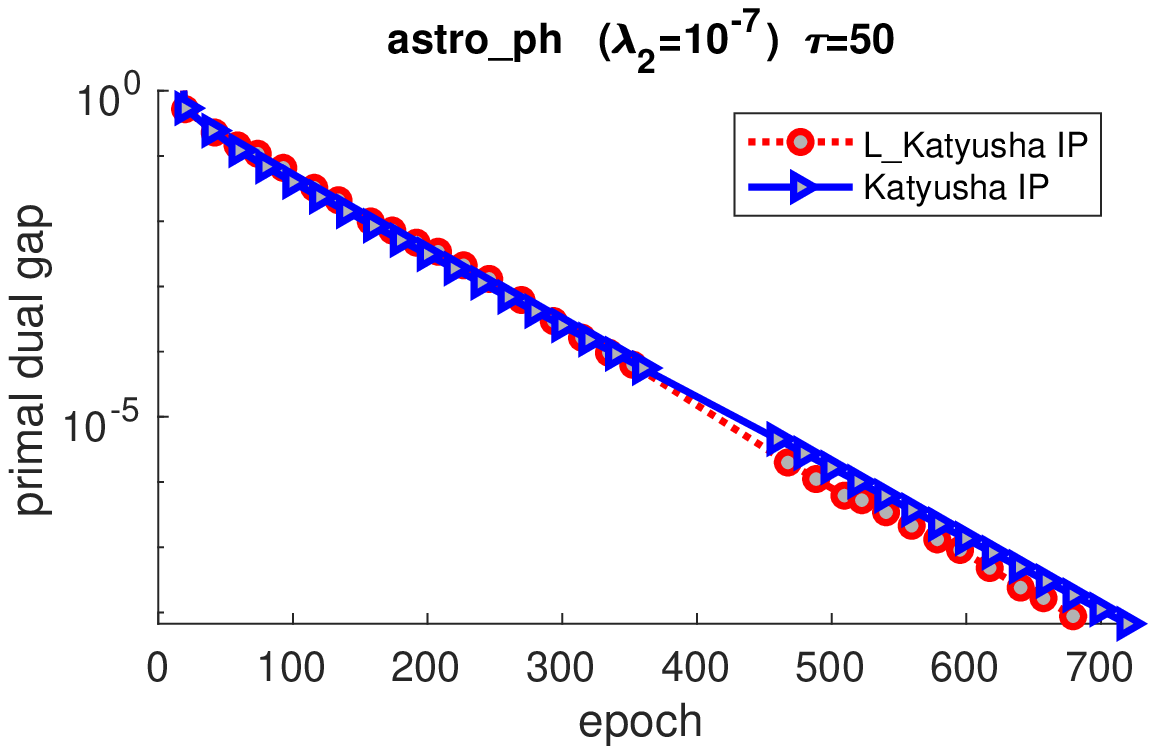}}
    \caption{L-Katyusha V.S. Katyusha, epoch plot, astro\_ph}\label{fig9}
 \end{figure}
\begin{figure}[!ht]
 \subfigure[$\tau=1$]{\label{10d1}\includegraphics[scale=0.38]{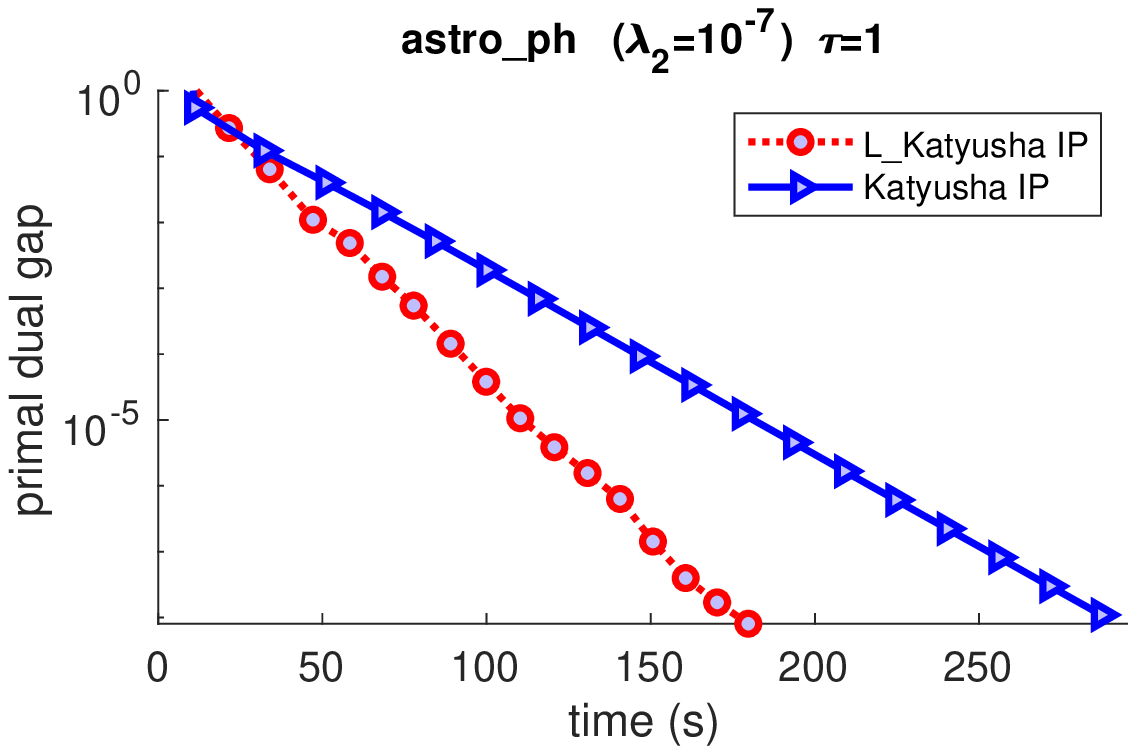}}
  \subfigure[$\tau=10$]{  \label{10d2}  \includegraphics[scale=0.38]{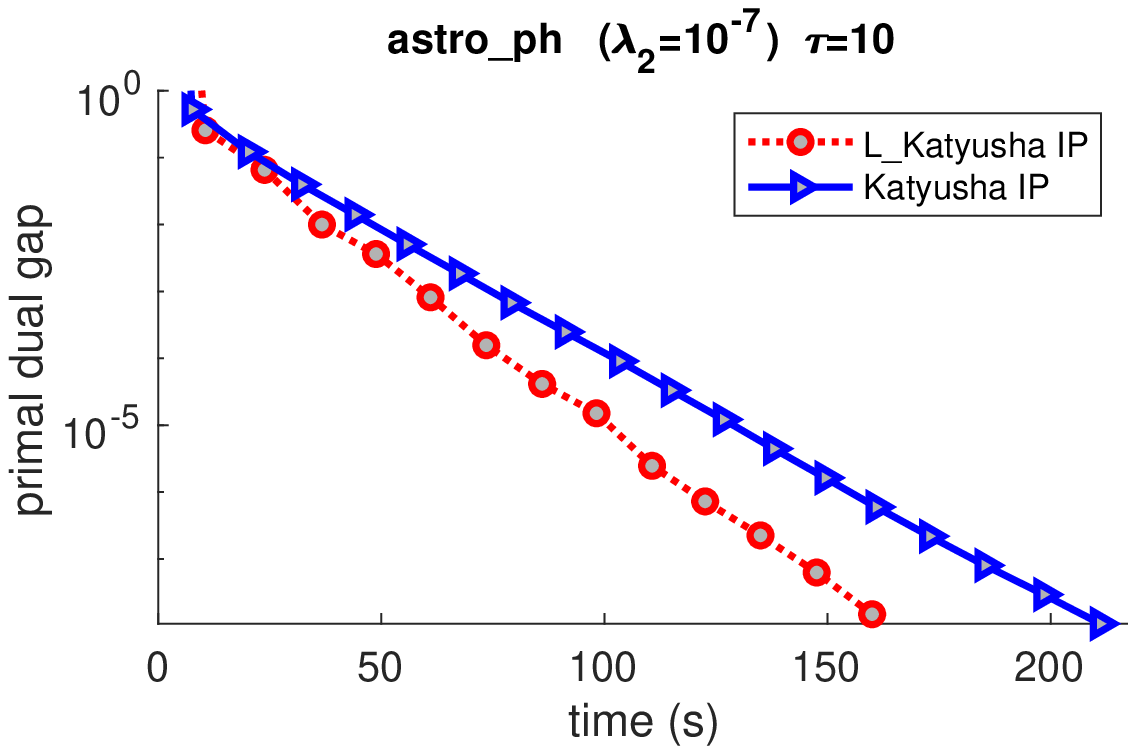}}
    \subfigure[$\tau=50$]{ \label{10d3} \includegraphics[scale=0.38]{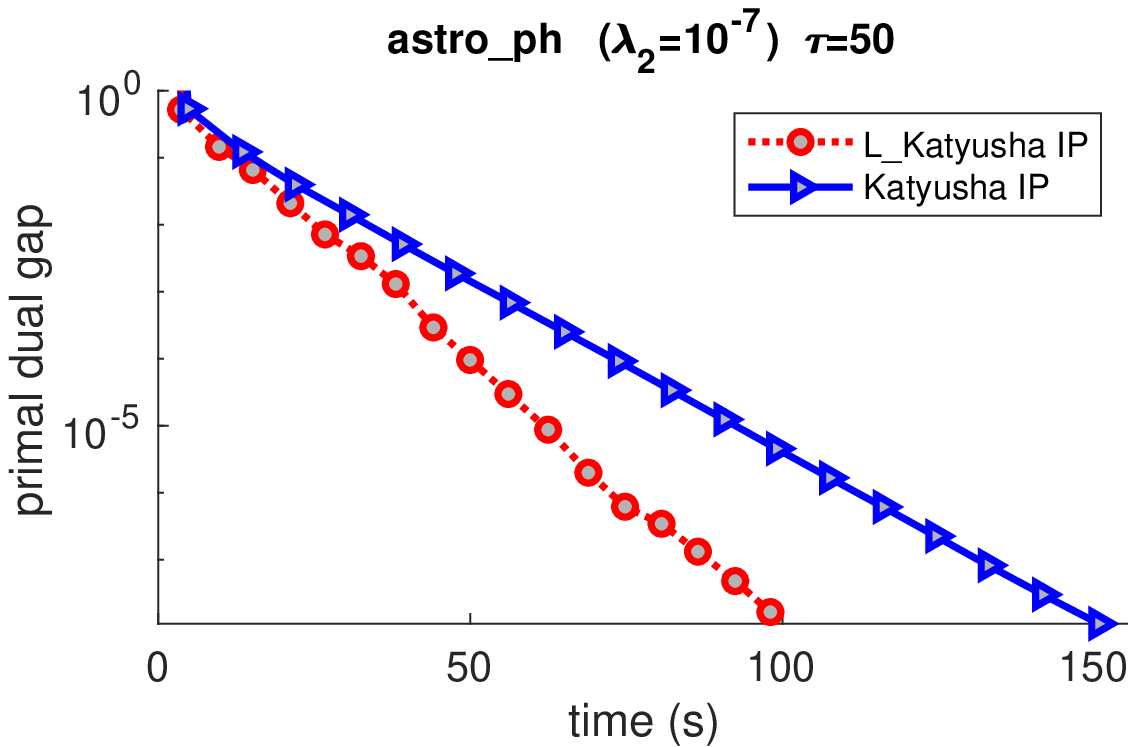}}
    \caption{L-Katyusha V.S. Katyusha, time plot, astro\_ph}\label{fig10}
 \end{figure}
 
\clearpage

\bibliographystyle{plain}
\bibliography{LSVRG_AS.bib}

\appendix
\newpage

\part*{Appendix}

\section{Proof of Lemma \ref{lm:groups}}

We construct a group sampling $S$ as follows. 
\vskip 2mm

First we distribute each index $i$ a $p_i$. Then we divide $[n]$ into several groups as follows. For the odered sequence $p_1, ..., p_n$, we add them from $p_1$ consecutively, until the summation is greater than one at $p_{i_1}$. We collect $\{ p_1, ..., p_{{i_1}-1} \}$ as a group $C_1$. In such way, $\sum_{i\in C_1}p_i$ is less than or equal to one. Next we repeat this procedure to the ordered sequence $p_{i_1}, ..., p_n$ until every index is divided into some group. The rest of the formation of the sampling is the same as the final step in the definition of group sampling. 

\vskip 2mm

Assume the number of the groups is $t$, and the groups we get from the above construction are ordered sets $C_1, ..., C_t$. According to the construction, we know 
$$
\sum_{i\in C_j} p_i + \sum_{i\in C_{j+1}}p_i > 1,
$$
for any $1\leq j < t$. 
Next, we consider two cases. 

{\bf Case 1.} Suppose $t$ is even.  Then $\frac{t}{2} < \tau$. If $\tau$ is an integer, then $t \leq 2\tau -2$, otherwise, $t< 2\tau$. \\

{\bf Case 2.} Suppose $t$ is odd. Then $\frac{t-1}{2} < \tau$. If $\tau$ is an integer, then $t \leq 2\tau -1$, otherwise, $t< 2\tau + 1$.

\section{Strongly Convex Case: Proof of Theorem \ref{Th:convp} }

\subsection{Lemmas} 

\begin{lemma}\label{lm:Dk+1}
	$$
	\mathbb{E}_k[{\cal D}^{k+1}_S] \leq (1-p)\mathbb{E}_k[{\cal D}^k_S] + \frac{8\eta^2{\cal L}_1}{1 + \eta\mu_{\psi}} (f(x^k) - f(x^*) - \langle \nabla f(x^*), x^k - x^* \rangle ).
	$$
\end{lemma}

\begin{proof}
	
	From Assumption \ref{as:expsmooth}, we have 
	\begin{eqnarray*}
		\mathbb{E}_k[{\cal D}^{k+1}_S] &=& (1-p)\mathbb{E}_k[{\cal D}^k_S] + \frac{4\eta^2}{1 + \eta\mu_{\psi}} \mathbb{E}_k\left[\left\|\frac{1}{n}({\bf G}(x^k) - {\bf G}(x^*)){\theta}_{S}{\bf I}_{S}e \right\|^2 \right]\\ 
		&\leq& (1-p)\mathbb{E}_k[{\cal D}^k_S] + \frac{8\eta^2{\cal L}_1}{1 + \eta\mu_{\psi}} (f(x^k) - f(x^*) - \langle \nabla f(x^*), x^k - x^* \rangle ). 
	\end{eqnarray*}
	
\end{proof}

\begin{lemma}\label{lm:gk2}
	For $g^k$, we have 
	$$
	\mathbb{E}_k[\|g^k - \nabla f(x^*)\|^2] \leq 4{\cal L}_1 (f(x^k) - f(x^*) - \langle \nabla f(x^*), x^k - x^* \rangle ) + \frac{(1 + \eta\mu_{\psi})p}{2\eta^2} \mathbb{E}_k[{\cal D}^k_S]. 
	$$
\end{lemma}

\begin{proof}
	
	\begin{eqnarray*}
		&& \mathbb{E}_k[\|g^k - \nabla f(x^*)\|^2] \\ 
		&=& \mathbb{E}\left[ \left\| \frac{1}{n}({\bf G}(x^k) - {\bf G}(w^k)){\theta}_{S_k}{\bf I}_{S_k}e + \frac{1}{n}{\bf G}(w^k)e - \nabla f(x^*)\right\|^2  \right] \\ 
		&\leq& 2\mathbb{E}_k\left[ \left\| \frac{1}{n}({\bf G}(x^k)-{\bf G}(x^*){\theta}_{S_k}{\bf I}_{S_k}e)\right\|^2 \right] \\ 
		&& + 2 \mathbb{E}_k\left[ \left\| \frac{1}{n}{\bf G}(w^k)e - \frac{1}{n}{\bf G}(x^*)e - \frac{1}{n}({\bf G}(w^k) - {\bf G}(x^*)){\theta}_{S}{\bf I}_{S}e\right\|^2 \right] \\ 
		&\overset{Assumption\ \ref{as:expsmooth}}{\leq}& 4{\cal L}_1  (f(x^k) - f(x^*) - \langle \nabla f(x^*), x^k - x^* \rangle )  \\ 
		&& + 2\mathbb{E}_k\left[ \left\| \frac{1}{n}({\bf G}(w^k) - {\bf G}(x^*)){\theta}_{S}{\bf I}_{S}e \right\|^2 \right] \\ 
		&=& 4{\cal L}_1  (f(x^k) - f(x^*) - \langle \nabla f(x^*), x^k - x^* \rangle ) + \frac{(1 + \eta\mu_{\psi})p}{2\eta^2} \mathbb{E}_k[{\cal D}^k_S].
	\end{eqnarray*}
	
\end{proof}

\subsection{Proof of Theorem \ref{Th:convp} }

	Since $x^*$ is the solution of problem (\ref{primal}), we have 
	$$
	x^* = \prox_{\eta} (x^* - \eta \nabla f(x^*)). 
	$$
	
	Then, 
	\begin{eqnarray*}
		\mathbb{E}_k[\| x^{k+1} - x^* \|^2 ] &=& \mathbb{E}_k[\| \prox_{\eta}(x^k - \eta g^k) - \prox_{\eta} (x^* - \eta \nabla f(x^*))\|^2 ] \\
		&\leq& \frac{1}{1+\eta\mu_{\psi}}\mathbb{E}_k[ \|x^k - \eta g^k - (x^* - \eta \nabla f(x^*)) \|^2 ] \\ 
		&=&\frac{1}{1 + \eta\mu_{\psi}} \left( \|x^k - x^* \|^2 - 2\eta \langle \nabla f(x^k) - \nabla f(x^*), x^k -x^* \rangle \right) \\ 
		&& + \frac{\eta^2}{1 + \eta\mu_{\psi}} \mathbb{E}_k[ \| g^k - \nabla f(x^*) \|^2 ]  \\ 
		&\leq& \frac{(1 - \eta \mu_f)}{1 + \eta\mu_{\psi}} \|x^k - x^* \|^2 - \frac{2\eta}{1 + \eta\mu_{\psi}} (f(x^k) - f(x^*) - \langle \nabla f(x^*), x^k - x^* \rangle ) \\ 
		&& +  \frac{\eta^2}{1 + \eta\mu_{\psi}} \mathbb{E}_k[ \| g^k - \nabla f(x^*) \|^2 ]. 
	\end{eqnarray*}
	
	Hence, 
	\begin{eqnarray*}
		&& \mathbb{E}_k [\|\|x^{k+1} -x^*\|^2 + {\cal D}^{k+1}_S \|^2 ] \\  &\overset{Lemma ~\ref{lm:Dk+1}}{\leq}& \frac{(1 - \eta \mu_f)}{1 + \eta\mu_{\psi}} \|x^k - x^* \|^2 + (1-p)\mathbb{E}_k[{\cal D}^k_S] +  \frac{\eta^2}{1 + \eta\mu_{\psi}} \mathbb{E}_k[ \| g^k - \nabla f(x^*) \|^2 ] \\ 
		&&  -  \frac{2\eta}{1 + \eta\mu_{\psi}} (1 - 4\eta {\cal L}_1) (f(x^k) - f(x^*) - \langle \nabla f(x^*), x^k - x^* \rangle ) \\ 
		&\overset{Lemma ~\ref{lm:gk2}}{\leq}& \frac{(1 - \eta \mu_f)}{1 + \eta\mu_{\psi}}\|x^k - x^* \|^2 + (1-\frac{p}{2})\mathbb{E}_k[{\cal D}^k_S] \\ 
		&& -  \frac{2\eta}{1 + \eta\mu_{\psi}} (1 - 6\eta {\cal L}_1) (f(x^k) - f(x^*) - \langle \nabla f(x^*), x^k - x^* \rangle ). 
	\end{eqnarray*}
	
	Now if the step size $\eta \leq \frac{1}{6{\cal L}_1}$, then by $\mu = \mu_f + \mu_{\psi}$,  we obtain the desired inequality: 
	$$
	\mathbb{E}_k\left[ \|x^{k+1} - x^*\|^2 + {\cal D}^{k+1}_S\|^2 \right]  \leq (1-\frac{\eta \mu}{1 + \eta\mu_{\psi}}) \|x^k-x^*\|^2 + (1-\frac{p}{2})\mathbb{E}_k \left[{\cal D}^k_S \right].
	$$
	Therefore, if $\eta = \frac{1}{6{\cal L}_1}$, then in order to gurantee $\mathbb{E}[{\Psi}^k_S] \leq \epsilon \cdot \mathbb{E}[{\Psi}^0_S]$, we only need to let 
	$$
	k\geq {\cal O}\left( \left( \frac{1}{p} + \frac{{\cal L}_1}{\mu}\right) \log\frac{1}{\epsilon}\right) .
	$$

\section{Strongly Convex Case: Proof of Theorem \ref{Th:LKa} }

\subsection{Lemmas}

From Assumption \ref{as:expL2}, we have the following lemma.
\begin{lemma}\label{lm:expL2}
	We have 
	\begin{equation}\label{eq:expL2}
	\mathbb{E}_k [ \|g^k - \nabla f(x^k)\|^2 ] \leq 2{\cal L}_2 (f(w^k) - f(x^k) - \langle \nabla f(x^k), w^k-x^k \rangle ). 
	\end{equation}
\end{lemma}

\begin{lemma}\label{lm:couple1}
	We have 
	\begin{equation}\label{eq:couple1}
	\langle g^k, x^*-z^{k+1} \rangle + \frac{\mu_f}{2}\|x^k -x^*\|^2 \geq \frac{L}{2\eta}\|z^k - z^{k+1}\|^2 + {\cal Z}^{k+1} - \frac{1}{1+\eta\sigma}{\cal Z}^k + \psi(z^{k+1}) - \psi(x^*). 
	\end{equation}
\end{lemma}

\begin{proof}
	
	Since 
	$$
	z^{k+1} = \prox_{\frac{\eta}{(1+ \eta\sigma_1)L}} \left( \frac{1}{1+ \eta \sigma_1}(\eta \sigma_1 x^k + z^k - \frac{\eta}{L}g^k )  \right), 
	$$
	we have 
	$$
	z^{k+1} -  \frac{1}{1+ \eta \sigma_1}(\eta \sigma_1 x^k + z^k - \frac{\eta}{L}g^k )  + \frac{\eta}{(1+\eta \sigma_1)L} g = 0,
	$$
	where $g$ is some subgradient of $\psi(z)$ at $z^{k+1}$. This along with $\mu_f = L\sigma_1$ implies that 
	\begin{equation}\label{eq:psig}
	g^k = \frac{L}{\eta} (z^k - z^{k+1}) + \mu_f (x^k - z^{k+1}) - g. 
	\end{equation}
	
	Therefore, we have 
	\begin{eqnarray*}
		\langle g^k, z^{k+1} - x^* \rangle &=& \mu_f \langle z^{k+1} - x^*, x^k -z^{k+1} \rangle + \frac{L}{\eta} \langle z^{k+1}-x^*, z^k - z^{k+1} \rangle - \langle z^{k+1} - x^*, g \rangle \\
		&=& \frac{\mu_f}{2} \left( \|x^k - x^*\|^2 - \|x^k - z^{k+1}\|^2 - \|z^{k+1} -x^*\|^2  \right) \\ 
		&& + \frac{L}{2\eta} \left(  \|z^k -x^*\|^2 - \|z^k - z^{k+1}\|^2 - \|z^{k+1} - x^*\|^2  \right) - \langle z^{k+1} - x^*, g \rangle \\
		&\leq& \frac{\mu_f}{2} \|x^k - x^*\|^2 + \frac{L}{2\eta} \left(  \|z^k - x^*\|^2 - (1+\eta \sigma_1)\|z^{k+1} -x^*\|^2  \right) \\ 
		&& - \frac{L}{2\eta} \|z^k - z^{k+1}\|^2 + \psi(x^*) - \psi(z^{k+1}) - \frac{\mu_{\psi}}{2} \|z^{k+1} - x^*\|^2 \\ 
		&=& \frac{\mu_f}{2} \|x^k - x^*\|^2 + \frac{L}{2\eta} \left(  \|z^k - x^*\|^2 - (1+\eta \sigma)\|z^{k+1} -x^*\|^2  \right) \\ 
		&& - \frac{L}{2\eta} \|z^k - z^{k+1}\|^2 + \psi(x^*) - \psi(z^{k+1})
	\end{eqnarray*}
	where the last inequality comes from $-\|x^k - z^{k+1}\|^2 \leq 0$ and $\psi$ is $\mu_{\psi}$-strongly convex, the last equality comes from $\sigma = \sigma_1+\sigma_2$. By the definition of ${\cal Z}^k$, we can obtain the result.
	
\end{proof}

\begin{lemma}\label{lm:couple2}
	We have 
	\begin{equation}\label{eq:couple2}
	\frac{1}{\theta_1}(f(y^{k+1}) - f(x^k)) - \frac{1}{4L\theta_1}\|g^k - \nabla f(x^k)\|^2 \leq \frac{L}{2\eta}\|z^{k+1} - z^k\|^2 + \langle g^k, z^{k+1} -z^k \rangle. 
	\end{equation}
\end{lemma}

\begin{proof}
	Since $L = \max\{ {\cal L}_2, L_f  \} \geq L_f$, the proof is the same as that of Lemma 5.3 in \cite{LSVRG}. In fact, by choosing $\theta_2 = \frac{1}{2}$ in \cite{LSVRG}, $\eta = \frac{1}{3\theta_1}$, which is the same as our $\eta$, and then, the result holds straightforward from Lemma 5.3 in \cite{LSVRG}. 
\end{proof}

\begin{lemma}\label{lm:wk}
	We have 
	\begin{equation}\label{eq:wk}
	\mathbb{E}_k[{\cal W}^{k+1}] = (1-p){\cal W}^k + \frac{\theta_2}{q}{\cal Y}^k
	\end{equation}
\end{lemma}

\begin{proof}
	
	\begin{eqnarray*}
		\mathbb{E}_k[{\cal W}^{k+1}] &=& \frac{\theta_2}{pq\theta_1} \mathbb{E}_k(P(w^{k+1}) - P(x^*)) \\ 
		&=&  \frac{\theta_2}{pq\theta_1} \left((1-p)P(w^k) + pP(y^k) - f(x^*) \right) \\ 
		&=& (1-p){\cal W}^k + \frac{\theta_2}{q}{\cal Y}^k.
	\end{eqnarray*}
\end{proof}

\subsection{Proof of Theorem \ref{Th:LKa} }

	First, we have 
	\begin{eqnarray*}
		f(x^*) &\overset{\textcircled{1}}{\geq}& f(x^k) + \langle \nabla f(x^k), x^* - x^k \rangle + \frac{\mu_f}{2}\|x^k - x^*\|^2 \\ 
		&=& f(x^k) + \frac{\mu_f}{2}\|x^k-x^*\|^2 + \langle \nabla f(x^k), x^*-z^k + z^k -x^k \rangle \\ 
		&=& f(x^k) + \frac{\mu_f}{2}\|x^k -x^*\|^2 + \langle \nabla f(x^k), x^*-z^k \rangle + \frac{\theta_2}{\theta_1}\langle \nabla f(x^k), x^k-w^k \rangle \\ 
		&& + \frac{1-\theta_1-\theta_2}{\theta_1} \langle \nabla f(x^k), x^k-y^k \rangle \\ 
		&\overset{\textcircled{2}}{\geq}& f(x^k) + \frac{\theta_2}{\theta_1}\langle \nabla f(x^k), x^k-w^k \rangle + \frac{1-\theta_1-\theta_2}{\theta_1} (f(x^k) - f(y^k)) \\ 
		&& + \mathbb{E}_k \left[ \frac{\mu_f }{2}\|x^k - x^*\|^2 + \langle g^k, x^* - z^{k+1} \rangle + \langle g^k, z^{k+1}-z^k \rangle \right].
	\end{eqnarray*}
	Above, inequality $\textcircled{1}$ uses $\mu_f$-strong convexity of $f$, and inequality $\textcircled{2}$ uses the convexity of $f$ and $\mathbb{E}_k[g^k] = \nabla f(x^k)$. For the last term in the above inequality, we have 
	\begin{eqnarray*}
		&& \mathbb{E}_k \left[ \frac{\mu_f}{2}\|x^k - x^*\|^2 + \langle g^k, x^* - z^{k+1} \rangle + \langle g^k, z^{k+1}-z^k \right] \\  &\overset{(\ref{eq:couple1})}{\geq}& \mathbb{E}_k[{\cal Z}^{k+1}] - \frac{{\cal Z}^k}{1 + \eta\sigma} + \mathbb{E}_k\left[ \langle g^k, z^{k+1} -z^k\rangle + \frac{L}{2\eta}\|z^k - z^{k+1}\|^2 \right] + \psi(z^{k+1}) - \psi(x^*) \\ 
		&\overset{(\ref{eq:couple2})}{\geq}& \mathbb{E}_k[{\cal Z}^{k+1}] - \frac{{\cal Z}^k}{1 + \eta\sigma} + \mathbb{E}_k \left[ \frac{1}{\theta_1}(f(y^{k+1}) - f(x^k)) - \frac{1}{4L\theta_1}\|g^k - \nabla f(x^k)\|^2 \right] + \psi(z^{k+1}) - \psi(x^*) \\ 
		&\overset{(\ref{eq:expL2})}{\geq}& \mathbb{E}_k[{\cal Z}^{k+1}] - \frac{{\cal Z}^k}{1 + \eta\sigma} + \psi(z^{k+1}) - \psi(x^*) \\ 
		&& + \mathbb{E}_k \left[ \frac{1}{\theta_1}(f(y^{k+1}) - f(x^k)) - \frac{{\cal L}_2}{2L\theta_1}(f(w^k) - f(x^k) - \langle \nabla f(x^k), w^k-x^k \rangle ) \right] \\ 
		&=& \mathbb{E}_k[{\cal Z}^{k+1}] - \frac{{\cal Z}^k}{1 + \eta\sigma} + \psi(z^{k+1}) - \psi(x^*) \\ 
		&& + \mathbb{E}_k \left[ \frac{1}{\theta_1}(f(y^{k+1}) - f(x^k)) - \frac{\theta_2}{\theta_1}(f(w^k) - f(x^k) - \langle \nabla f(x^k), w^k-x^k \rangle ) \right]. 
	\end{eqnarray*}
	Therefore, 
	\begin{eqnarray*}
		f(x^*) &\geq& f(x^k) + \frac{1-\theta_1-\theta_2}{\theta_1} (f(x^k) - f(y^k)) + \mathbb{E}_k[{\cal Z}^{k+1}] - \frac{{\cal Z}^k}{1 + \eta\sigma}  \\ 
		&& + \mathbb{E}_k \left[ \frac{1}{\theta_1}(f(y^{k+1}) - f(x^k)) \right] - \frac{\theta_2}{\theta_1}(f(w^k) - f(x^k) ) + \psi(z^{k+1}) - \psi(x^*) \\ 
		&=& \mathbb{E}_k[{\cal Z}^{k+1}] - \frac{{\cal Z}^k}{1 + \eta\sigma} - \frac{1-\theta_1-\theta_2}{\theta_1}f(y^k) + \frac{1}{\theta_1}\mathbb{E}_k[f(y^{k+1})] - \frac{\theta_2}{\theta_1} f(w^k) + \psi(z^{k+1}) - \psi(x^*). 
	\end{eqnarray*}
	
	Moreover, since $\psi$ is convex, and 
	$$
	y^{k+1} = x^k + \theta_1(z^{k+1} -z^k) = \theta_1z^{k+1} + \theta_2w^k + (1-\theta_1 -\theta_2)y^k, 
	$$
	we have 
	$$
	\psi(z^{k+1}) \geq \frac{1}{\theta_1}\psi(y^{k+1}) - \frac{\theta_2}{\theta_1}\psi(w^k) - \frac{1-\theta_1 -\theta_2}{\theta_1} \psi(y^k). 
	$$
	
	Hence, we arrive at  
	$$
	f(x^*) \geq \mathbb{E}_k[{\cal Z}^{k+1}] - \frac{{\cal Z}^k}{1 + \eta\sigma} - \frac{1-\theta_1-\theta_2}{\theta_1}P(y^k) + \frac{1}{\theta_1}\mathbb{E}_k[P(y^{k+1})] - \frac{\theta_2}{\theta_1} P(w^k) - \psi(x^*). 
	$$

	After rearranging we get 
	$$
	\mathbb{E}_k[{\cal Z}^{k+1} + {\cal Y}^{k+1}] \leq \frac{{\cal Z}^k}{1 + \eta\sigma} + (1-\theta_1-\theta_2){\cal Y}^k + pq{\cal W}^k. 
	$$
	From Lemma \ref{lm:wk}, we have 
	\begin{eqnarray*}
		\mathbb{E}_k[{\cal Z}^{k+1} + {\cal Y}^{k+1} + {\cal W}^{k+1}] &\leq& \frac{{\cal Z}^k}{1 + \eta\sigma} + (1-\theta_1-\theta_2){\cal Y}^k + pq{\cal W}^k + (1-p){\cal W}^k + \frac{\theta_2}{q}{\cal Y}^k \\
		&=& \frac{1}{1+\eta\sigma}{\cal Z}^k + (1-(\theta_1 + \theta_2 - \frac{\theta_2}{q})){\cal Y}^k + (1 - p(1-q)){\cal W}^k.
	\end{eqnarray*}
	
	Hence, we know $\mathbb{E}[\Psi^k] \leq \epsilon \Psi^0$ as long as 
	\begin{eqnarray*}
		k& \geq& \max\left\{ 1 + \frac{1}{\eta\sigma}, \frac{1}{\theta_1 + \theta_2 - \frac{\theta_2}{q}}, \frac{1}{p(1-q)}  \right\} \log\frac{1}{\epsilon} \\
		&=&  \max\left\{ 1 + \frac{3\theta_1L}{\mu}, \frac{1}{\theta_1 + \theta_2 - \frac{\theta_2}{q}}, \frac{1}{p(1-q)}  \right\} \log\frac{1}{\epsilon}. 
	\end{eqnarray*}

	{\bf Case 1.} Suppose $L_f \leq \frac{{\cal L}_2}{p}$. In this case, $\theta_1 = \min\{ \sqrt{\frac{\mu}{{\cal L}_2p}}\theta_2 , \theta_2 \}$.  Recall that $\theta_2 = \frac{{\cal L}_2}{2L}$, which implies $L = \frac{{\cal L}_2}{2\theta_2}$. Furthermore, 
	\begin{equation}\label{eq:theta2}
	\frac{2}{\theta_2} = \frac{4L}{{\cal L}_2} = 4\frac{\max\{ {\cal L}_2, L_f \}}{{\cal L}_2} \leq \frac{4}{p}. 
	\end{equation}
	
	\vskip 2mm
	
	{\bf Case 1.1.} Suppose $\sqrt{\frac{\mu}{{\cal L}_2p}}  \geq 1$. In this subcase, $\theta_1 = \theta_2$ and $\frac{1}{p} \geq \frac{{\cal L}_2}{\mu}$. By choosing $q = \frac{2}{3}$, we have 
	$$
	\max\left\{ 1 + \frac{3\theta_1L}{\mu}, \frac{1}{\theta_1 + \theta_2 - \frac{\theta_2}{q}}, \frac{1}{p(1-q)}  \right\} \leq \max\left\{ 1+\frac{3{\cal L}_2}{2\mu}, \frac{2}{\theta_2}, \frac{3}{p}  \right\} \overset{(\ref{eq:theta2})}{\leq} \frac{4}{p}.
	$$
	
	{\bf Case 1.2.} Suppose $\sqrt{\frac{\mu}{{\cal L}_2p}}  < 1$. In this subcase, $\theta_1 = \sqrt{\frac{\mu}{{\cal L}_2p}}\theta_2$ and $\frac{1}{p} < \frac{{\cal L}_2}{\mu}$. By choosing $q = 1 - \frac{1}{3}\sqrt{\frac{\mu}{{\cal L}_2p}} \geq \frac{2}{3}$, we have 
	$$
	\frac{1}{\theta_1 + \theta_2 - \frac{\theta_2}{q}}  = \frac{1}{\theta_1 - \theta_2\frac{1-q}{q}} = \frac{1}{\theta_1(1-\frac{1}{3q})} \leq \frac{2}{\theta_2} \overset{(\ref{eq:theta2})}{\leq} \frac{4}{p} \leq 4\sqrt{\frac{{\cal L}_2}{\mu p}}. 
	$$
	Hence 
	$$
	\max\left\{ 1 + \frac{3\theta_1L}{\mu}, \frac{1}{\theta_1 + \theta_2 - \frac{\theta_2}{q}}, \frac{1}{p(1-q)}  \right\} \leq \max\left\{  1+ \frac{3}{2}\sqrt{\frac{{\cal L}_2}{\mu p}}, 4\sqrt{\frac{{\cal L}_2}{\mu p}}, 3\sqrt{\frac{{\cal L}_2}{\mu p}}  \right\} \leq 4\sqrt{\frac{{\cal L}_2}{\mu p}}. 
	$$
	
	{\bf Case 2.} Suppose $L_f > \frac{{\cal L}_2}{p}$. In this case, $\theta_1 = \min\{ \sqrt{\frac{\mu}{L_f}}, \frac{p}{2} \}$, $L = L_f$, and $\theta_2 = \frac{{\cal L}_2}{2L_f} < \frac{p}{2}$. 
	
	\vskip 2mm
	
	{\bf Case 2.1.} Suppose $\sqrt{\frac{\mu}{L_f}} \geq \frac{p}{2}$. In this subcase, $\theta_1 = \frac{p}{2}$. Let $q = \frac{2}{3}$. Then 
	$$
	\frac{1}{\theta_1 + \theta_2 - \frac{\theta_2}{q}}  = \frac{1}{\theta_1 - \theta_2/2} < \frac{4}{p}, 
	$$
	and 
	$$
	\frac{3\theta_1L}{\mu} = \frac{3pL_f}{2\mu} \leq 3\sqrt{\frac{L_f}{\mu}} \leq \frac{6}{p}. 
	$$
	Hence 
	$$
	\max\left\{ 1 + \frac{3\theta_1L}{\mu}, \frac{1}{\theta_1 + \theta_2 - \frac{\theta_2}{q}}, \frac{1}{p(1-q)}  \right\} \leq \max\left\{ 1+ \frac{6}{p}, \frac{4}{p}, \frac{3}{p}  \right\} < \frac{7}{p}. 
	$$
	
	{\bf Case 2.2.} Suppose $\sqrt{\frac{\mu}{L_f}} < \frac{p}{2}$. In this subcase, $\theta_1 = \sqrt{\frac{\mu}{L_f}} $.  Let $q = 1 - \frac{2}{3p}\sqrt{\frac{\mu}{L_f}} > \frac{2}{3}$. Then  
	$$
	\frac{1}{\theta_1 + \theta_2 - \frac{\theta_2}{q}} = \frac{1}{\theta_1 - \theta_2\frac{1-q}{q}} = \frac{1}{\theta_1 - \theta_2\frac{2}{3qp}\sqrt{\frac{\mu}{L_f}}} < \frac{1}{\theta_1 - \frac{\theta_2}{p}\sqrt{\frac{\mu}{L_f}}} < 2\sqrt{\frac{L_f}{\mu}}. 
	$$
	Hence 
	$$
	\max\left\{ 1 + \frac{3\theta_1L}{\mu}, \frac{1}{\theta_1 + \theta_2 - \frac{\theta_2}{q}}, \frac{1}{p(1-q)}  \right\} \leq \max\left\{ 1+ 3\sqrt{\frac{L_f}{\mu}}, 2\sqrt{\frac{L_f}{\mu}}, \frac{3}{2}\sqrt{\frac{L_f}{\mu}}  \right\} \leq 4\sqrt{\frac{L_f}{\mu}}. 
	$$

\section{Non-Strongly Convex Case: proof of Theorem \ref{Th:gconvex1} }

\subsection{Lemmas} 

\begin{lemma}\label{lm:gkgconvex}
	We have 
	$$
	\mathbb{E}_k[\| g^k - \nabla f(x^k)\|^2] \leq 4{\cal L}_1 (f(x^k) - f(x^*) - \langle \nabla f(x^*), x^k - x^*\rangle ) + 2\mathbb{E}_k [{\cal H}^k_S]. 
	$$
\end{lemma}

\begin{proof}
	
	\begin{eqnarray*}
		&& \mathbb{E}_k[\|g^k - \nabla f(x^k)\|^2] \\ 
		&=& \mathbb{E}\left[ \left\| \frac{1}{n}({\bf G}(x^k) - {\bf G}(w^k)){\theta}_{S_k}{\bf I}_{S_k}e + \frac{1}{n}{\bf G}(w^k)e - \frac{1}{n}{\bf G}(x^k)e\right\|^2  \right] \\ 
		&\leq& 2\mathbb{E}_k\left[ \left\| \frac{1}{n}({\bf G}(x^k)-{\bf G}(x^*){\theta}_{S_k}{\bf I}_{S_k}e) - \frac{1}{n}{\bf G}(x^k)e + \frac{1}{n}{\bf G}(x^*)e  \right\|^2 \right] \\ 
		&& + 2 \mathbb{E}_k\left[ \left\| \frac{1}{n}{\bf G}(w^k)e - \frac{1}{n}{\bf G}(x^*)e - \frac{1}{n}({\bf G}(w^k) - {\bf G}(x^*)){\theta}_{S}{\bf I}_{S}e\right\|^2 \right] \\ 
		&\leq& 2\mathbb{E}_k\left[ \left\| \frac{1}{n}({\bf G}(x^k)-{\bf G}(x^*){\theta}_{S_k}{\bf I}_{S_k}e)\right\|^2 \right] \\ 
		&& + 2 \mathbb{E}_k\left[ \left\| \frac{1}{n}{\bf G}(w^k)e - \frac{1}{n}{\bf G}(x^*)e - \frac{1}{n}({\bf G}(w^k) - {\bf G}(x^*)){\theta}_{S}{\bf I}_{S}e\right\|^2 \right] \\ 
		&\overset{Assumption\ \ref{as:expsmooth}}{\leq}& 4{\cal L}_1  (f(x^k) - f(x^*) - \langle \nabla f(x^*), x^k - x^* \rangle )  \\ 
		&& + 2\mathbb{E}_k\left[ \left\| \frac{1}{n}({\bf G}(w^k) - {\bf G}(x^*)){\theta}_{S}{\bf I}_{S}e \right\|^2 \right] \\ 
		&=& 4{\cal L}_1  (f(x^k) - f(x^*) - \langle \nabla f(x^*), x^k - x^* \rangle ) + 2 \mathbb{E}_k[{\cal H}^k_S].
	\end{eqnarray*}
	
\end{proof}

\begin{lemma}\label{lm:Hk+1gconvex}
	We have 
	$$
	\mathbb{E}_k \left[{\cal H}^{k+1}_S\right] \leq (1-p) \mathbb{E}_k [{{\cal H}^k_S}]+ 2p{\cal L}_1 \left( f(x^k)- f(x^*) - \langle \nabla f(x^*), x-x^* \rangle \right). 
	$$
\end{lemma}

\begin{proof}
	
	The proof is similar to Lemma \ref{lm:Dk+1}. 
	
\end{proof}

\begin{lemma}\label{lm:gkxk+1gconvex}
	
	We have 
	$$
	\langle g^k, x^* - x^{k+1} \rangle \geq \psi(x^{k+1}) - \psi(x^*) + \frac{1}{2\eta} \|x^k-x^{k+1}\|^2 + \frac{1}{2\eta}\|x^{k+1}-x^*\|^2 - \frac{1}{2\eta} \|x^k - x^*\|^2. 
	$$
\end{lemma}

\begin{proof}
	
	This is by Lemma 2.5 in \cite{Katyusha}. 
\end{proof}

\subsection{Proof of Theorem \ref{Th:gconvex1} }

Since $f$ is convex and $\mathbb{E}_k[g^k] = \nabla f(x^k)$, we have 
\begin{eqnarray*}
	f(x^*) &\geq& f(x^k) + \langle \nabla f(x^k), x^* - x^k\rangle \\ 
	&=& f(x^k) + \mathbb{E}_k[ \langle g^k, x^* - x^{k+1} + x^{k+1} -x^k \rangle ] \\ 
	&=& f(x^k) + \mathbb{E}_k [\langle g^k, x^* - x^{k+1} \rangle ] + \mathbb{E}_k[\langle g^k - \nabla f(x^k), x^{k+1}-x^k \rangle] \\ 
	&& + \mathbb{E}_k[\langle \nabla f(x^k), x^{k+1}-x^k \rangle] \\ 
	&\geq& \mathbb{E}_k[f(x^{k+1})] - \frac{L_f}{2}\mathbb{E}_k[\|x^{k+1}-x^k\|^2] + \mathbb{E}_k [\langle g^k, x^* - x^{k+1} \rangle ]  \\ 
	&& + \mathbb{E}_k[\langle g^k - \nabla f(x^k), x^{k+1}-x^k \rangle]   \\ 
	&\geq& \mathbb{E}_k[f(x^{k+1})] - \frac{L_f}{2}\mathbb{E}_k[\|x^{k+1}-x^k\|^2] + \mathbb{E}_k [\langle g^k, x^* - x^{k+1} \rangle ]  \\ 
	&& -\frac{1}{2\beta} \mathbb{E}_k [\|g^k - \nabla f(x^k)\|^2] - \frac{\beta}{2} \mathbb{E}_k[x^{k+1} - x^k\|^2], 
\end{eqnarray*}
where the second inequality comes from that $f$ is $L_f$-smooth and the third inequality comes from Young's inequality. Moreover, from Lemmas  \ref{lm:gkgconvex}  and \ref{lm:gkxk+1gconvex}, we can obtain 
\begin{eqnarray*}
	P(x^*) &\geq& \mathbb{E}_k[P(x^{k+1})] + \left(  \frac{1}{2\eta} - \frac{\beta}{2} - \frac{L_f}{2}  \right) \mathbb{E}_k[\|x^{k+1} - x^k\|^2] + \frac{1}{2\eta} \mathbb{E}_k[\|x^{k+1} - x^*\|^2] \\ 
	&& - \frac{1}{2\eta} \mathbb{E}_k[\|x^k - x^*\|^2] - \frac{2{\cal L}_1}{\beta}(f(x^k) - f(x^*) - \langle \nabla f(x^*), x^k - x^*\rangle ) - \frac{1}{\beta} \mathbb{E}_k[ {\cal H}^k_S ]. 
\end{eqnarray*}

Since $\beta = \frac{5}{6\eta}$ and $\eta \leq \frac{1}{6L_f}$, we have $ \frac{1}{2\eta} - \frac{\beta}{2} - \frac{L_f}{2}  \geq 0$. Therefore, 
\begin{eqnarray*}
	&& \frac{1}{2\eta} \mathbb{E}_k[\|x^{k+1} -x^*\|^2] + \mathbb{E}_k[P(x^{k+1})] - P(x^*) \\ 
	&\leq&  \frac{1}{2\eta} \mathbb{E}_k[\|x^k - x^*\|^2] + \frac{2{\cal L}_1}{\beta}(f(x^k) - f(x^*) - \langle \nabla f(x^*), x^k - x^*\rangle ) + \frac{1}{\beta} \mathbb{E}_k[ {\cal H}^k_S ]. 
\end{eqnarray*}
From Lemma \ref{lm:Hk+1gconvex}, we have 
\begin{eqnarray*}
	&& \mathbb{E}_k[\Psi^{k+1}] + \mathbb{E}_k[P(x^{k+1})] - P(x^*) \\ 
	&\leq&  \frac{1}{2\eta} \mathbb{E}_k[\|x^k - x^*\|^2] + \left( \frac{2{\cal L}_1}{\beta} + 2\alpha p {\cal L}_1 \right)(f(x^k) - f(x^*) - \langle \nabla f(x^*), x^k - x^*\rangle ) \\ 
	&& + \frac{1}{\beta} \mathbb{E}_k[ {\cal H}^k_S ] + \alpha (1-p)\mathbb{E}_k[ {\cal H}^k_S ] \\ 
	&=& \mathbb{E}_k[\Psi^k] + \frac{4{\cal L}_1}{\beta} (f(x^k) - f(x^*) - \langle \nabla f(x^*), x^k - x^*\rangle ) \\ 
	&\leq& \mathbb{E}_k[\Psi^k] + \frac{3}{5}  (f(x^k) - f(x^*) - \langle \nabla f(x^*), x^k - x^*\rangle ) . 
\end{eqnarray*}

Since $x^*$ is an optimal solution, we have $- \nabla f(x^*) \in \partial \psi(x^*)$, which along with the convexity of $\psi$ implies that 
\begin{equation}\label{eq:fP}
f(x^k) - f(x^*) - \langle \nabla f(x^*), x^k - x^*\rangle \leq P(x^k) - P(x^*). 
\end{equation}
Thus we can obtain the first result.


Since $\eta = \frac{1}{8{\cal L}_1}$, we have 
\begin{eqnarray*}
	\mathbb{E}[\Psi^0] &=& \frac{1}{2\eta}\|x^0 - x^*\|^2 + \alpha \mathbb{E}[{\cal H}^0_S] \\
	&=& 4{\cal L}_1 \|x^0 - x^*\|^2 + \frac{3}{20p{\cal L}_1} \mathbb{E}[{\cal H}^0_S] \\ 
	&\overset{Assumption ~\ref{as:expsmooth}}{\leq}& 4{\cal L}_1 \|x^0 - x^*\|^2 + \frac{3}{10p} (f(x^0) - f(x^*) - \langle \nabla f(x^*), x^0 - x^*\rangle ) \\ 
	&\overset{(\ref{eq:fP})}{\leq}& 4{\cal L}_1 \|x^0 - x^*\|^2 + \frac{3}{10p}(P(x^0) - P(x^*)). 
\end{eqnarray*}

From Theorem \ref{Th:gconvex1}, we have 
$$
\mathbb{E}[P(x^i) - P(x^*)] - \frac{3}{5}\mathbb{E}[P(x^{i-1}) - P(x^*)] \leq \mathbb{E}[\Psi^{i-1}] - \mathbb{E}[\Psi^i], 
$$
for $i = 1, ..., k$. Hence 
$$
\mathbb{E}[P(x^k) - P(x^*)] + \frac{2}{5} \sum_{i=0}^{k-1} \mathbb{E}[P(x^i) - P(x^*)] \leq P(x^0) - P(x^*) + \mathbb{E}[\Psi^0] - \mathbb{E}[\Psi^{k}], 
$$
which implies that 
$$
\frac{2}{5} \sum_{i=0}^{k} \mathbb{E}[P(x^i) - P(x^*)] \leq 4{\cal L}_1 \|x^0 - x^*\|^2 + ( 1 + \frac{3}{10p})(P(x^0) - P(x^*)). 
$$
This along with $\mathbb{E}[P({\tilde x}^k) - P(x^*)] \leq \frac{1}{k+1} \sum_{i=0}^{k} \mathbb{E}[P(x^i) - P(x^*)] $ implies the result. 

\section{Nonconvex Case: Proof of Theorem \ref{Th:nonconvex1} }

\subsection{Lemmas}

\begin{lemma}\label{lm:gknonconvex}
	For $g^k$, we have 
	$$
	\mathbb{E}_k [\|g^k\|^2] \leq 2\|\nabla f(x^k) \|^2 + 2{\cal L}_3 \|x^k - w^k\|^2.  
	$$
\end{lemma}

\begin{proof}
	
	\begin{eqnarray*}
		\mathbb{E}_k [\|g^k\|^2] &=& \mathbb{E}_k [\|g^k - \nabla f(x^k) + \nabla f(x^k) \|^2] \\ 
		&\leq& 2\|\nabla f(x^k)\|^2 + 2\mathbb{E}_k[\|g^k - \nabla f(x^k)\|^2] \\
		&\overset{Assumption \ref{as:expL3}}{\leq}& 2 \|\nabla f(x^k)\|^2 + 2{\cal L}_3\|x^k - w^k\|^2. 
	\end{eqnarray*}
	
\end{proof}

\begin{lemma}\label{lm:x-wnonconvex}
	For any $\beta>0$, we have 
	$$
	\mathbb{E}_k[\| x^{k+1} - w^{k+1}\|^2] \leq \eta^2\mathbb{E}_k[\|g^k\|^2] + (1-p)(1+\eta\beta) \|x^k - w^k\|^2 + (1-p)\frac{\eta}{\beta}\|\nabla f(x^k)\|^2. 
	$$
	
\end{lemma}

\begin{proof}
	
	\begin{eqnarray*}
		\mathbb{E}_k[\|x^{k+1} - w^{k+1}\|^2] &=& p\mathbb{E}_k[\|x^{k+1} - x^k\|^2] + (1-p)\mathbb{E}_k[\|x^{k+1} -w^k\|^2 ] \\ 
		&=& p\eta^2\mathbb{E}_k[\|g^k\|^2] + (1-p)\mathbb{E}_k[\|x^{k+1} - w^k\|^2]. 
	\end{eqnarray*}
	For $\mathbb{E}_k[\|x^{k+1} - w^k\|^2]$, we have 
	\begin{eqnarray*}
		\mathbb{E}_k[\|x^{k+1} - w^k\|^2] &=& \mathbb{E}_k [ \|x^k - \eta g^k - w^k\|^2] \\
		&=& \|x^k-w^k\|^2 + \eta^2 \mathbb{E}_k [\|g^k\|^2] - 2\eta \mathbb{E}_k[\langle x^k-w^k, \nabla f(x^k) \rangle] \\ 
		&\leq& \|x^k - w^k\|^2 + \eta^2 \mathbb{E}_k[\|g^k\|^2] + \eta \left(  \frac{1}{\beta}\|\nabla f(x^k)\|^2 + \beta \|x^k-w^k\|^2  \right) \\ 
		&=& (1+\eta \beta)\|x^k-w^k\|^2 + \eta^2 \mathbb{E}_k[\|g^k\|^2] + \frac{\eta}{\beta}\|\nabla f(x^k) \|^2,
	\end{eqnarray*}
	where the inequality is from $|2\langle a, b\rangle | \leq \frac{1}{\beta}\|a\|^2 + \beta\|b\|^2$ for any $\beta>0$. 
	Combining all the above results, we can obtain the result. 
	
\end{proof}

\subsection{Proof of Theorem \ref{Th:nonconvex1} }

Since $f$ is $L_f$-smooth, we have 
$$
f(x^{k+1}) \leq f(x^k) + \langle \nabla f(x^k), x^{k+1}-x^k \rangle + \frac{L_f}{2}\|x^{k+1}-x^k\|^2, 
$$
which implies 
$$
\mathbb{E}_k[f(x^{k+1})] \leq f(x^k) - \eta \|\nabla f(x^k)\|^2 + \frac{L_f \eta^2}{2}\mathbb{E}_k[\|g^k\|^2]. 
$$
Hence, we have 
\begin{eqnarray*}
	\mathbb{E}_k[\Psi^{k+1}] &=&\mathbb{E}_k[f(x^{k+1}) + \alpha \|x^{k+1}-w^{k+1}\|^2] \\ 
	&\leq& f(x^k) - \eta \|\nabla f(x^k)\|^2 + \frac{L_f \eta^2}{2}\mathbb{E}_k[\|g^k\|^2] + \alpha \mathbb{E}_k [\|x^{k+1} - w^{k+1}\|^2] \\
	&\overset{Lemma \ref{lm:x-wnonconvex}}{\leq}&  f(x^k) - \eta  \|\nabla f(x^k) \|^2 + \eta^2(\frac{L_f}{2} + \alpha)\mathbb{E}_k[\|g^k\|^2] \\ 
	&& + \alpha(1-p)(1+\eta\beta) \|x^k - w^k\|^2 + \alpha(1-p)\frac{\eta}{\beta}\|\nabla f(x^k)\|^2 \\ 
	&\overset{Lemma \ref{lm:gknonconvex}}{\leq}& f(x^k) - \eta (1 - \frac{\alpha(1-p)}{\beta}) \|\nabla f(x^k) \|^2 + \alpha(1-p)(1+\eta\beta) \|x^k - w^k\|^2 \\ 
	&& + \eta^2 (\frac{L_f}{2} + \alpha) \left(  2\|\nabla f(x^k)\|^2 + 2{\cal L}_3\|x^k - w^k\|^2  \right) \\ 
	&=& f(x^k) - \eta \left(1 - \frac{\alpha(1-p)}{\beta} - L_f\eta - 2\alpha\eta \right) \|\nabla f(x^k)\|^2 \\ 
	&& + \alpha \left(  (1-p)(1+\eta\beta) + \eta^2(\frac{L_f}{\alpha} + 2){\cal L}_3  \right)\|x^k - w^k\|^2. 
\end{eqnarray*}

Since $\alpha = 3\eta^2L_f{\cal L}_3/p$ and $\beta = p/3\eta$, we have 
$$
(1-p)(1+\eta\beta) + \eta^2(\frac{L_f}{\alpha} + 2){\cal L}_3 \leq 1 - \frac{p}{3} + 2{\cal L}_3\eta^2. 
$$
Let 
$$
2{\cal L}_3\eta^2 \leq \frac{p}{3}, \quad \frac{\alpha}{\beta} = \frac{9\eta^3L_f{\cal L}_3}{p^2} \leq \frac{1}{4}, \quad L_f\eta\leq \frac{1}{4}, \quad 2\alpha \eta = \frac{6\eta^3L_f{\cal L}_3}{p} \leq \frac{1}{4}, 
$$
which implies 
$$
\eta \leq \min \left\{ \frac{1}{4L_f}, \frac{p^{\frac{2}{3}} }{36^{\frac{1}{3}}(L_f{\cal L}_3)^{\frac{1}{3}} }, \frac{\sqrt{p}}{\sqrt{6{\cal L}_3}}   \right\}. 
$$
Then $(1-p)(1+\eta\beta) + \eta^2(\frac{L_f}{\alpha} + 2){\cal L}_3 \leq 1$ and $1 - \frac{\alpha(1-p)}{\beta} - L_f\eta - 2\alpha\eta \geq \frac{1}{4}$, which indicate that 
$$
\mathbb{E}_k [\Psi^{k+1}] \leq \Psi^k - \frac{\eta}{4}\|\nabla f(x^k) \|^2.
$$

\subsection{Proof of Corollary \ref{co:nonconvex1} }

If the stepsize $\eta$ satisfies (\ref{eq:etanonconvex}), from Theorem \ref{Th:nonconvex1}, we have 
$$
\mathbb{E}[\|\nabla f(x^k)\|] \leq \frac{4}{\eta} (\mathbb{E}[\Psi^{k}] - \mathbb{E}[\Psi^{k+1}]), 
$$
which implies that 
\begin{eqnarray*}
	\mathbb{E}[\|\nabla f(x^a)\|^2] &=& \frac{1}{k+1} \sum_{i=0}^k \mathbb{E}[\| \nabla f(x^i)\|^2] \\ 
	&\leq& \frac{1}{k+1} \cdot \frac{4}{\eta} \left(  \Psi^0 - \mathbb{E}[\Psi^{k+1}]  \right) \\ 
	&=&  \frac{1}{k+1} \cdot \frac{4}{\eta} \left(  f(x^0) -  \mathbb{E}[f(x^{k+1})] - \alpha \mathbb{E}[\|x^{k+1} -w^{k+1}\|^2]  \right) \\
	&\leq& \frac{4}{\eta}\cdot \frac{f(x^0) - f(x^*)}{k+1}. 
\end{eqnarray*}

\section{Estimation of ${\cal L}_1$ and ${\cal L}_2$ }\label{sec:L1L2}

In this section, we estimate the expected smoothness parameters ${\cal L}_1$ and ${\cal L}_2$ comprehensively. It should be noticed that for Katyusha \cite{Katyusha} in minibatch setting, if we replace ${\bar L}/\tau$ with ${\cal L}_2$, then we can obtain the same result straigtforward by using Lemma \ref{lm:expL2} instead of Lemma D.2 in \cite{Katyusha}. Hence, the estimation of ${\cal L}_2$ implies the convergence result of Katyusha with arbitrary sampling as well.

\subsection{Estimation for sampling $S$ }

Let $f_S \eqdef \frac{1}{n}F{\theta}_S{\bf I}_Se$, and the Lipschitz smoothness constant of $f_S$ be $L_S$. Obviously 
$$
L_S \leq \frac{1}{n} \sum_{i\in S} L_i{\theta}^i_S. 
$$
Let ${\cal L}_{\max} \eqdef \max_{i\in [n]} \sum_{C: i\in C}p_C L_C {\theta}^i_C$. Let ${\bf P}\in \R^{n\times n}$ be defined by ${\bf P}_{ij} = \mathbb{P}[i\in S \ \& \ j\in S]$.  Recall that  
\begin{equation}\label{betai}
\compactify \beta_i = \sum_{C \subseteq[n] : i\in C}p_C|C|(\theta_C^i)^2, \quad i\in [n],
\end{equation}
where $|C|$ is the cardinality of the set $C$. The property of $\beta_i$ can be found in Lemma 3.4 in \cite{SAGA-AS}. Then we have the following lemma. 

\begin{lemma}\label{lm:L1samplinggeneral}
	
	Let $S$ be a proper sampling and $f_i$ be $L_i$-smooth and convex. \\
	(i) The expected smoothness constants ${\cal L}_1$ in Assumption \ref{as:expsmooth} and ${\cal L}_2$ in Assumption \ref{as:expL2} satisfy 
	$$
	{\cal L}_i \leq {\cal L}_{\max} \leq \frac{1}{n}\max_{i\in [n]} \left\{ \sum_{j\in [n]}\sum_{C: i, j\in C}p_C{\theta}^i_C{\theta}^j_C L_j   \right\}, \quad i = 1, 2. 
	$$
	Specifically, if ${\theta}^i_C = \frac{1}{p_i}$ for all $i$ and $C$, then 
	$$
	{\cal L}_i \leq {\cal L}_{\max} \leq  \frac{1}{n}\max_{i\in [n]} \left\{ \sum_{j\in [n]}\frac{{\bf P}_{ij}}{p_ip_j}L_j   \right\}, \quad i = 1, 2. 
	$$
	(ii) We have  
	$$
	{\cal L}_i \leq \frac{1}{n} \max_i \{  L_i \beta_i  \}, \quad i = 1, 2. 
	$$
	
\end{lemma}

\begin{proof}
	
	(i) From 
	$$
	\|\nabla f_S(x) - \nabla f_S(y)\|^2 \leq 2L_S (f_S(x) - f_S(y) - \langle \nabla f_S(y), x-y \rangle ),
	$$
	we have 
	\begin{eqnarray*}
		\mathbb{E}\left[\left\|\frac{1}{n}({\bf G}(x) - {\bf G}(y)){\theta}_{S}{\bf I}_{S}e \right\|^2 \right] &\leq& 2\sum_{C} p_C L_C(f_C(x) - f_C(y) - \langle \nabla f_C(y), x-y \rangle ) \\ 
		&=& 2\sum_{C} p_C L_C \sum_{i\in C} \frac{{\theta}^i_C}{n}(f_i(x) - f_i(y) - \langle \nabla f_i(y), x-y \rangle ) \\ 
		&=& \frac{2}{n} \sum_{i\in [n]} \sum_{C: i\in C}p_C L_C {\theta}^i_C (f_i(x) - f_i(y) - \langle \nabla f_i(y), x-y \rangle ) \\ 
		&\leq& \frac{2}{n}{\cal L}_{\max} \sum_{i\in [n]} (f_i(x) - f_i(y) - \langle \nabla f_i(y), x-y \rangle ) \\ 
		&=& 2{\cal L}_{\max} (f(x) - f(y) - \langle \nabla f(y), x-y \rangle ). 
	\end{eqnarray*}
	Hence ${\cal L}_1 \leq {\cal L}_{\max} = \max_{i\in [n]} \sum_{C: i\in C}p_C L_C {\theta}^i_C$. For all $i$, since $L_S \leq \frac{1}{n} \sum_{i\in S} L_i{\theta}^i_S$,  we have 
	\begin{eqnarray*}
		\sum_{C: i\in C}p_C L_C{\theta}^i_C &\leq& \sum_{C: i\in C}p_C{\theta}^i_C \frac{1}{n} \sum_{j\in C}L_j{\theta}^j_C \\ 
		&=& \frac{1}{n} \sum_{j\in [n]}\sum_{C: i, j\in C}p_C{\theta}^i_C{\theta}^j_C L_j, 
	\end{eqnarray*}
	which implies 
	$$
	{\cal L}_{\max} \leq \frac{1}{n}\max_{i\in [n]} \left\{ \sum_{j\in [n]}\sum_{C: i, j\in C}p_C{\theta}^i_C{\theta}^j_C L_j   \right\}. 
	$$
	Furthermore, if ${\theta}^i_C = \frac{1}{p_i}$ for all $i$ and $C$, then 
	$$
	{\cal L}_{\max} \leq \frac{1}{n}\max_{i\in [n]} \left\{ \sum_{j\in [n]}\sum_{C: i, j\in C}p_C{\theta}^i_C{\theta}^j_C L_j   \right\} = \frac{1}{n}\max_{i\in [n]} \left\{ \sum_{j\in [n]}\frac{{\bf P}_{ij}}{p_ip_j}L_j   \right\}. 
	$$
	
	\noindent For ${\cal L}_2$, since $\mathbb{E}[\|X - \mathbb{E}[X] \|^2] \leq \mathbb{E}[\| X \|^2]$, we have 
	\begin{equation}\label{eq:boundL2}
	\mathbb{E}\left[\left\|\frac{1}{n}({\bf G}(x) - {\bf G}(y)){\theta}_{S}{\bf I}_{S}e - \frac{1}{n}({\bf G}(x) - {\bf G}(y))e\right\|^2 \right] \leq \mathbb{E}\left[\left\|\frac{1}{n}({\bf G}(x) - {\bf G}(y)){\theta}_{S}{\bf I}_{S}e \right\|^2 \right]. 
	\end{equation}
	Then we get the same upper bound for ${\cal L}_2$.

	(ii) From the definition of $\beta_i$, we have 
	\begin{eqnarray}
	\mathbb{E}\left[\left\|\frac{1}{n}({\bf G}(x) - {\bf G}(y)){\theta}_{S}{\bf I}_{S}e \right\|^2 \right] &=& \sum_{C} p_C\left \|  \frac{1}{n}({\bf G}(x) - {\bf G}(y)){\theta}_{C}{\bf I}_{C}e \right\|^2 \nonumber \\ 
	&=&\frac{1}{n^2} \sum_C p_C \left \| \sum_{i\in C} \theta^i_C (\nabla f_i(x) - \nabla f_i(y)) \right \|^2 \nonumber \\ 
	&\leq& \frac{1}{n^2} \sum_C p_C |C| \sum_{i\in C} (\theta^i_C)^2 \|\nabla f_i(x) - \nabla f_i(y) \|^2 \nonumber \\ 
	&=& \frac{1}{n^2} \sum_{i=1}^n \sum_{C: i\in C} p_C |C| (\theta^i_C)^2 \|\nabla f_i(x) - \nabla f_i(y) \|^2\nonumber \\ 
	&=& \frac{1}{n^2} \sum_{i=1}^n \beta_i \|\nabla f_i(x) - \nabla f_i(y) \|^2 \label{eq:L12beta}\\ 
	&\leq& \frac{1}{n^2} \sum_{i=1}^n 2L_i \beta_i (f_i(x) - f_i(y) - \langle \nabla f_i(y), x-y \rangle ) \nonumber \\ 
	&\leq& \frac{2}{n} \max_i\{ L_i \beta_i  \} (f(x) - f(y) - \langle \nabla f(y), x-y \rangle ). \nonumber 
	\end{eqnarray}
	This along with (\ref{eq:boundL2}) implies the results. 
	
\end{proof}

\begin{lemma}\label{lm:L1samplingnsgs}
	Let $S$ be a proper sampling, $\theta^i_S = 1/p_i$, $f_i$ be $L_i$-smooth and convex, and $f$ be $L_f$-smooth and convex. \\
	(i) For $\tau$-nice sampling $S$, the expected smoothness constant ${\cal L}_1$ in Assumption \ref{as:expsmooth} satisfies 
	$$
	{\cal L}_1 \leq \frac{n(\tau -1)}{\tau (n-1)} L_f + \frac{n-\tau}{\tau (n-1)} \max_i \{ L_i  \}. 
	$$
	(ii) For group sampling $S$, denote the isolated index set as ${\cal I}$, then we have 
	$$
	{\cal L}_1 \leq L_f + \frac{1}{n} \max\left\{  \max_{i\in {\cal I}} (\frac{1}{p_i} - 1)L_i ,\  \max_{i \notin {\cal I}} \frac{L_i}{p_i} \right\}. 
	$$
\end{lemma}

\begin{proof}
	(i) This is Proposition 3.8(ii) in \cite{SGDAS}. 
	
	\vskip 2mm
	
	\noindent (ii) Since $f_i$ is $L_i$-smooth and convex, and $f$ is $L_f$-smooth and convex, we have 
	\begin{equation}\label{eq:fismooth}
	\| \nabla f_i(x) - \nabla f_i(y) \|^2 \leq 2L_i (f_i(x) - f_i(y) - \langle \nabla f_i(y), x-y \rangle ), 
	\end{equation}
	and 
	\begin{equation}\label{eq:fsmooth}
	\| \nabla f(x) - \nabla f(y) \|^2 \leq 2L_f (f(x) - f(y) - \langle \nabla f(y), x-y \rangle ). 
	\end{equation}
	
	From ${\theta}^i_S = \frac{1}{p_i}$, we have 
	\begin{eqnarray*}
		&& \mathbb{E}\left[\left\|\frac{1}{n}({\bf G}(x) - {\bf G}(y)){\theta}_{S}{\bf I}_{S}e \right\|^2 \right] \\ 
		&=& \frac{1}{n^2} \mathbb{E} \left[  \|\sum_{i \in S} \frac{1}{p_i} (\nabla f_i(x) - \nabla f_i(y))\|^2  \right] \\ 
		&=& \frac{1}{n^2} \sum_C p_C \left \langle \sum_{i \in C} \frac{1}{p_i} (\nabla f_i(x) - \nabla f_i(y)), \sum_{i \in C} \frac{1}{p_i} (\nabla f_i(x) - \nabla f_i(y)) \right \rangle \\ 
		&=& \frac{1}{n^2} \sum_C p_C \sum_{i, j \in C} \left \langle  \frac{1}{p_i} (\nabla f_i(x) - \nabla f_i(y)), \frac{1}{p_j} (\nabla f_j(x) - \nabla f_j(y)) \right \rangle \\ 
		&=& \frac{1}{n^2} \sum_{i, j=1}^n \sum_{C: i, j \in C} p_C \left \langle  \frac{1}{p_i} (\nabla f_i(x) - \nabla f_i(y)), \frac{1}{p_j} (\nabla f_j(x) - \nabla f_j(y)) \right \rangle \\ 
		&=& \frac{1}{n^2} \sum_{i, j=1}^n  \frac{{\bf P}_{ij}}{p_ip_j} \left \langle  (\nabla f_i(x) - \nabla f_i(y)), (\nabla f_j(x) - \nabla f_j(y)) \right \rangle. 
	\end{eqnarray*}
	
	For group sampling, we have ${\bf P}_{ij} = 0$ if $i, j$ are in the same group, and ${\bf P}_{ij} = p_ip_j$ if $i, j$ are in different groups. Assume $S$ have $t$ groups $C_j$ with $j=1, ..., t$, and denote ${\cal I}$ as the isolated index set. Then we have 
	\begin{eqnarray}
	&& \mathbb{E}\left[\left\|\frac{1}{n}({\bf G}(x) - {\bf G}(y)){\theta}_{S}{\bf I}_{S}e \right\|^2 \right] \nonumber \\ 
	&=& \frac{1}{n^2} \sum_{i, j=1}^n  \frac{{\bf P}_{ij}}{p_ip_j} \left \langle  (\nabla f_i(x) - \nabla f_i(y)), (\nabla f_j(x) - \nabla f_j(y)) \right \rangle \nonumber \\ 
	&=& \frac{1}{n^2} \sum_{i\neq j}  \frac{{\bf P}_{ij}}{p_ip_j} \left \langle  (\nabla f_i(x) - \nabla f_i(y)), (\nabla f_j(x) - \nabla f_j(y)) \right \rangle \nonumber \\ 
	&& + \frac{1}{n^2} \sum_{i=1}^n \frac{1}{p_i} \|\nabla f_i(x) - \nabla f_i(y)\|^2 \nonumber \\ 
	&=&  \| \nabla f(x) - \nabla f(y)\|^2 - \frac{1}{n^2} \sum_{j =1}^t \left \| \sum_{i \in C_j} (\nabla f_i(x) - \nabla f_i(y)) \right \|^2 \nonumber \\ 
	&& + \frac{1}{n^2} \sum_{i=1}^n \frac{1}{p_i} \|\nabla f_i(x) - \nabla f_i(y)\|^2  \label{eq:groupsL} \\
	&\overset{{(\ref{eq:fsmooth})}}{\leq}& 2L_f(f(x) - f(y) - \langle \nabla f(y), x-y \rangle ) \nonumber \\ 
	&& + \frac{1}{n^2} \sum_{i\in {\cal I}} (\frac{1}{p_i} - 1)\|\nabla f_i(x) - \nabla f_i(y)\|^2 + \frac{1}{n^2} \sum_{i \notin {\cal I}}\frac{1}{p_i} \|\nabla f_i(x) - \nabla f_i(y)\|^2 \nonumber \\ 
	&\overset{(\ref{eq:fismooth})}{\leq}&  2L_f(f(x) - f(y) - \langle \nabla f(y), x-y \rangle ) \nonumber \\ 
	&& + \frac{2}{n} \max\left\{  \max_{i\in {\cal I}} (\frac{1}{p_i} - 1)L_i ,\  \max_{i \notin {\cal I}} \frac{L_i}{p_i} \right\} (f(x) - f(y) - \langle \nabla f(y), x-y \rangle ). \nonumber
	\end{eqnarray}
	
\end{proof}

\begin{lemma}\label{lm:L2samplingnsgs}
	Let $S$ be a proper sampling, $\theta^i_S = 1/p_i$, $f_i$ be $L_i$-smooth and convex, and $f$ be $L_f$-smooth and convex. \\
	(i) For $\tau$-nice sampling $S$, the ${\cal L}_2$ in Assumption \ref{as:expL2} satisfies 
	$$
	{\cal L}_2 \leq  \frac{n-\tau}{\tau (n-1)} \max_i \{ L_i  \}. 
	$$
	(ii) For group sampling $S$, denote the isolated index set as ${\cal I}$, then we have 
	$$
	{\cal L}_2 \leq  \frac{1}{n} \max\left\{  \max_{i\in {\cal I}} (\frac{1}{p_i} - 1)L_i ,\  \max_{i \notin {\cal I}} \frac{L_i}{p_i} \right\}. 
	$$
\end{lemma}

\begin{proof}
	Noticing that $\mathbb{E}[\| X - \mathbb{E}[X] \|^2] = \mathbb{E}[\|X \|^2] - \| \mathbb{E}[X] \|^2$, we have 
	\begin{eqnarray}\label{eq:forL2-1}
	&& \mathbb{E}\left[\left\|\frac{1}{n}({\bf G}(x) - {\bf G}(y)){\theta}_{S}{\bf I}_{S}e - \frac{1}{n}({\bf G}(x) - {\bf G}(y))e\right\|^2 \right] \nonumber \\ 
	&=& \mathbb{E}\left[\left\|\frac{1}{n}({\bf G}(x) - {\bf G}(y)){\theta}_{S}{\bf I}_{S}e \right\|^2 \right] - \left \| \frac{1}{n}({\bf G}(x) - {\bf G}(y))e \right \|^2 \nonumber \\ 
	&=&   \mathbb{E}\left[\left\|\frac{1}{n}({\bf G}(x) - {\bf G}(y)){\theta}_{S}{\bf I}_{S}e \right\|^2 \right] - \left \| \nabla f(x) - \nabla f(y) \right \|^2. 
	\end{eqnarray}
	
	Then similar to the proof of Lemma \ref{lm:L1samplingnsgs}, we can obtain the results.
	
\end{proof}

Consider $f_i(x) = \phi_i({\bf A}_i^\top x)$, where ${\bf A}_i \in{ \mathbb{R}}^{d\times m}$, $\phi_i: {\mathbb{R}}^m \to \mathbb{R}$ is $1/\gamma$-smooth and convex. 

\vskip 2mm

The parameters $v_1, ..., v_n$ are assumed to satisfy the following expected separable overapproximation (ESO) inequality, which needs to hold for all $h_i\in \R^m$:
\begin{equation}\label{eq:ESOfirst}
\compactify \mathbb{E}_S\left[ \left\|\sum_{i\in S} {\bf A}_ih_i \right\|^2 \right] \leq \sum_{i=1}^n p_iv_i\|h_i\|^2.
\end{equation}

\begin{lemma}\label{lm:phiexp} If $\phi_i$ is $1/\gamma$-smooth and convex, then for any $x, y\in \R^d$, we have 
	$$
	\left\| \nabla\phi_i({\bf A}_i^\top x)-\nabla \phi_i({\bf A}_i^\top y)\right\|^2 \leq \frac{2}{\gamma} \left( f_i(x) - f_i(y) - \langle \nabla f_i(y), x-y \rangle  \right)
	$$
\end{lemma}

\begin{proof}
	Since $\phi_i$ is $1/\gamma$-smooth, we have 
	$$
	\left\| \nabla \phi_i(\tilde x) - \nabla \phi_i(\tilde y)\right\|^2 \leq \frac{2}{\gamma} \left(  \phi_i(\tilde x) - \phi_i(\tilde y) - \langle \nabla \phi_i(\tilde y), {\tilde x} - {\tilde y}  \right). 
	$$
	Letting $\tilde x = {\bf A}_i^\top x$, and $\tilde y = {\bf A}_i^\top y$ in the above inequlity yields 
	\begin{eqnarray*}
		\left\| \nabla \phi_i({\bf A}_i^\top x) - \nabla \phi_i({\bf A}_i^\top y)\right\|^2 &\leq&  \frac{2}{\gamma} \left(  \phi_i({\bf A}_i^\top x) - \phi_i({\bf A}_i^\top y) - \langle \nabla \phi_i({\bf A}_i^\top y ), {{\bf A}_i^\top x} - {{\bf A}_i^\top y}  \right) \\
		&=& \frac{2}{\gamma} \left( f_i(x) - f_i(y) - \langle \nabla f_i(y), x-y \rangle  \right)
	\end{eqnarray*}
	
\end{proof}

\begin{lemma}\label{lm:esoexp}
	Let $S$ be a proper sampling, and $\theta^i_S = 1/p_i$. If the ESO inequality (\ref{eq:ESOfirst}) holds, then the expected smoothness constants ${\cal L}_1$ in Assumption \ref{as:expsmooth} and ${\cal L}_2$ in Assumption \ref{as:expL2} satisfy
	$$
	{\cal L}_i \leq \frac{1}{n\gamma} \max_i \{  \frac{v_i}{p_i}  \}, \quad i = 1, 2. 
	$$
\end{lemma}

\begin{proof}
	
	\begin{eqnarray}
	\mathbb{E}\left[\left\|\frac{1}{n}({\bf G}(x) - {\bf G}(y)){\theta}_{S}{\bf I}_{S}e \right\|^2 \right] &=& \frac{1}{n^2} \mathbb{E} \left[ \left\| \sum_{i\in S} \theta^i_S(\nabla f_i(x) - \nabla f_i(y)) \right\|^2 \right] \nonumber \\
	&=& \frac{1}{n^2} \mathbb{E} \left[ \left\| \sum_{i\in S} \frac{1}{p_i}{\bf A}_i(\nabla \phi_i({\bf A}_i^\top x) - \nabla \phi_i({\bf A}_i^\top y)) \right\|^2 \right] \nonumber \\
	&\overset{(\ref{eq:ESOfirst})}{\leq}& \frac{1}{n^2} \sum_{i=1}^n p_iv_i\cdot \frac{1}{p_i^2} \left\| \nabla\phi_i({\bf A}_i^\top x)-\nabla \phi_i({\bf A}_i^\top y)\right\|^2 \label{eq:forL3exo} \\ 
	&\overset{Lemma  \ref{lm:phiexp}}{\leq}& \frac{1}{n^2} \sum_{i=1}^n \frac{v_i}{p_i}\cdot \frac{2}{\gamma}  \left( f_i(x) - f_i(y) - \langle \nabla f_i(y), x-y \rangle  \right) \nonumber \\ 
	&\leq& \frac{2}{n\gamma} \max_i \frac{v_i}{p_i} \cdot \frac{1}{n} \sum_{i=1}^n  \left( f_i(x) - f_i(y) - \langle \nabla f_i(y), x-y \rangle  \right) \nonumber \\
	&=& \frac{2}{n\gamma} \max_i \frac{v_i}{p_i} \left( f(x) - f(y) - \langle \nabla f(y), x-y \rangle  \right). \nonumber
	\end{eqnarray}
	
	This along with (\ref{eq:boundL2}) implies the results.

\end{proof}

\subsection{Estimation for sampling with replacement $S$ }

\begin{lemma}\label{lm:L1sp}	
	Let ${\tilde p}_i >0$ in distribution ${\tilde D}$ for the sampling with replacement $S$, $f_i$ be $L_i$-smooth and convex, and $f$ be $L_f$-smooth and convex. Then the expected smoothness constant ${\cal L}_1$ in Assumption \ref{as:expsmooth} satisfies 
	$$
	{\cal L}_1 \leq  (1- \frac{1}{\tau})L_f + \frac{1}{n\tau} \max_{i} \frac{L_i}{{\tilde p}_i}.  
	$$
\end{lemma}

\begin{proof}
	
	\begin{eqnarray}\label{eq:forL2-2}
	&& \mathbb{E}\left[\left\|\frac{1}{n}({\bf G}(x) - {\bf G}(y)){\theta}_{S}{\bf I}_{S}e \right\|^2 \right] \nonumber \\ 
	&=& \mathbb{E} \left[ \left\| \sum_{i\in S} \frac{1}{n\tau {\tilde p}_i}\left( \nabla f_i(x) -\nabla f_i(y) \right) \right\|^2  \right] \\ 
	&=& \frac{1}{\tau^2} \mathbb{E} \left[ \sum_{i,j\in S} \left \langle \frac{1}{n{\tilde p}_i}(\nabla f_i(x) - \nabla f_i(y)), \frac{1}{n{\tilde p}_i}(\nabla f_j(x) - \nabla f_j(y)) \right \rangle \right] \nonumber \\ 
	&=& \frac{\tau}{\tau^2} {\mathbb{E}}_{i \sim {\tilde {\cal D}}}\left[ \left\| \frac{1}{n{\tilde p}_i}(\nabla f_i(x) - \nabla f_i(y)) \right\|^2 \right] + \frac{\tau^2 - \tau}{\tau^2} \left\| \nabla f(x) - \nabla f(y) \right\|^2. \nonumber 
	\end{eqnarray}
	
	From (\ref{eq:fismooth}), we have 
	\begin{equation}\label{eq:forL2-3}
	{\mathbb{E}}_{i \sim {\tilde {\cal D}}}\left[ \left\| \frac{1}{n{\tilde p}_i}(\nabla f_i(x) - \nabla f_i(y)) \right\|^2 \right] \leq \frac{2}{n}\max_{i}\frac{L_i}{{\tilde p}_i}(f(x) - f(y) - \langle \nabla f(y), x-y \rangle ). 
	\end{equation}
	
	Combining (\ref{eq:fsmooth}) and (\ref{eq:forL2-3}), we can obtain 
	$$
	\mathbb{E}\left[\left\|\frac{1}{n}({\bf G}(x) - {\bf G}(y)){\theta}_{S}{\bf I}_{S}e \right\|^2 \right] \leq 2\left( \frac{1}{n\tau} \max_{i} \frac{L_i}{{\tilde p}_i} + (1- \frac{1}{\tau})L_f \right) (f(x) - f(y) - \langle \nabla f(y), x-y \rangle ).
	$$
\end{proof}

\begin{lemma}\label{lm:L2sp}	
	Let ${\tilde p}_i >0$ in distribution ${\tilde D}$ for the sampling with replacement $S$, $f_i$ be $L_i$-smooth and convex, and $f$ be $L_f$-smooth and convex. Then the expected smoothness constant ${\cal L}_2$ in Assumption \ref{as:expL2} satisfies 
	$$
	{\cal L}_2 \leq  \frac{1}{n\tau} \max_{i} \frac{L_i}{{\tilde p}_i}.  
	$$
\end{lemma}

\begin{proof}
	Combing (\ref{eq:forL2-1}) and (\ref{eq:forL2-2}), we have 
	\begin{eqnarray}
	&& \mathbb{E}\left[\left\|\frac{1}{n}({\bf G}(x) - {\bf G}(y)){\theta}_{S}{\bf I}_{S}e - \frac{1}{n}({\bf G}(x) - {\bf G}(y))e\right\|^2 \right] \nonumber \\
	&=&  \frac{\tau}{\tau^2} {\mathbb{E}}_{i \sim {\tilde {\cal D}}}\left[ \left\| \frac{1}{n{\tilde p}_i}(\nabla f_i(x) - \nabla f_i(y)) \right\|^2 \right]  - \frac{ \tau}{\tau^2} \left\| \nabla f(x) - \nabla f(y) \right\|^2 \nonumber \\ 
	&\leq&  \frac{\tau}{\tau^2} {\mathbb{E}}_{i \sim {\tilde {\cal D}}}\left[ \left\| \frac{1}{n{\tilde p}_i}(\nabla f_i(x) - \nabla f_i(y)) \right\|^2 \right] \label{eq:forL3sp} \\ 
	&\overset{(\ref{eq:forL2-3})}{\leq}&  \frac{2}{n\tau}\max_{i}\frac{L_i}{{\tilde p}_i}(f(x) - f(y) - \langle \nabla f(y), x-y \rangle ). \nonumber 
	\end{eqnarray}

\end{proof}

\section{Importance Sampling and Importance sampling with replacement}\label{sec:importance}

In this section, we contruct importance sampling and importance sampling with replacement respectively. 

\vskip 2mm

Let $\tau$ be expected minibatch size for sampling or the number of copies for sampling with replacement, and $p = \frac{\tau}{n}$. Then by Theorem \ref{Th:convp}, the iteration complexity for L-SVRG is 
\begin{equation}\label{eq:itersvrg}
{\cal O}\left( \left( \frac{n}{\tau} + \frac{{\cal L}_1}{\mu}\right) \log\frac{1}{\epsilon}\right),
\end{equation}
and by Theorem \ref{Th:LKa}, the iteration complexity for L-Katyusha is 
\begin{equation}\label{eq:iterkat}
{\cal O}\left( \left( \frac{n}{\tau} + \sqrt{\frac{L_f}{\mu}} + \sqrt{\frac{{\cal L}_2n}{\mu \tau}}  \right) \log\frac{1}{\epsilon}\right). 
\end{equation}
From (\ref{eq:itersvrg}) and (\ref{eq:iterkat}), we can see that we need to make ${\cal L}_1$ and ${\cal L}_2$ as small as possible.

\subsection{Importance sampling}

We focus on the group sampling. From Lemma \ref{lm:L1samplingnsgs} (ii) and Lemma \ref{lm:L2samplingnsgs} (ii), we need to minimize  
\begin{equation}\label{eq:ims11}
\max\left\{  \max_{i\in {\cal I}} (\frac{1}{p_i} - 1)L_i ,\  \max_{i \notin {\cal I}} \frac{L_i}{p_i} \right\}, 
\end{equation}
where ${\cal I}$ is the isolated index set. The minimization of (\ref{eq:ims11}) is not easy generally, next we focus on finding an approximal solution. Let 
$$
q_i = \frac{L_i}{\sum_{i=1}^n L_i} \cdot \tau, 
$$
and $T = \{  i | q_i > 1  \}$. If $T = \emptyset$, by choosing $p_i = q_i$ for all $i$, we can get 
$$
\max\left\{  \max_{i\in {\cal I}} (\frac{1}{p_i} - 1)L_i ,\  \max_{i \notin {\cal I}} \frac{L_i}{p_i} \right\} \leq \frac{n{\bar L}}{\tau}.  
$$
If $T \neq \emptyset$, we can choosing $p_i = 1$ for $i \in T$, and $q_i \leq p_i \leq 1$ such that $\sum_{i=1}^n p_i = \tau$. In this way, noticing that $p_i = 1$ implies $i$ is an isolated index by the definition of group sampling, we have $(\frac{1}{p_i} - 1)L_i = 0$ for $i\in T$. Hence we can also obtain 
$$
\max\left\{  \max_{i\in {\cal I}} (\frac{1}{p_i} - 1)L_i ,\  \max_{i \notin {\cal I}} \frac{L_i}{p_i} \right\} \leq \frac{n{\bar L}}{\tau}. 
$$

To summarize the above two cases, by choosing $\min\{ q_i, 1  \} \leq p_i \leq 1$ such that $\sum_{i\in [n]}p_i = \tau$, we have 
\begin{equation}\label{eq:ims22}
\max\left\{  \max_{i\in {\cal I}} (\frac{1}{p_i} - 1)L_i ,\  \max_{i \notin {\cal I}} \frac{L_i}{p_i} \right\} \leq \frac{n{\bar L}}{\tau}. 
\end{equation}
It should be noticed that in practice, we can just choose $p_i = \min\{ q_i, 1 \}$ for convenience, and then  (\ref{eq:ims22})  also holds, but with $\mathbb{E}[|S|] = \sum_{i\in [n]}p_i \leq \tau$. From (\ref{eq:ims22}) and Lemmas \ref{lm:L1samplingnsgs} and \ref{lm:L2samplingnsgs}, we have 
${\cal L}_1 \leq L_f + \frac{{\bar L}}{\tau}$ and ${\cal L}_2 \leq \frac{{\bar L}}{\tau}$. Therefore, from (\ref{eq:itersvrg}) the iteration complexity for L-SVRG becomes 
\begin{equation}\label{eq:itersvrgis}
{\cal O}\left( \left( \frac{n}{\tau} + \frac{L_f}{\mu} + \frac{{\bar L}}{\tau \mu} \right) \log\frac{1}{\epsilon}\right), 
\end{equation}
which has linear speed up with respect to $\tau$ when $\tau \leq \frac{{\bar L}}{L_f}$. While, when $\tau \geq \frac{{\bar L}}{L_f}$, (\ref{eq:itersvrgis}) becomes 
$$
{\cal O}\left( \left( \frac{n}{\tau} + \frac{L_f}{\mu}  \right) \log\frac{1}{\epsilon}\right). 
$$ 
From (\ref{eq:iterkat}), the iteration complexity for L-Katyusha becomes 
\begin{equation}\label{eq:iterkatis}
{\cal O}\left( \left( \frac{n}{\tau} + \sqrt{\frac{L_f}{\mu}} + \frac{1}{\tau}\sqrt{\frac{{\bar L}n}{\mu} }  \right) \log\frac{1}{\epsilon}\right),  
\end{equation}
which has linear speed up with respect to $\tau$ when $\tau \leq \sqrt{\frac{{\bar L}n}{L_f}}$. While when $\tau \geq \sqrt{\frac{{\bar L}n}{L_f}}$, (\ref{eq:iterkatis}) becomes 
$$
{\cal O}\left( \left( \frac{n}{\tau} + \sqrt{\frac{L_f}{\mu}}   \right) \log\frac{1}{\epsilon}\right). 
$$

\subsection{Importance sampling with replacement} 

From Lemmas \ref{lm:L1sp} and \ref{lm:L2sp}, we need to minimize $ \max_{i} {L_i}/{{\tilde p}_i}$. It is easy to see that by choosing ${\tilde p}_i = L_i/n{\bar L}$, the minimum of $ \max_{i} {L_i}/{{\tilde p}_i}$ is $n{\bar L}$. In this case, ${\cal L}_1 = (1 - \frac{1}{\tau}) L_f + \frac{{\bar L}}{\tau}$, and ${\cal L}_2 = \frac{{\bar L}}{\tau}$. Hence, from (\ref{eq:itersvrg}), the iteration  complexity for L-SVRG becomes 
\begin{equation}\label{eq:itersvrgisp}
{\cal O}\left( \left( \frac{n}{\tau} + \frac{L_f}{\mu} + \frac{{\bar L} - L_f}{\tau \mu} \right) \log\frac{1}{\epsilon}\right), 
\end{equation}
which has linear speed up with respect to $\tau$ when $\tau \leq \frac{{\bar L}}{L_f} - 1$. While, when $\tau \geq \frac{{\bar L}}{L_f}-1$, (\ref{eq:itersvrgis}) becomes 
$$
{\cal O}\left( \left( \frac{n}{\tau} + \frac{L_f}{\mu}  \right) \log\frac{1}{\epsilon}\right). 
$$ 
From (\ref{eq:iterkat}), the iteration complexity for L-Katyusha becomes 
\begin{equation}\label{eq:iterkatisp}
{\cal O}\left( \left( \frac{n}{\tau} + \sqrt{\frac{L_f}{\mu}} + \frac{1}{\tau}\sqrt{\frac{{\bar L}n}{\mu} }  \right) \log\frac{1}{\epsilon}\right),  
\end{equation}
which has linear speed up with respect to $\tau$ when $\tau \leq \sqrt{\frac{{\bar L}n}{L_f}}$. While when $\tau \geq \sqrt{\frac{{\bar L}n}{L_f}}$, (\ref{eq:iterkatis}) becomes 
$$
{\cal O}\left( \left( \frac{n}{\tau} + \sqrt{\frac{L_f}{\mu}}   \right) \log\frac{1}{\epsilon}\right). 
$$

\subsection{Comparison} 

For L-SVRG, (\ref{eq:itersvrgis}) and (\ref{eq:itersvrgisp}) have the essentially same bounds with the iteration complexity of SAGA with importance sampling in \cite{SAGA-AS}. From (\ref{eq:itersvrg}) and Lemma \ref{lm:L1samplinggeneral}, the iteration complexity for L-SVRG with arbitrary sampling becomes 
$$
{\cal O}\left( \left( \frac{n}{\tau} + \frac{1}{n\mu} \max_i\{  L_i\beta_i  \} \right)  \log\frac{1}{\epsilon} \right). 
$$
While the iteration complexity of SAGA with arbitrary sampling \cite{SAGA-AS} is 
$$
{\cal O}\left(  \max_i\left\{  \frac{1}{p_i} + \frac{4L_i\beta_i}{n\mu}  \right\} \log\frac{1}{\epsilon} \right). 
$$

From (\ref{eq:itersvrg}) and Lemma \ref{lm:esoexp}, the iteration complexity of L-SVRG becomes 
$$
{\cal O}\left( \left( \frac{n}{\tau} + \frac{1}{n\gamma \mu} \max_i\{  \frac{v_i}{p_i}  \} \right) \log\frac{1}{\epsilon}\right). 
$$
While the iteration complexity of Quartz in \cite{Quartz} is 
$$
{\cal O}\left(  \max_i\left\{  \frac{1}{p_i} + \frac{v_i}{p_in\gamma\mu}  \right\} \log\frac{1}{\epsilon} \right). 
$$
Since $\max_i \{  \frac{1}{p_i}  \} \geq \frac{n}{\tau}$, we can see the iteration complexity of L-SVRG is at least as good as that of SAGA and Quartz.

For L-Katyusha, (\ref{eq:iterkatis}) and (\ref{eq:iterkatisp}) have the essentially same bounds with the iteration complexity of Katyusha with importance sampling with replacement in \cite{Katyusha}.

\section{Estimation of ${\cal L}_3$}\label{sec:L3}

\subsection{Estimation for sampling $S$} 

\begin{lemma}\label{lm:L3sampling}
	Let $S$ be a proper sampling and $f_i$ be $L_i$-smooth. \\
	(i) The constant ${\cal L}_3$ in Assumption \ref{as:expL3} satisfies 
	$$
	{\cal L}_3 \leq \frac{1}{n^2} \sum_{i, j=1}^n \sum_{C: i, j \in C}p_C \theta^i_C\theta^j_C L_iL_j. 
	$$
	Specifically, if $\theta^i_C = \frac{1}{p_i}$ for all $i$ and $C$, then 
	$$
	{\cal L}_3 \leq \frac{1}{n^2} \sum_{i, j =1}^n \frac{{\bf P}_{ij}}{p_ip_j} L_iL_j. 
	$$
	(ii) We have 
	$$
	{\cal L}_3 \leq \frac{1}{n^2} \sum_{i=1}^n \beta_i L_i^2. 
	$$
	
\end{lemma}

\begin{proof}
	(i) First, we have 
	\begin{eqnarray}
	&& \mathbb{E}\left[\left\|\frac{1}{n}({\bf G}(x) - {\bf G}(y)){\theta}_{S}{\bf I}_{S}e \right\|^2 \right]  \nonumber\\ 
	&=& \frac{1}{n^2} \mathbb{E} \left[  \|\sum_{i \in S} {\theta}^i_S (\nabla f_i(x) - \nabla f_i(y))\|^2  \right] \nonumber \\ 
	&=& \frac{1}{n^2} \sum_C p_C \left \langle \sum_{i \in C} \theta^i_C (\nabla f_i(x) - \nabla f_i(y)), \sum_{i \in C} \theta^i_C (\nabla f_i(x) - \nabla f_i(y)) \right \rangle \nonumber  \\ 
	&=& \frac{1}{n^2} \sum_C p_C \sum_{i, j \in C} \theta^i_C \theta^j_C \left \langle  (\nabla f_i(x) - \nabla f_i(y)),  (\nabla f_j(x) - \nabla f_j(y)) \right \rangle \nonumber \\ 
	&=& \frac{1}{n^2} \sum_{i, j=1}^n \sum_{C: i, j \in C} p_C \theta^i_C \theta^j_C \left \langle   (\nabla f_i(x) - \nabla f_i(y)),  (\nabla f_j(x) - \nabla f_j(y)) \right \rangle \label{eq:nicesL}\\ 
	&\leq&  \frac{1}{n^2} \sum_{i, j=1}^n \sum_{C: i, j \in C} p_C \theta^i_C \theta^j_C L_iL_j \|x-y\|^2. \nonumber 
	\end{eqnarray}
	
	This along with (\ref{eq:boundL2}) implies the result. If $\theta^i_C = \frac{1}{p_i}$, then 
	$$
	\sum_{C: i, j \in C} p_C \theta^i_C \theta^j_C = \frac{{\bf P}_{ij}}{p_ip_j}.  
	$$

	\vskip 2mm 
	
	(ii) From (\ref{eq:L12beta}), we have 
	\begin{eqnarray*}
		\mathbb{E}\left[\left\|\frac{1}{n}({\bf G}(x) - {\bf G}(y)){\theta}_{S}{\bf I}_{S}e \right\|^2 \right] &\leq& \frac{1}{n^2} \sum_{i=1}^n \beta_i \|\nabla f_i(x) - \nabla f_i(y) \|^2 \\ 
		&\leq& \frac{1}{n^2} \sum_{i=1}^n \beta_i L_i^2 \|x-y\|^2. 
	\end{eqnarray*}
	
	This along with (\ref{eq:boundL2}) implies the result. 
	
\end{proof}

\begin{lemma}\label{lm:L3nsgs}
	Let $S$ be a proper sampling, $\theta^i_S = 1/p_i$. Let $f_i$ be $L_i$-smooth. \\ 
	(i) For $\tau$-nice sampling $S$, the ${\cal L}_3$ in Assumption \ref{as:expL3} satisfies 
	$$
	{\cal L}_3 \leq \frac{n-\tau}{\tau (n-1)}\cdot \frac{1}{n}\sum_{i=1}^n L_i^2. 
	$$
	(ii) For group sampling $S$, denote the isolated index set as ${\cal I}$, then we have 
	$$
	{\cal L}_3 \leq \frac{1}{n^2} \left(  \sum_{i\in {\cal I}} (\frac{1}{p_i} -1)L_i^2 + \sum_{i \notin {\cal I}} \frac{1}{p_i} L_i^2  \right). 
	$$
	
\end{lemma}

\begin{proof}
	
	(i) From (\ref{eq:nicesL}), we have 
	\begin{eqnarray*}
		&& \mathbb{E}\left[\left\|\frac{1}{n}({\bf G}(x) - {\bf G}(y)){\theta}_{S}{\bf I}_{S}e \right\|^2 \right] \\ 
		&=&  \frac{1}{n^2} \sum_{i, j=1}^n \sum_{C: i, j \in C} p_C \theta^i_C \theta^j_C \left \langle   (\nabla f_i(x) - \nabla f_i(y)),  (\nabla f_j(x) - \nabla f_j(y)) \right \rangle \\ 
		&=&  \frac{1}{n^2} \sum_{i, j=1}^n \frac{{\bf P}_{ij}}{p_ip_j} \left \langle   (\nabla f_i(x) - \nabla f_i(y)),  (\nabla f_j(x) - \nabla f_j(y)) \right \rangle. 
	\end{eqnarray*}
	
	For $\tau$-nice sampling, ${\bf P}_{ij} = \frac{\tau(\tau-1)}{n(n-1)}$ for $i\neq j$ and $p_i = \frac{\tau}{n}$. Hence 
	\begin{eqnarray*}
		&& \mathbb{E}\left[\left\|\frac{1}{n}({\bf G}(x) - {\bf G}(y)){\theta}_{S}{\bf I}_{S}e \right\|^2 \right] \\ 
		&=& \frac{1}{n^2} \sum_{i \neq j} \frac{n(\tau-1)}{\tau(n-1)} \left \langle   (\nabla f_i(x) - \nabla f_i(y)),  (\nabla f_j(x) - \nabla f_j(y)) \right \rangle + \frac{1}{n^2}  \sum_{i=1}^n \frac{n}{\tau} \|\nabla f_i(x) -\nabla f_i(y) \|^2 \\ 
		&=& \frac{1}{n^2} \sum_{i, j=1}^n \frac{n(\tau-1)}{\tau(n-1)} \left \langle   (\nabla f_i(x) - \nabla f_i(y)),  (\nabla f_j(x) - \nabla f_j(y)) \right \rangle + \frac{1}{n^2}  \sum_{i=1}^n \frac{n(n-\tau)}{\tau(n-1)} \|\nabla f_i(x) -\nabla f_i(y) \|^2 \\ 
		&=& \|\nabla f(x) - \nabla f(y)\|^2 + \frac{1}{n^2}  \sum_{i=1}^n \frac{n(n-\tau)}{\tau(n-1)} \|\nabla f_i(x) -\nabla f_i(y) \|^2. 
	\end{eqnarray*}
	
	From the above equality and (\ref{eq:forL2-1}), we have 
	\begin{eqnarray*}
		&& \mathbb{E}\left[\left\|\frac{1}{n}({\bf G}(x) - {\bf G}(y)){\theta}_{S}{\bf I}_{S}e - \frac{1}{n}({\bf G}(x) - {\bf G}(y))e\right\|^2 \right] \nonumber \\ 
		&=& \frac{1}{n^2}  \sum_{i=1}^n \frac{n(n-\tau)}{\tau(n-1)} \|\nabla f_i(x) -\nabla f_i(y) \|^2 \\ 
		&\leq& \frac{1}{n^2}  \sum_{i=1}^n \frac{n(n-\tau)}{\tau(n-1)} L_i^2 \|x-y\|^2. 
	\end{eqnarray*}

	(ii) From (\ref{eq:groupsL}) and (\ref{eq:forL2-1}), we have 
	\begin{eqnarray*}
		&& \mathbb{E}\left[\left\|\frac{1}{n}({\bf G}(x) - {\bf G}(y)){\theta}_{S}{\bf I}_{S}e - \frac{1}{n}({\bf G}(x) - {\bf G}(y))e\right\|^2 \right] \nonumber \\ 
		&=& - \frac{1}{n^2} \sum_{j =1}^t \left \| \sum_{i \in C_j} (\nabla f_i(x) - \nabla f_i(y)) \right \|^2 \nonumber + \frac{1}{n^2} \sum_{i=1}^n \frac{1}{p_i} \|\nabla f_i(x) - \nabla f_i(y)\|^2 \\ 
		&\leq& \frac{1}{n^2} \sum_{i \in {\cal I}} (\frac{1}{p_i} -1)\|\nabla f_i(x) - \nabla f_i(y) \|^2 + \frac{1}{n^2} \sum_{i \notin {\cal I}} \frac{1}{p_i} \|\nabla f_i(x) - \nabla f_i(y)\|^2 \\ 
		&\leq& \frac{1}{n^2} \left(  \sum_{i\in {\cal I}} (\frac{1}{p_i} -1)L_i^2 + \sum_{i \notin {\cal I}} \frac{1}{p_i} L_i^2   \right) \|x-y\|^2. 
	\end{eqnarray*}
	
\end{proof}

\begin{proposition}\label{pro:L3gs}
	Let $f_i$ be $L_i$-smooth. For group sampling $S$ with $\mathbb{E}[|S|] = \sum_{i=1}^np_i = \tau$, by choosing $\min\{  q_i, 1  \} \leq p_i \leq 1$, where $q_i = L_i\tau / \sum_{i=1}^nL_i$, we have 
	$$
	{\cal L}_3 \leq \frac{{\bar L}^2}{\tau}. 
	$$
\end{proposition}

\begin{proof}
	
	Denote the isolated index set as ${\cal I}$, and $T = \{  i | q_i>1  \}$. If $q_i >1$, then $p_i =1$, and $i$ must belong to ${\cal I}$. Hence $T \subset {\cal I}$. From Lemma \ref{lm:L3nsgs} (ii), we have 
	\begin{eqnarray*}
		{\cal L}_3 &\leq& \frac{1}{n^2} \left(  \sum_{i\in {\cal I}} (\frac{1}{p_i} -1)L_i^2 + \sum_{i \notin {\cal I}} \frac{1}{p_i} L_i^2  \right) \\ 
		&=& \frac{1}{n^2} \left(  \sum_{i\in T} (\frac{1}{p_i} -1)L_i^2 + \sum_{i\in {\cal I} - T} (\frac{1}{p_i} -1)L_i^2 +  \sum_{i \notin {\cal I}} \frac{1}{p_i} L_i^2  \right)\\ 
		&=& \frac{1}{n^2} \left( \sum_{i\in {\cal I} - T} (\frac{1}{p_i} -1)L_i^2 +  \sum_{i \notin {\cal I}} \frac{1}{p_i} L_i^2  \right) \\ 
		&\leq& \frac{1}{n^2} \sum_{i \notin T}\frac{1}{p_i}L_i^2 \\ 
		&=& \frac{1}{n^2} \sum_{i \notin T} \frac{L_i \sum_{i=1}^nL_i}{\tau} \\
		&\leq& \frac{{\bar L}^2}{\tau}. 
	\end{eqnarray*}
	
\end{proof}

\begin{lemma}\label{lm:L3eso}
	Let $S$ be a proper sampling, and $\theta^i_S = 1/p_i$. Let $\phi_i$ be $1/\gamma$-smooth. If the ESO inequality (\ref{eq:ESOfirst}) holds, then the ${\cal L}_3$ in Assumption \ref{as:expL3} satisfies
	$$
	{\cal L}_3 \leq \frac{1}{n^2 \gamma^2} \sum_{i=1}^n \frac{v_i\|{\bf A}_i\|^2}{p_i}. 
	$$
\end{lemma}

\begin{proof}
	
	From (\ref{eq:forL3exo}), we have 
	\begin{eqnarray*}
		&& \mathbb{E}\left[\left\|\frac{1}{n}({\bf G}(x) - {\bf G}(y)){\theta}_{S}{\bf I}_{S}e \right\|^2 \right] \\ 
		&\leq& \frac{1}{n^2} \sum_{i=1}^n p_iv_i\cdot \frac{1}{p_i^2} \left\| \nabla\phi_i({\bf A}_i^\top x)-\nabla \phi_i({\bf A}_i^\top y)\right\|^2 \\ 
		&\leq& \frac{1}{n^2} \sum_{i=1}^n \frac{v_i}{p_i} \cdot \frac{\|{\bf A}_i\|^2}{\gamma^2} \|x-y\|^2. 
	\end{eqnarray*}
	
	This along with (\ref{eq:boundL2}) implies the result. 
	
\end{proof}

\subsection{Estimation for sampling with replacement $S$ }

\begin{lemma}\label{lm:L3sp}	
	Let ${\tilde p}_i >0$ in distribution ${\tilde D}$ for the sampling with replacement $S$, and $f_i$ be $L_i$-smooth. Then the ${\cal L}_3$ in Assumption \ref{as:expL3} satisfies 
	$$
	{\cal L}_3 \leq  \frac{1}{n^2\tau} \sum_{i=1}^n \frac{L_i^2}{{\tilde p}_i}.   
	$$
\end{lemma}

\begin{proof}
	From (\ref{eq:forL3sp}),  we have 
	\begin{eqnarray*}
		&& \mathbb{E}\left[\left\|\frac{1}{n}({\bf G}(x) - {\bf G}(y)){\theta}_{S}{\bf I}_{S}e - \frac{1}{n}({\bf G}(x) - {\bf G}(y))e\right\|^2 \right] \nonumber \\
		&\leq&  \frac{\tau}{\tau^2} {\mathbb{E}}_{i \sim {\tilde {\cal D}}}\left[ \left\| \frac{1}{n{\tilde p}_i}(\nabla f_i(x) - \nabla f_i(y)) \right\|^2 \right] \\ 
		&\leq& \frac{1}{\tau} \sum_{i=1}^n \frac{1}{n^2 {\tilde p}_i} \|\nabla f_i(x) - \nabla f_i(y)\|^2 \\ 
		&\leq&  \frac{1}{n^2\tau} \sum_{i=1}^n \frac{L_i^2}{{\tilde p}_i} \|x - y\|^2. 
	\end{eqnarray*}
	
\end{proof}

\begin{proposition}\label{pro:L3sp}
	Let $f_i$ be $L_i$-smooth. For the sampling with replacement, let the number of copies be $\tau$. By choosing ${\tilde p}_i = L_i/\sum_{i=1}^nL_i$, we have 
	$$
	{\cal L}_3 \leq \frac{{\bar L}^2}{\tau}. 
	$$
\end{proposition}

\begin{proof}
	For the following linearly constrained minimization problem 
	\begin{equation}
	\begin{array}{rcl}
	\min & &  \frac{1}{n^2\tau} \sum_{i=1}^n \frac{L_i^2}{{\tilde p}_i} \\
	{\rm s.t.\ }& &  \sum_{i=1}^n {\tilde p}_i = 1, \   0< p_i \leq 1, \forall i, 
	\end{array} \nonumber 
	\end{equation}
	it is not hard to see that the optimal solution is ${\tilde p}_i = \frac{L_i}{\sum_{i=1}^nL_i}$. Hence, by Lemma \ref{lm:L3sp}, we have ${\cal L}_3 \leq \frac{{\bar L}^2}{\tau}$. 
	
\end{proof}

\section{Efficient Implementation}\label{sec:EI}
The delayed update is a standard technique for more efficiency when the Jacobian matrix $\bG(x)$ is sparse.  For the sake of completeness, we provide details in the case when 
$$
\psi(x)\equiv\frac{\lambda_2}{2}\|x\|^2+\lambda_1\|x\|_1,
$$
for some $\lambda_2>0$.
For Algorithm~\ref{alg:lsvrg}, the $i$th coordinate of the iterates $\{x^k\}$ satisfy:
\begin{align}\label{a:xkp}
x_i^{k+1}=\arg\min\left\{ \frac{\lambda_2}{2}x^2 +\lambda_1 |x|+\frac{1}{2\eta}\left(x-x^k_i+\eta g_i^k\right)^2  \right\}.
\end{align}
Let $t_0<t_1$ be two positive integers.
Suppose that
$$
g_i^k=\hat g_i,\enspace \forall k=t_0,\dots,t_1-1,
$$
then the value of $x_{i}^{t_1}$ can be obtained without explicitly computing the value of $x_i^{t_0+1},\dots,x_i^{t_1-1}$. 
 The details of computation can be found in~\cite{SPDC}. For convenience we give the pseudocode in Algorithm~\ref{alg:du}, so that
$$
x_i^{t_1}=\mathrm{delayed\_update}(t_0, t_1, \hat g_i, x_i^{t_0}, \eta).
$$
Note that the complexity of Algorithm~\ref{alg:du} is $O(\log(t_1-t_0))$~\cite{SPDC} while direct computation of $x_i^{t_1}$   from 
$x_i^{t_0}$ yields a time
 complexity $O(t_1-t_0)$. This is how the computation load can be reduced when $\bG(x)$ is sparse.
\begin{algorithm}[h]
	\caption{ $\tilde x=$delayed\_update($t_0$, $t_1$, $u$, $x$, $\eta$)}
	\label{alg:du}
	\begin{algorithmic}[1]
	\If{$t_1=t_0$} $\tilde x=x$;
	return;
	\EndIf
	        \State $\alpha={(1+\eta\lambda_2)^{t_0-t_1} }$
		\If{ $x=0$}  
		   \If{$\lambda_1+u<0$}
		    $\tilde x=\alpha x-(1-\alpha)(u+\lambda_1)/\lambda_2$
		    \EndIf
		    \If{$\lambda_1+u>0$}
		    $\tilde x=\alpha x-(1-\alpha)(u-\lambda_1)/\lambda_2$
		    \EndIf
		    \If{$\lambda_1+u=0$}
		      $\tilde x=0$
		    \EndIf
		\Else 
		   \If{$x>0$}
	          	\If{$\lambda_1+u\leq 0$}
	                 	$\tilde x=\alpha x-(1-\alpha)(u+\lambda_1)/\lambda_2$
		        \Else 
		               ~ $t=t_0-\log(\frac{1+\lambda_2 x}{\lambda_1+u})\log(1+\eta\lambda_2)$
		                \If{$t<t_1$} \\ \qquad\qquad\qquad\qquad $t'=\lfloor t \rfloor $\\ \qquad\qquad\qquad\qquad$\alpha'={(1+\eta\lambda_2)^{t_0-t'} }$\\
		                \qquad\qquad\qquad\qquad $x'=\alpha' x-(1-\alpha')(u+\lambda_1)/\lambda_2$
		                \\\qquad\qquad\qquad\qquad $x''=\arg\min\left\{ \frac{\lambda_2}{2}x^2 +\lambda_1 |x|+\frac{1}{2\eta}\left(x-x'+u\right)^2  \right\}$\\
		              \qquad\qquad\qquad\qquad  $\tilde x=\mathrm{delayed\_update}( t'+1, t_1, u, x'',\eta)$
		                 \Else \\
		       \qquad\qquad\qquad\qquad  $\tilde x=\alpha x-(1-\alpha)(u+\lambda_1)/\lambda_2$
		                 \EndIf
		         \EndIf
		         \Else \\
		          \qquad\quad  $\tilde x=-\mathrm{delayed\_update}( t_0, t_1, -u, -x, \eta)$
		   \EndIf
		\EndIf
		\State Output $\tilde x$
			\end{algorithmic}
\end{algorithm}

\subsection{Efficient implementation for L-Katyusha}
For Algorithm~\ref{alg:lkatyusha}, the $i$th coordinate of the iterates $\{x^k, y^k, z^k\}$ satisfy:
$$\left\{\begin{array}{l}
x_i^k = \theta_1 z_i^k + \theta_2 w_i^k + (1-\theta_1 -\theta_2)y_i^k\\
z_i^{k+1} = 
\arg\min\left\{ \frac{\lambda_2}{2}x^2 +\lambda_1 |x|+\frac{(1+\eta\sigma_1)L}{2\eta}\left(x-\frac{\eta \sigma_1 x^k_i+z_i^k}{1+\eta \sigma_1}+ \frac{\eta}{(1+\eta\sigma_1)L g_i^k} \right)^2  \right\}\\
y_i^{k+1} = x_i^k + \theta_1 (z_i^{k+1} - z_i^k)
\end{array}\right.$$
We eliminate $x^k_i$ and obtain:
$$\left\{\begin{array}{l}
z_i^{k+1} = 
\arg\min\left\{ \frac{\lambda_2}{2}x^2 +\lambda_1 |x|+\frac{(1+\eta\sigma_1)L}{2\eta}\left(x-\frac{
(\eta\sigma_1\theta_1+1) z_i^k + \eta\sigma_1 \theta_2 w_i^k + \eta\sigma_1(1-\theta_1 -\theta_2)y_i^k}{1+\eta \sigma_1}+ \frac{\eta g_i^k}{(1+\eta\sigma_1)L} \right)^2  \right\}\\
y_i^{k+1} = \theta_1 z_i^{k+1} + \theta_2 w_i^k + (1-\theta_1 -\theta_2)y_i^k\end{array}\right.$$
\begin{itemize}
\item
If $\lambda_1=0$, then the above system can be written as
 $$\left\{\begin{array}{l}
z_i^{k+1} = \frac{
(\eta\sigma_1\theta_1+1) L  }{
\eta\lambda_1 + L(1+\eta\sigma_1)} z_i^k + \frac{\eta\sigma_1(1-\theta_1 -\theta_2)L} {
\eta\lambda_1 + L(1+\eta\sigma_1)}y_i^k + \frac{ \eta\sigma_1 \theta_2  Lw_i^k -\eta g_i^k }{
\eta\lambda_1 + L(1+\eta\sigma_1)}\\
y_i^{k+1} =  \frac{\theta_1
(\eta\sigma_1\theta_1+1) L  }{
\eta\lambda_1 + L(1+\eta\sigma_1)}  z_i^{k}  + \left(1-\theta_1 -\theta_2 +\frac{\theta_1\eta\sigma_1(1-\theta_1 -\theta_2)L} {
\eta\lambda_1 + L(1+\eta\sigma_1)}\right) y_i^k +\theta_2 w_i^k+\frac{ \theta_1(\eta\sigma_1 \theta_2  Lw_i^k -\eta g_i^k) }{
\eta\lambda_1 + L(1+\eta\sigma_1)}\end{array}\right.$$
Then if $$
g_i^k=\hat g_i, w_i^k=\hat w_i \enspace \forall k=t_0,\dots,t_1-1,
$$
$z_i^{t_1}$ and $y_i^{t_1}$ can be computed by
\begin{align}\label{a:z}
\begin{pmatrix}
z_i^{t_1}\\y_{i}^{t_1}
\end{pmatrix}=A^{t_1-t_0} \begin{pmatrix}
z_i^{t_0}\\y_{i}^{t_0}
\end{pmatrix}+ \left( \sum_{s=0}^{t_1-t_0-1} A^s \right) b
\end{align}
with $$
A=\begin{pmatrix}
\frac{
(\eta\sigma_1\theta_1+1) L  }{
\eta\lambda_1 + L(1+\eta\sigma_1)}    & \frac{\eta\sigma_1(1-\theta_1 -\theta_2)L} {
\eta\lambda_1 + L(1+\eta\sigma_1)} \\
\frac{\theta_1
(\eta\sigma_1\theta_1+1) L  }{
\eta\lambda_1 + L(1+\eta\sigma_1)} & 1-\theta_1 -\theta_2+\frac{\theta_1\eta\sigma_1(1-\theta_1 -\theta_2)L} {
\eta\lambda_1 + L(1+\eta\sigma_1)}
\end{pmatrix},\enspace b=\begin{pmatrix}
+ \frac{ \eta\sigma_1 \theta_2  Lw_i^k -\eta g_i^k }{
\eta\lambda_1 + L(1+\eta\sigma_1)}\\
\frac{ \theta_1(\eta\sigma_1 \theta_2  Lw_i^k -\eta g_i^k) }{
\eta\lambda_1 + L(1+\eta\sigma_1)}
\end{pmatrix}.
$$
It is clear that~\eqref{a:z} can be computed in $O(\log(t_1-t_0))$ time.
\item  If $\lambda_1>0$, we need to require $\sigma_1=0$ to 
have
reduced computation load.  In this case, we have a simplified recursive relation:
$$\left\{\begin{array}{l}
z_i^{k+1} = 
\arg\min\left\{ \frac{\lambda_2}{2}x^2 +\lambda_1 |x|+\frac{L}{2\eta}\left(x-z_i^k + \frac{\eta g_i^k}{L} \right)^2  \right\}\\
y_i^{k+1} = \theta_1 z_i^{k+1} + \theta_2 w_i^k + (1-\theta_1 -\theta_2)y_i^k\end{array}\right.$$

Suppose that $$
g_i^k=\hat g_i, w_i^k=\hat w_i \enspace \forall k=t_0,\dots,t_1-1,
$$
Since the $z_i^k$ follows the same recursive formula as~\eqref{a:xkp}, we can apply 
Algorithm~\ref{alg:du} to compute $z_i^{t_1}$, i.e., 
$$
z_i^{t_1}=\mathrm{delayed\_update}(t_0, t_1, \hat g_i, z_i^{t_0},\eta/L).
$$
Let $\theta_3=1-\theta_1-\theta_2$, then for any integers $k\geq 0$ land $s>0$,
$$
y_i^{k+s}=\theta_1\left(z_i^{k+s}+\theta_3 z_i^{k+s-1}+\dots+\theta_3^{s-1}z_i^{k+1}\right)
+\theta_2\left(1+\theta_3+\dots+\theta_3^{s-1}\right) w_i^k +\theta_3^s y_i^k.
$$
If 
\begin{align}\label{a:ffd}
z_i^{k+l}=q^l(z_i^k+h)-h,\enspace \forall l=1,\dots,s,
\end{align}
for some $q>0$ and $h\in \R$, then
\begin{align}\label{a:tisd}
y_i^{k+s}&=\theta_1\left(
\sum_{l=1}^s \theta_3^{s-l} \left(q^l(z_i^k+h)-h\right)\right)+\frac{\theta_2(1-\theta_3^s)}{1-\theta_3}
w_i^k+\theta_3^s y_i^k\\ \notag
&=\theta_1
\theta_3^s (z_i^k+h) \sum_{l=1}^s \left(q\theta_3^{-1}\right)^l+\frac{(\theta_2 w_i^k -\theta_1 h)(1-\theta_3^s)}{1-\theta_3}
+\theta_3^s y_i^k.
\end{align}
Based on the above computation we can write down the efficient implementation for L-Katyusha, given 
in Algorithm~\ref{alg:ecLK}. And we have:
$$
(y_i^{t_1},z_i^{t_1})=\mathrm{delayed\_update2}( t_0, t_1, \hat g_i, y_i^{t_0}, z_i^{t_1}, \hat w_i, \eta/L).
$$
It is easy to check that the computational complexity of Algorithm~\ref{alg:ecLK} is $O(\log(t_1-t_0))$.
\begin{algorithm}[h]
	\caption{ $(\tilde y,\tilde z)=$delayed\_update2($t_0$, $t_1$, $u$, $y$,  $z$,   $w$, $\eta$)}
	\label{alg:ecLK}
	\begin{algorithmic}[1]
	\If{$t_1=t_0$} $\tilde z=z$; $\tilde y=y$; 
	return;
	\EndIf
	        \State $\alpha={(1+\eta\lambda_2)^{t_0-t_1} }$
	        \State $q=1/(1+\eta\lambda_2)$
		\If{ $x=0$}  
		   \If{$\lambda_1+u<0$} $\tilde z=\alpha z-(1-\alpha)(u+\lambda_1)/\lambda_2$; $h=(u+\lambda_1)/\lambda_2$
		    \EndIf
		    \If{$\lambda_1+u>0$}
		    $\tilde x=\alpha z-(1-\alpha)(u-\lambda_1)/\lambda_2$; $h=(u-\lambda_1)/\lambda_2$
		    \EndIf
		    \If{$\lambda_1+u=0$}
		      $\tilde x=0$; $q=0$; $h=0$
		    \EndIf \\
		 $\qquad\tilde y=\theta_1
\theta_3^{t_1-t_0} (z+h) \sum_{l=1}^{t_1-t_0} \left(q\theta_3^{-1}\right)^l+\frac{(\theta_2 w -\theta_1 h)(1-\theta_3^{t_1-t_0})}{1-\theta_3}
+\theta_3^{t_1-t_0} y.$
		\Else  \\
		  $\qquad h=(u+\lambda_1)/\lambda_2$
		   \If{$x>0$}
	          	\If{$\lambda_1+u\leq 0$}\\
	                 	$\qquad\qquad\tilde z=\alpha z-(1-\alpha)(u+\lambda_1)/\lambda_2$\\ $\qquad\qquad\tilde y=\theta_1
\theta_3^{t_1-t_0} (z+h) \sum_{l=1}^{t_1-t_0} \left(q\theta_3^{-1}\right)^l+\frac{(\theta_2 w -\theta_1 h)(1-\theta_3^{t_1-t_0})}{1-\theta_3}
+\theta_3^{t_1-t_0} y.$
		        \Else 
		               ~ $t=t_0-\log(\frac{1+\lambda_2 z}{\lambda_1+u})\log(1+\eta\lambda_2)$
		                \If{$t<t_1$} \\ \qquad\qquad\qquad\qquad $t'=\lfloor t \rfloor $\\ \qquad\qquad\qquad\qquad$\alpha'={(1+\eta\lambda_2)^{t_0-t'} }$\\
		                \qquad\qquad\qquad\qquad $z'=\alpha' z-(1-\alpha')(u+\lambda_1)/\lambda_2$\\
		                 \qquad\qquad\qquad\qquad$y'=\theta_1
\theta_3^{t'-t_0} (z+h) \sum_{l=1}^{t'-t_0} \left(q\theta_3^{-1}\right)^l+\frac{(\theta_2 w -\theta_1 h)(1-\theta_3^{t'-t_0})}{1-\theta_3}
+\theta_3^{t'-t_0} y.$
		                \\\qquad\qquad\qquad\qquad $z''=\arg\min\left\{ \frac{\lambda_2}{2}x^2 +\lambda_1 |x|+\frac{1}{2\eta}\left(x-z'+u\right)^2  \right\}$\\
		                \qquad\qquad\qquad\qquad $y''= \theta_1 z'' + \theta_2 w + (1-\theta_1 -\theta_2)y'$\\
		              \qquad\qquad\qquad\qquad  $(\tilde y, \tilde z)=\mathrm{delayed\_update2}( t'+1, t_1, u, y'',z'',w, \eta)$
		                 \Else \\
		       \qquad\qquad\qquad\qquad  $ \tilde z=\alpha z-(1-\alpha)(u+\lambda_1)/\lambda_2$\\
		        \qquad\qquad\qquad\qquad  $\tilde y=\theta_1
\theta_3^{t_1-t_0} (z+h) \sum_{l=1}^{t_1-t_0} \left(q\theta_3^{-1}\right)^l+\frac{(\theta_2 w -\theta_1 h)(1-\theta_3^{t_1-t_0})}{1-\theta_3}
+\theta_3^{t_1-t_0} y.$
		                 \EndIf
		         \EndIf
		         \Else \\
		          \qquad\quad  $(\tilde y, \tilde z)=-\mathrm{delayed\_update2}( t_0, t_1, -u, -y,-z,w, \eta)$
		   \EndIf
		\EndIf
		\State Output $(\tilde y,\tilde z)$
			\end{algorithmic}
\end{algorithm}
\end{itemize}

 \subsection{Efficient implementation for Katyusha}
As mentioned, one major difference between the original Katyusha~\cite{Katyusha} and our loopless variant is in the update of reference point. Let $m$ be the size of inner loop in Katyusha. After $s$ outer loops, Katyusha requires to compute a convex combination of $y^k$:
$$
\tilde x^{s+1}=\frac{\sum_{j=0}^{m-1} \theta^j y^{sm+j+1}}{\sum_{j=0}^{m-1}\theta^j}
$$
for some $\theta>1$.  For any $1\leq t\leq m$, define:
$$
\hat x^{t}=\frac{\sum_{j=0}^{t-1} \theta^j y^{sm+j+1}}{\sum_{j=0}^{t-1}\theta^j}.
$$
Then 
$$
\hat x^{t+1}=\left( 1-\frac{\theta^t-\theta^{t-1}}{\theta^{t}-1} \right)\hat x^t+\frac{\theta^t-\theta^{t-1}}{\theta^{t}-1} y^{sm+t},
$$
Suppose that\footnote{Recall that our $g^k$ corresponds to $\tilde \nabla_k$ in Katyusha~\cite{Katyusha}.} $$
g_i^{sm+k}=\hat g_i, \enspace \forall k=t_0,\dots,t_1-1.
$$
In order to compute $\hat x^{t_1}$ from $\hat x^{t_0}$ in $O(\log(t_1-t_0))$, we consider the case when $\sigma_1=0$. First note that
$$
\hat x^{t_1}=\left( 1-\frac{\theta^{t_1}-\theta^{t_0}}{\theta^{t_1}-1} \right)\hat x^{t_0}+\frac{y^{sm+t_0+1}+\theta y^{sm+t_0+2}+\dots+\theta^{t_1-t_0-1}y^{sm+t_1}}{\theta^{t_1-t_0-1}+\dots+\theta^{-t_0}}
$$
Assume that
\begin{align}\label{a:ffd}
z_i^{sm+t_0+l}=q^l(z_i^{sm+t_0}+h)-h,\enspace \forall l=1,\dots,t_1-t_0.
\end{align}
The same as~\eqref{a:tisd} here we have
\begin{align}\notag
&y^{sm+t_0+1}+\theta y^{sm+t_0+2}+\dots+\theta^{t_1-t_0-1}y^{sm+t_1}
\\&= \sum_{k=1}^{t_1-t_0} \theta^{k-1} \left(\theta_1
\theta_3^k (z_i^{sm+t_0}+h) \sum_{l=1}^k \left(q\theta_3^{-1}\right)^l+\frac{(\theta_2 w_i^k -\theta_1 h)(1-\theta_3^k)}{1-\theta_3}
+\theta_3^k y_i^{sm+t_0}\right). \label{a:dfe}
\end{align}
After rearranging, we can compute~\eqref{a:dfe} and then $x^{t_1}$ from $x^{t_0}$ in $O(\log(t_1-t_0))$ time when~\eqref{a:ffd} holds. Then we can update $\hat x^{t}$ in the same efficient way as we update
the three inner iterates $\{x^k,y^k,z^k\}$ with Algorithm~\ref{alg:ecLK}. We omit further details as this
is not the main topic of our paper. However, the above discussion shows that the implementation
 of original
Katyusha is more complicated than our loopless variant, due to the use of weighted average as reference point.

\section{More Experimental Results}
\begin{figure}[!ht]
   \subfigure[$\lambda_2=1$]{ \label{5d1}  \includegraphics[scale=0.38]{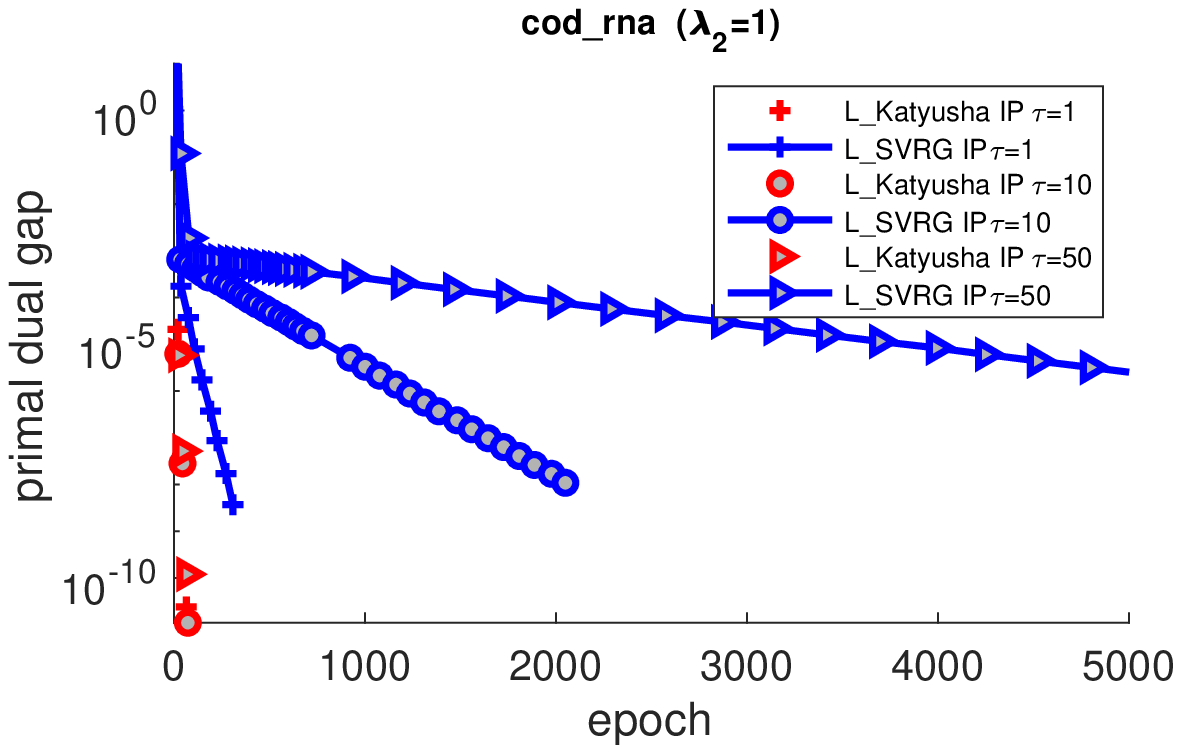}}
  \subfigure[$\lambda_2=10^{-2}$]{\label{5d2}\includegraphics[scale=0.38]{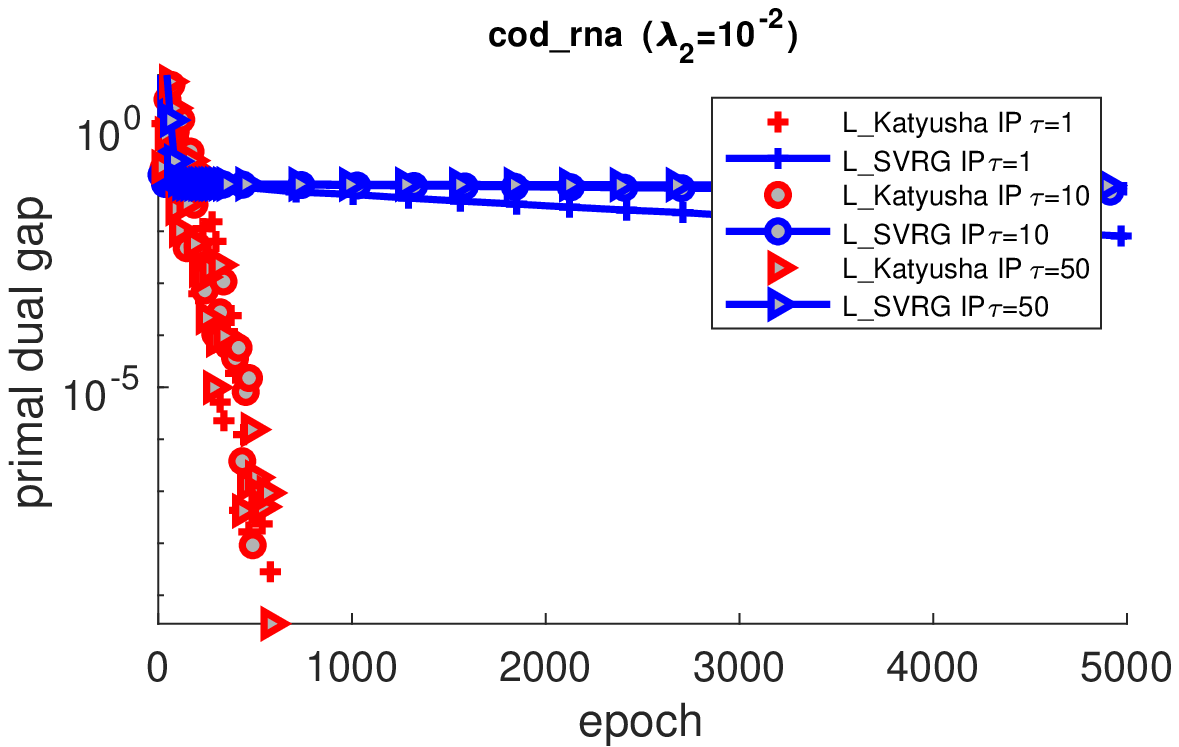}}
  \subfigure[$\lambda_2=10^{-4}$]{\label{5d3}\includegraphics[scale=0.38]{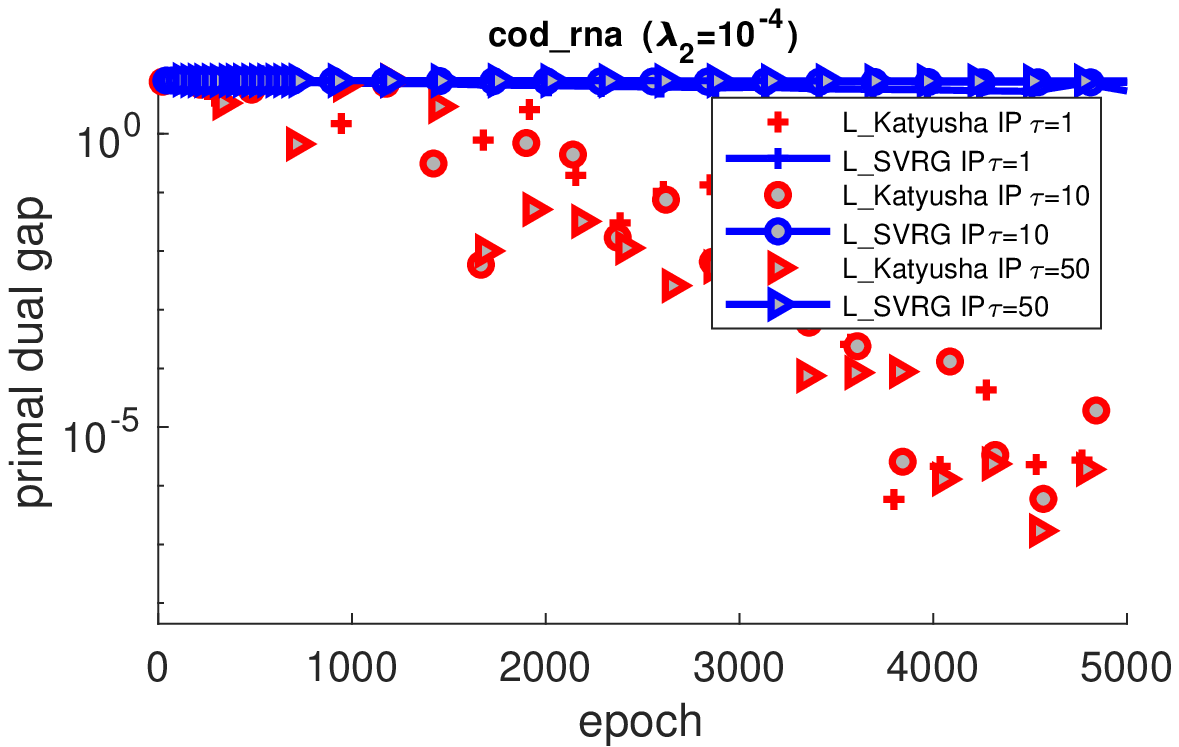}}
\caption{L-SVRG V.S. L-Katyusha, cod-rna}\label{fig5}
 \end{figure}

\begin{figure}[!ht]
 \subfigure[uniform V.S. IP]{\label{3d1}\includegraphics[scale=0.38]{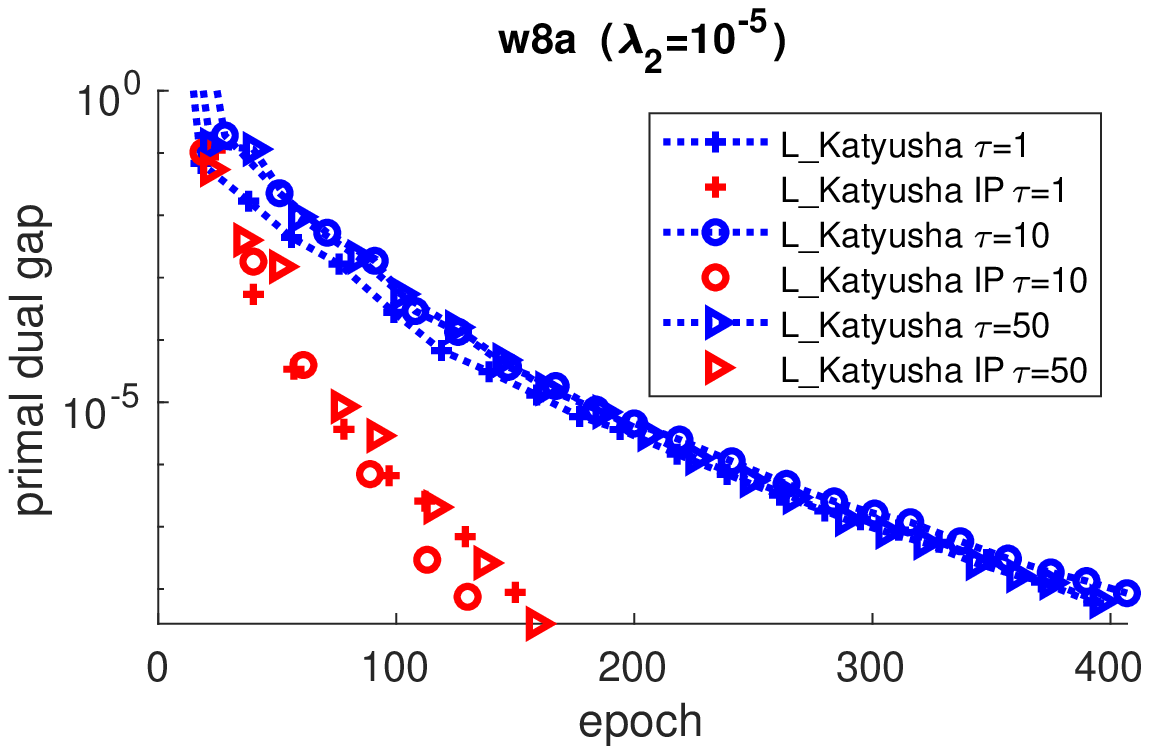}}
     \subfigure[different $p$, epoch]{\label{3d2} \includegraphics[scale=0.38]{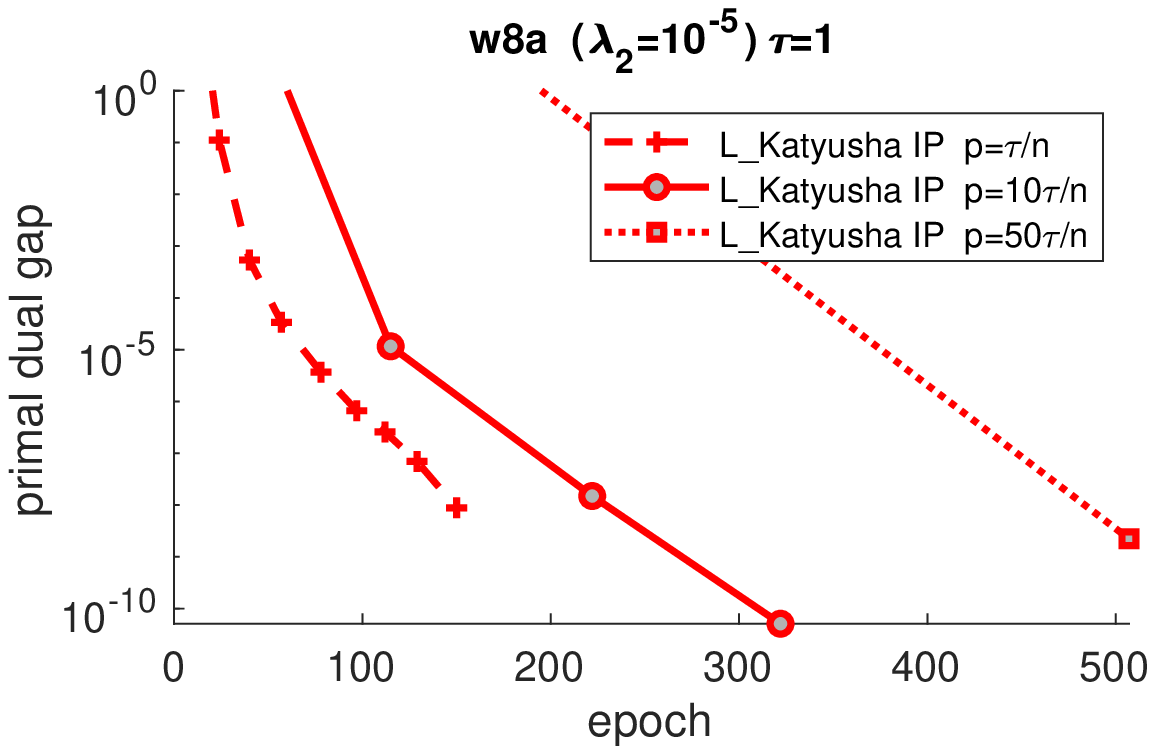}}
    \subfigure[different $p$, time]{\label{3d3}  \includegraphics[scale=0.38]{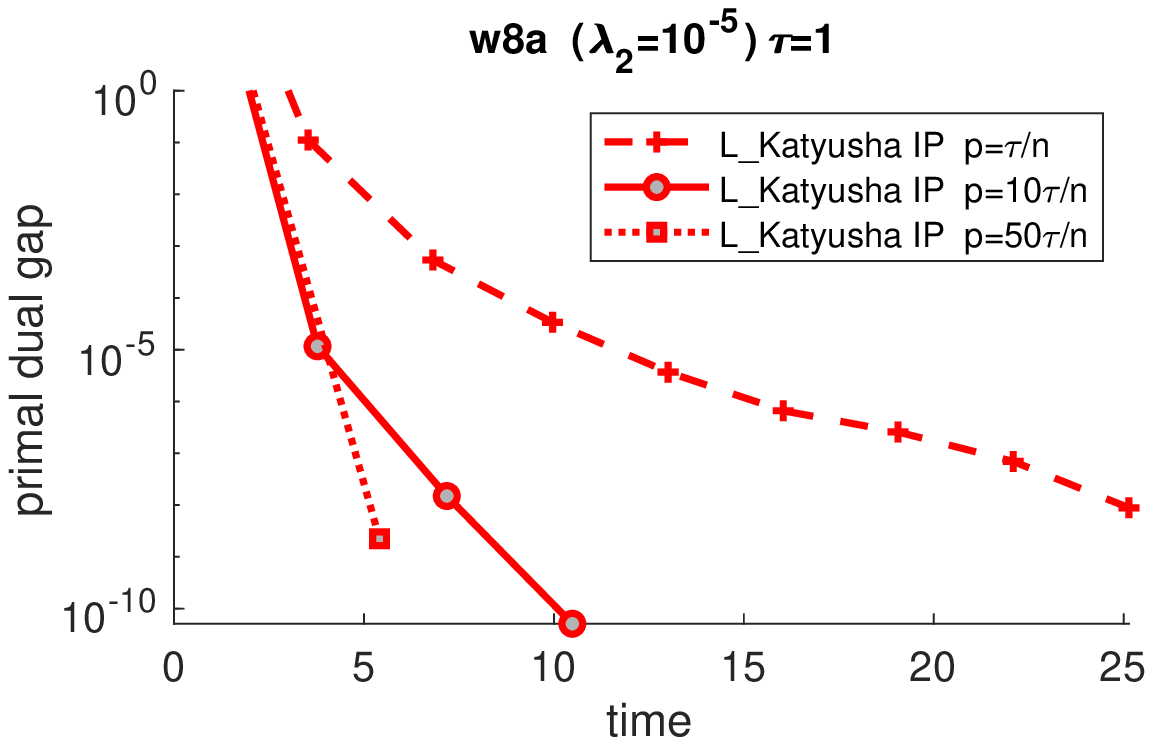}}
   \caption{w8a}\label{fig3}
 \end{figure}

 \begin{figure}[!ht]
 \subfigure[$\tau=1$]{\label{11d1}\includegraphics[scale=0.38]{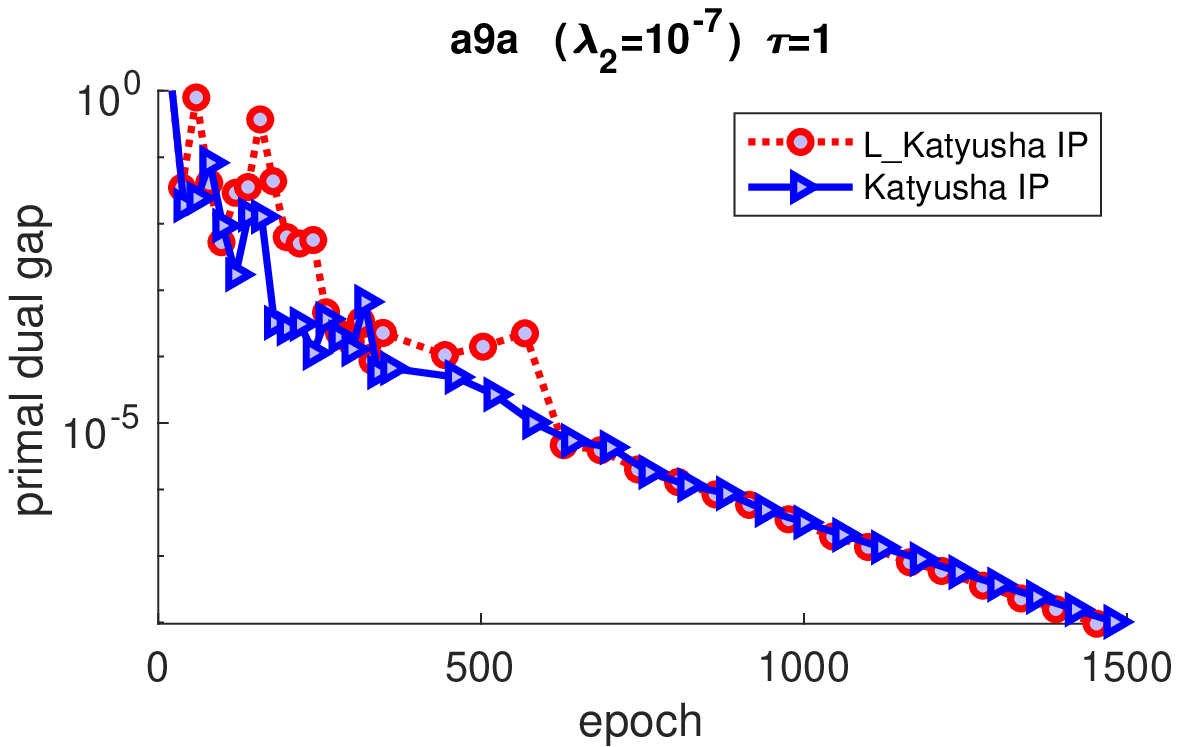}}
  \subfigure[$\tau=10$]{  \label{11d2}  \includegraphics[scale=0.38]{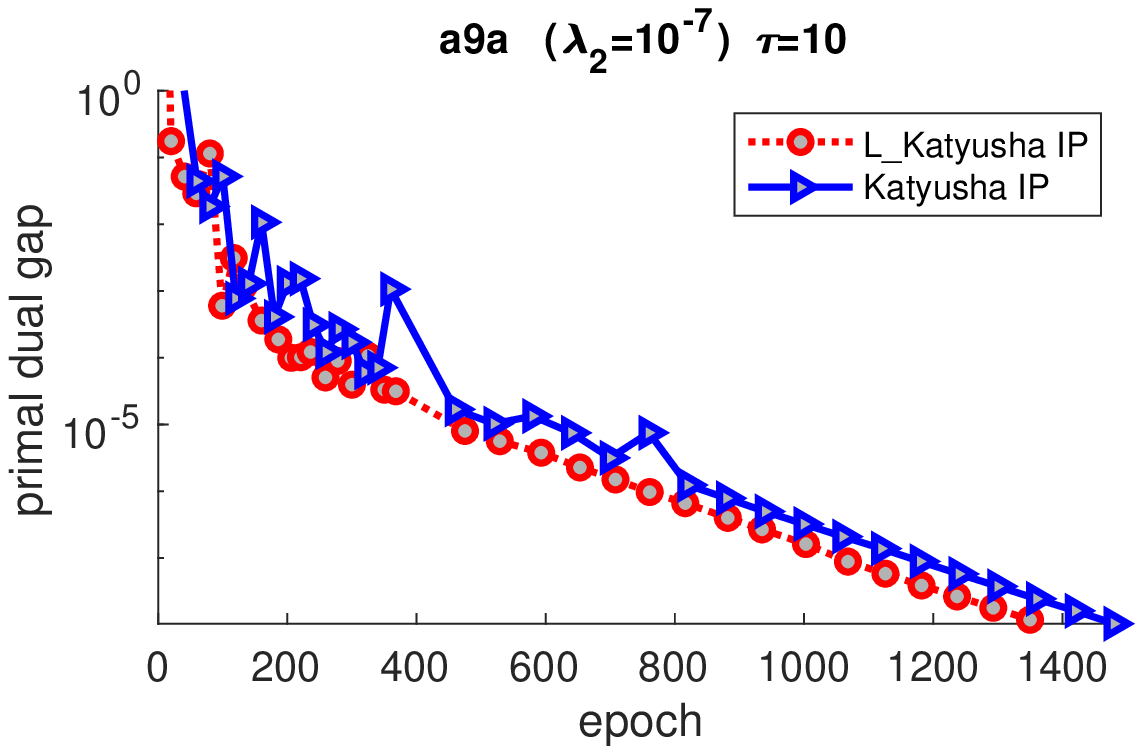}}
    \subfigure[$\tau=50$]{ \label{11d3} \includegraphics[scale=0.38]{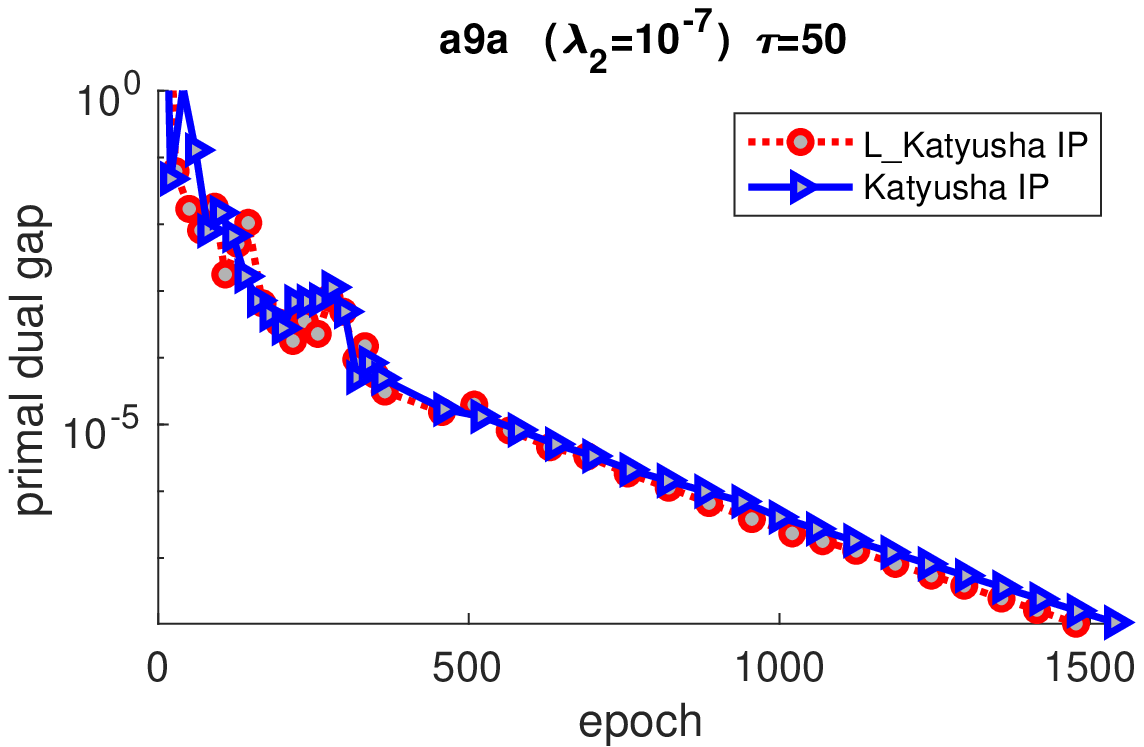}}
    \caption{L-Katyusha V.S. Katyusha, epoch plot, real-sim}\label{fig11}
 \end{figure}
 \vspace{-2cm}
\begin{figure}[!ht]
 \subfigure[$\tau=1$]{\label{12d1}\includegraphics[scale=0.38]{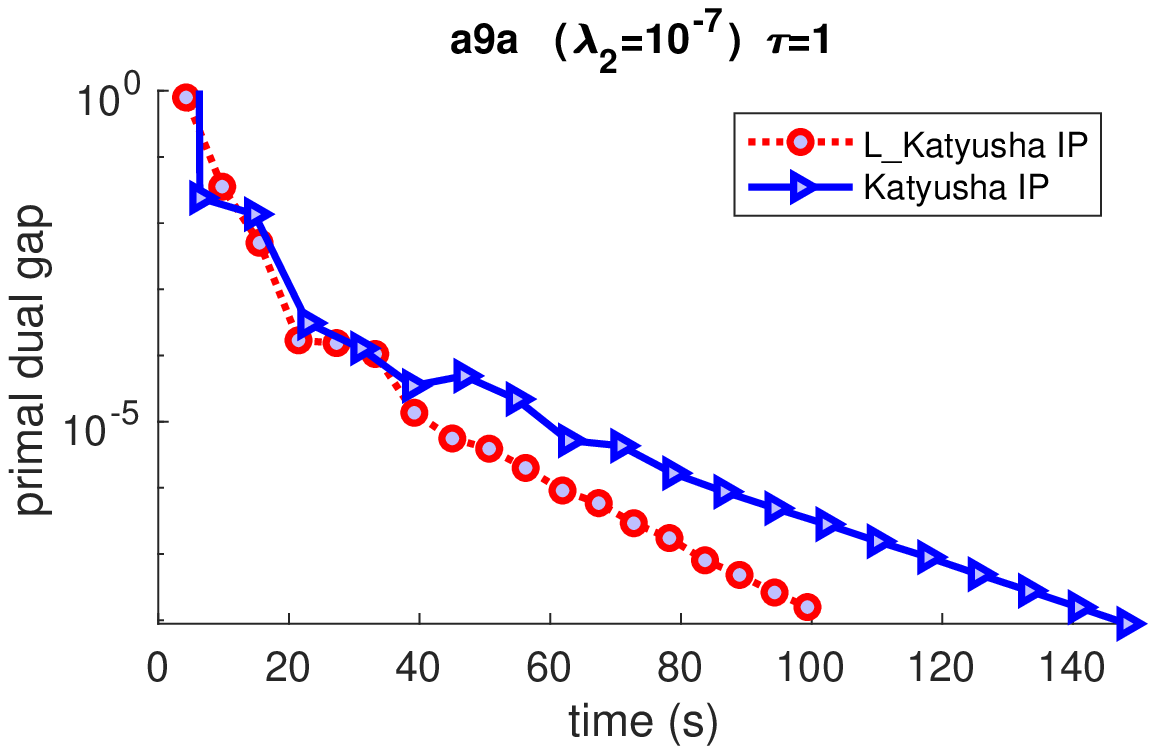}}
  \subfigure[$\tau=10$]{  \label{12d2}  \includegraphics[scale=0.38]{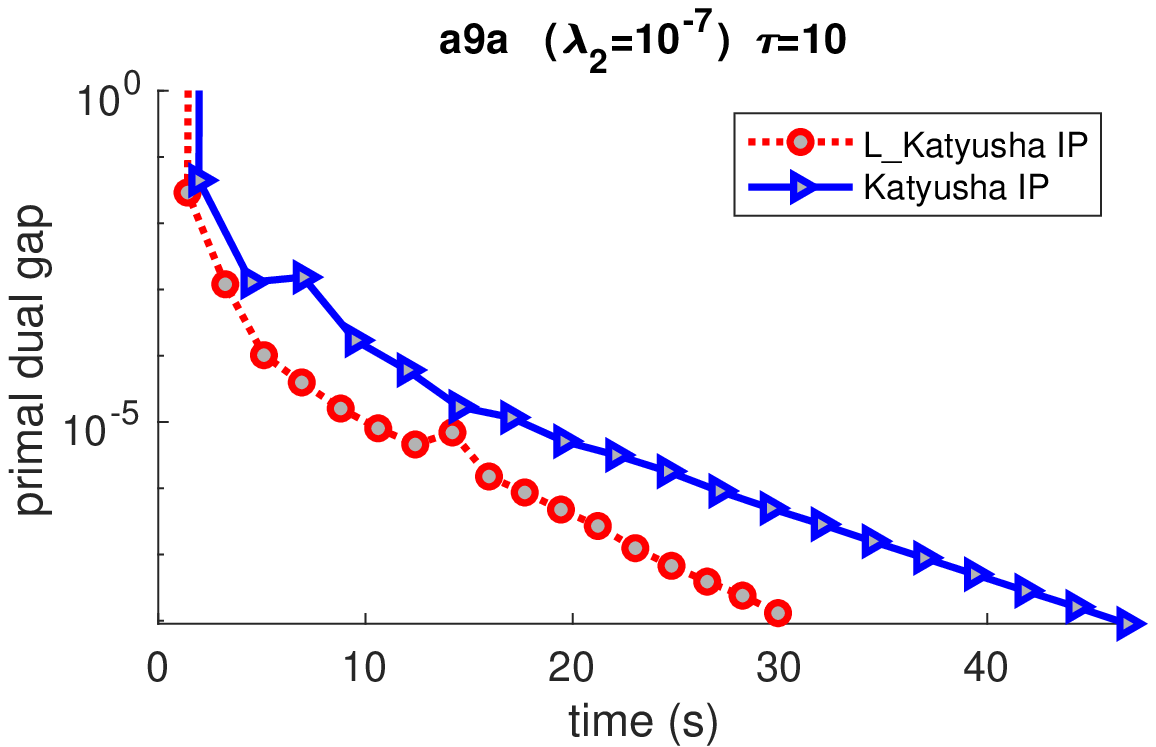}}
    \subfigure[$\tau=50$]{ \label{12d3} \includegraphics[scale=0.38]{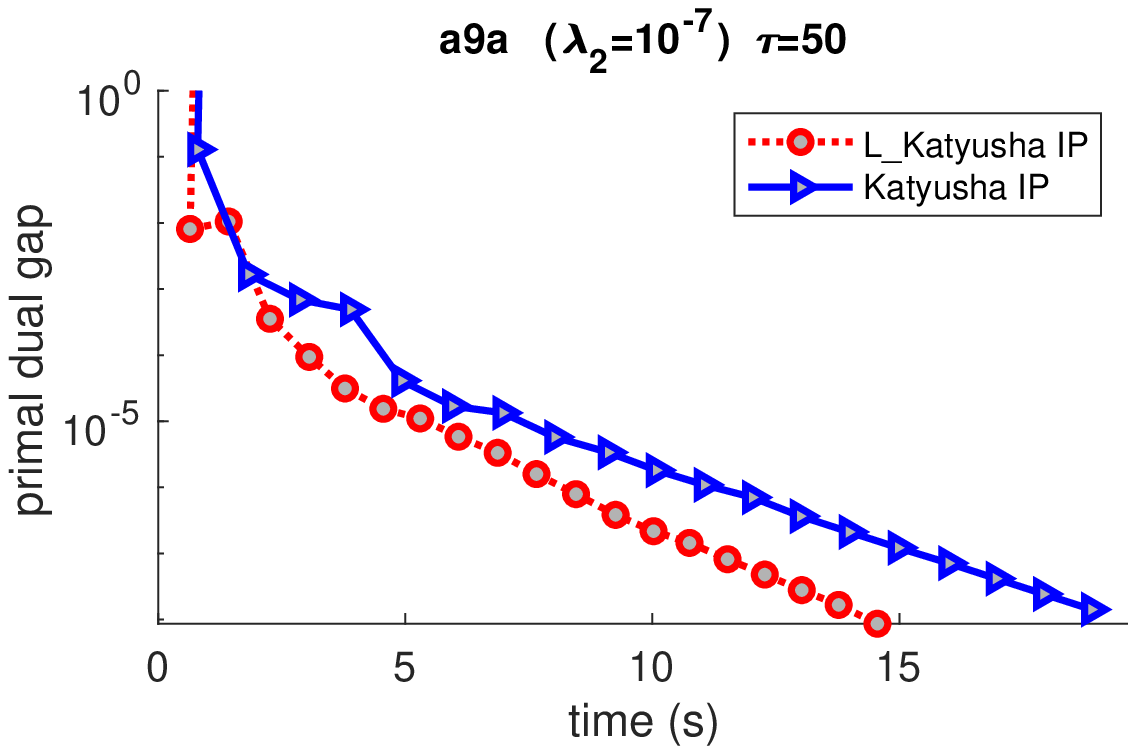}}
    \caption{L-Katyusha V.S. Katyusha, time plot, real-sim}\label{fig12}
 \end{figure}
 
 \begin{figure}[!ht]
 \subfigure[$\tau=1$]{\label{13d1}\includegraphics[scale=0.38]{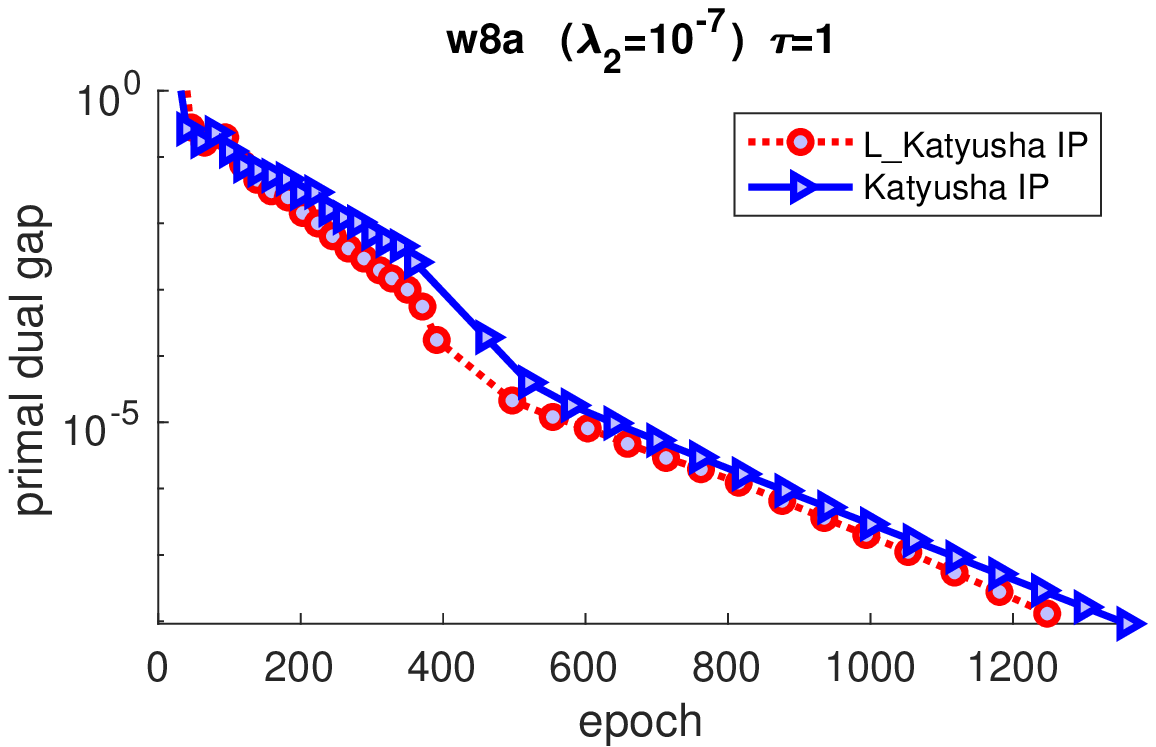}}
  \subfigure[$\tau=10$]{  \label{13d2}  \includegraphics[scale=0.38]{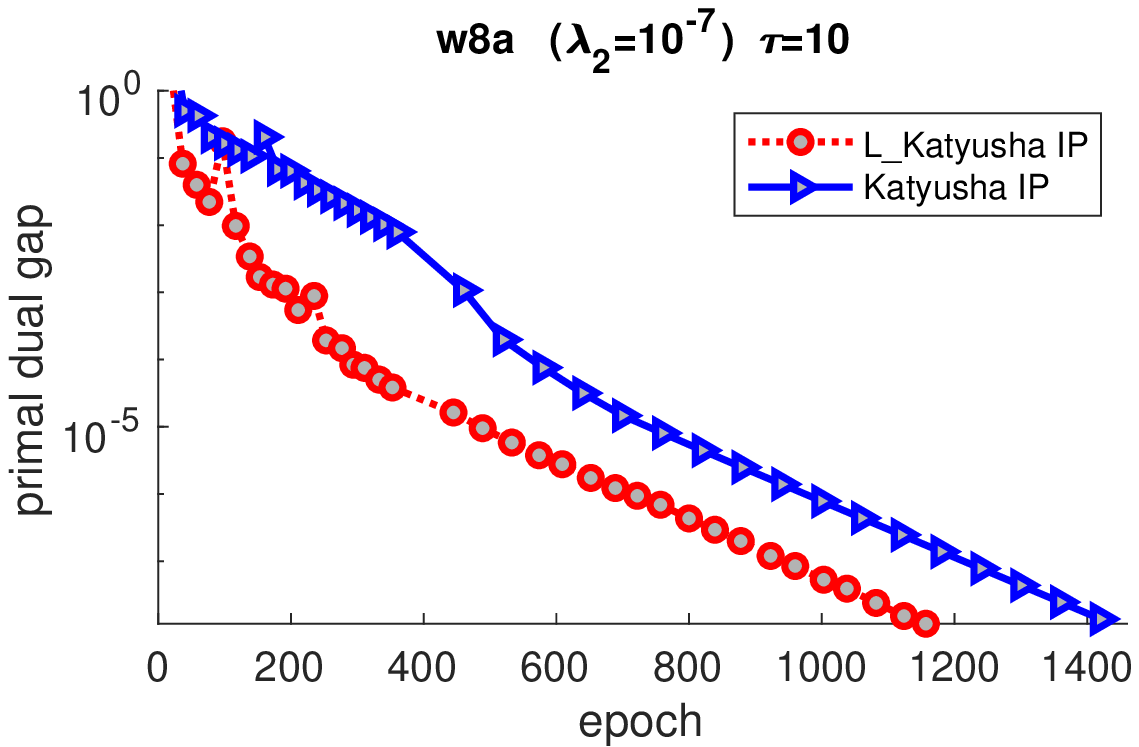}}
    \subfigure[$\tau=50$]{ \label{13d3} \includegraphics[scale=0.38]{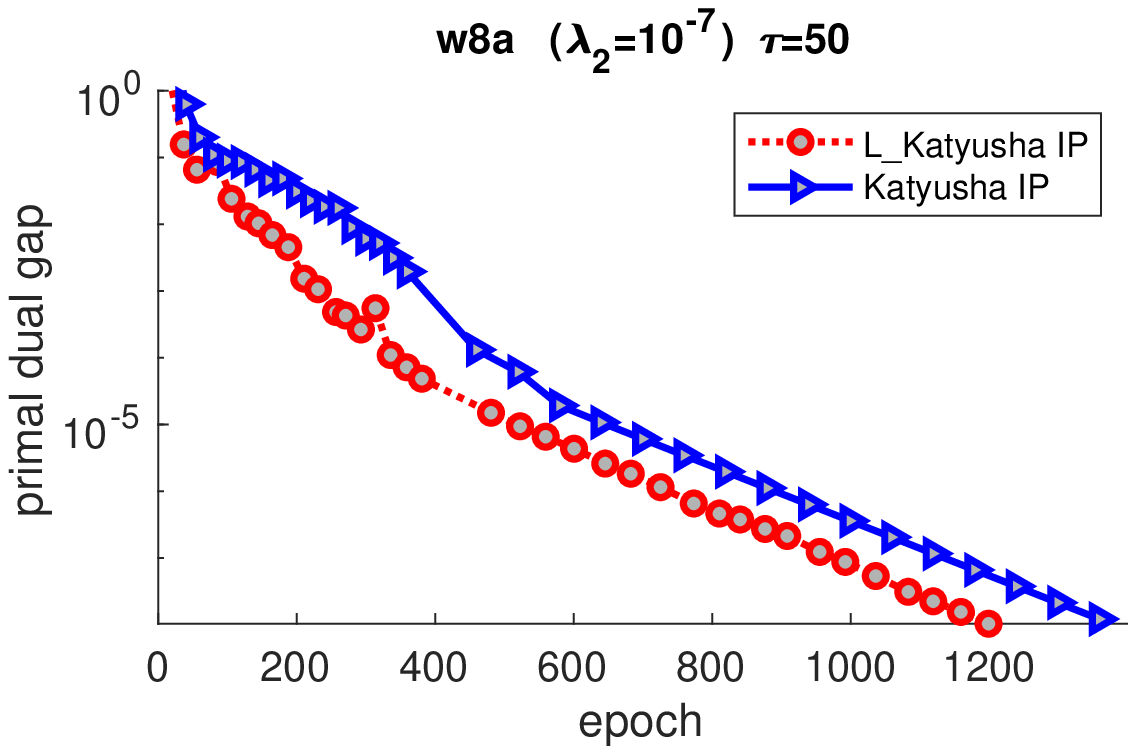}}
    \caption{L-Katyusha V.S. Katyusha, epoch plot, astro\_ph}\label{fig13}
 \end{figure}
\begin{figure}[!ht]
 \subfigure[$\tau=1$]{\label{14d1}\includegraphics[scale=0.38]{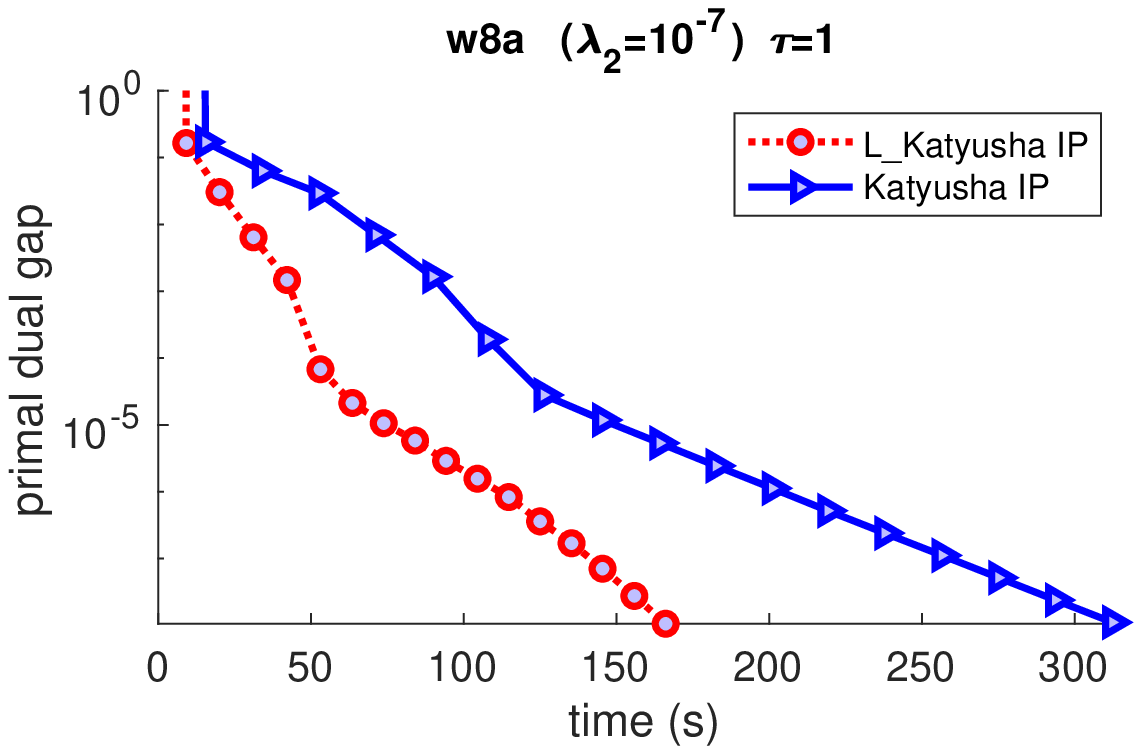}}
  \subfigure[$\tau=10$]{  \label{14d2}  \includegraphics[scale=0.38]{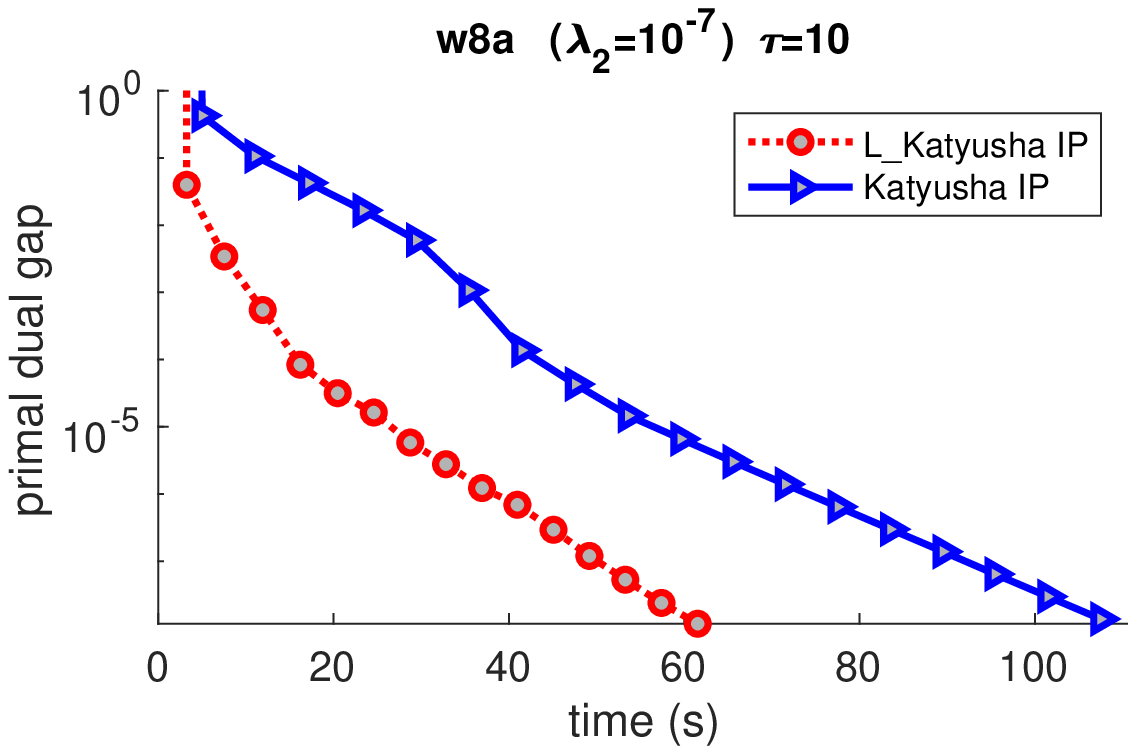}}
    \subfigure[$\tau=50$]{ \label{14d3} \includegraphics[scale=0.38]{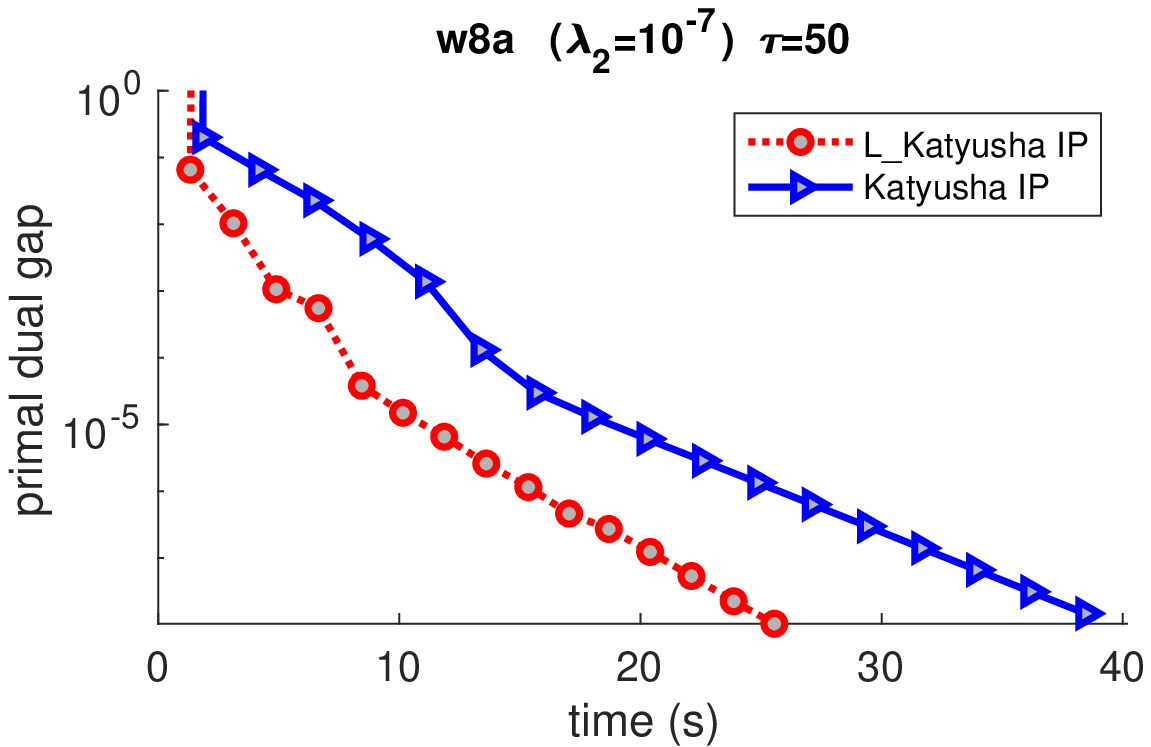}}
    \caption{L-Katyusha V.S. Katyusha, time plot, astro\_ph}\label{fig14}
 \end{figure}


\end{document}